\documentclass[12pt,reqno]{amsart}
\usepackage{amscd}
\usepackage{graphicx}
\usepackage{mathrsfs}
\usepackage{cite}
\usepackage{ifpdf}
\usepackage{fancyhdr}
\usepackage{multirow}
\usepackage{array}
\usepackage{appendix}
\usepackage{longtable}
\usepackage{bm}
\usepackage{cancel}
\usepackage{makecell}

\usepackage{amssymb}
\usepackage{amsthm}
\usepackage{amsmath}
\usepackage{cancel}
\usepackage{tikz}
\usepackage{graphicx}
\usepackage{xcolor}
\usepackage{dsfont}
\usepackage{csquotes}
\usepackage{float}
\usepackage{soul}
\usepackage{extarrows}
\usepackage[shortlabels]{enumitem}
\usepackage[foot]{amsaddr}
\usepackage{doi}
\usepackage{bropd}
\usepackage{setspace}
\usepackage{ragged2e}
\usepackage{bookmark}
\usepackage{bm}
\usepackage{amsaddr}
\usepackage{subfig}
\usepackage{overpic}
\usepackage[a4paper,scale=0.76,twoside=false]{geometry}
\usepackage[skip=0pt plus1pt, indent=22pt]{parskip}
\allowdisplaybreaks
\usepackage{color,verbatim}
\usepackage[intlimits]{esint}
\usepackage[numbers,square,sort&compress]{natbib}
\usepackage{csquotes}
\usepackage{hyperref}
\usepackage[capitalise, nameinlink, noabbrev]{cleveref}
\hypersetup{colorlinks=true,
	linkcolor=blue,
	citecolor=red,
	allcolors=blue
}
\numberwithin{equation}{section}
\newtheorem{thm}{Theorem}[section]
\newtheorem{lem}[thm]{Lemma}
\newtheorem*{thmA-1}{Theorem A-1}
\newtheorem*{thmA-2}{Theorem A-2}
\newtheorem*{thmB}{Theorem B}

\newtheorem{prop}[thm]{Proposition}

\newtheorem{rem}[thm]{Remark}

\newtheorem{assume}[thm]{Assumption}
\newtheorem{define}{Definition}[section]

\newcommand{\hrefemail}[1]{\href{mailto:#1}{#1}}

\usepackage[cal=euler,scr=boondox]{mathalpha}
\title[Singular and regular analysis for two-phase free boundary with gravity]{Singular and regular analysis for the free boundaries of two-phase inviscid fluids in gravity field}
\subjclass[2010]{Primary 35Q35, 76B15; Secondary 35R35, 35B44}
\keywords{Free boundary; Two-phase fluid; Stagnation point; Singularity; Regularity.}
\author[Du]{$^{\dagger,\ddagger,1}$Lili Du}
\email{$^{1}$\hrefemail{ dulili@scu.edu.cn}}
\address{$^{\dagger}$School of Mathematical Sciences, Shenzhen University, 
Shenzhen, 518000, P.~R.~China}

\author[Ji]{$^{\ddagger,2}$Feng Ji}
\email{$^{2}$\hrefemail{jifeng\_math@126.com}}
\address{$^{\ddagger}$, Department of Mathematics, Sichuan University, Chengdu, 610000, P.~R.~China}
\begin{document}
\begin{abstract}
    In this paper, we consider a free boundary problem of two-phase inviscid incompressible fluid in gravity field. The presence of the gravity field induces novel phenomena that there might be some stagnation points on free surface of the two-phase flow, where the velocity field of the fluid vanishes. From the mathematical point of view, the gradient of the stream function degenerates near the stagnation point, leading to singular behaviors on the free surface. The primary objective of this study is to investigate the singularity and regularity of the two-phase free surface, considering their mutual interaction between the two incompressible fluids in two dimensions. More precisely, if the two fluids meet locally at a single point, referred to as the possible two-phase stagnation point, we demonstrate that the singular side of the two-phase free surface exhibits a symmetric Stokes singular profile, while the regular side near this point maintains the $C^{1,\alpha}$ regularity. On the other hand, if the free surfaces of the two fluids stick together and have non-trivial overlapping common boundary at the stagnation point, then the interaction between the two fluids will break the symmetry of the Stokes corner profile, which is attached to the $C^{1,\alpha}$ regular free surface on the other side. As a byproduct of our analysis, it's shown that the velocity field for the two fluids cannot vanish simultaneously on the two-phase free boundary.
    
    Our results generalize the significant works on the Stokes conjecture in [V$\check{a}$rv$\check{a}$ruc$\check{a}$-Weiss, Acta Math., 206, (2011)] for one-phase gravity water wave, and on regular results on the free boundaries in [De Philippis-Spolaor-Velichkov, Invent. Math., 225, (2021)] for two-phase fluids without gravity.
\end{abstract}
\maketitle
\tableofcontents
\section{Introduction and main results}

\subsection{Two-phase Bernoulli-type free boundary problem}

In this paper we will investigate the free surface (in particular singularities) of two-phase incompressible fluids with gravity from the perspective of two-phase Bernoulli type free boundary problems.

Most of the flows encountered in nature are multi-fluid flows, which encompasses flows of non-miscible fluids like water, oil and air. There can be small amplitude waves propagating at the interface between the two fluids. Each fluid obeys its own model and the coupling occurs through the two-phase free interface. In hydrodynamics, the two-phase free boundary problem describes both water waves and the equally physical problem of the equilibrium state of a fluid when pumping in water from one lateral boundary and sucking it out at the other lateral boundary. The theory of displacement of one fluid by another is of great interest and has received much attention in the literature and experiment, for example, phase transition at the sharp interface between two flows as in \cite{D94} for two incompressible flows, the multi-component multi-phase fluid flow in two-dimensional anisotropic heterogeneous porous media as in \cite{MM19}, and the reduced gravity model simulating the stratification in the ocean as a two-layer fluid in \cite{YZZ24}, and so on. It is nature that the branch point may appear in drop configuration in view of the existence of surface capillary forces, and a cavity may appear inside the fluid, leading to some new behavior of free boundaries.

\subsection{Formulation of the two-phase problem}

We consider the geometric profile of the interface between two incompressible inviscid irrotational fluids, namely fluid 1 and fluid 2, of different densities under the gravity field in dimension $2$. The two-phase free boundary problem acted on by gravity writes
\begin{equation} \label{eq1.1}
\begin{cases}
\Delta u = 0 \qquad\qquad\qquad\ \ \;\, \text{in} \quad \{u\neq0\}, \\
|\nabla u^+|^2 - |\nabla u^-|^2 = \Lambda \quad \text{on} \quad \partial\{u>0\}\cap\partial\{u<0\}, \\
|\nabla u^+|^2 = -x_2 \qquad\qquad\, \text{on} \quad \partial\{u>0\}\backslash\partial\{u<0\}, \\
|\nabla u^-|^2 = -x_2-\Lambda \qquad\; \text{on} \quad \partial\{u<0\}\backslash\partial\{u>0\}
\end{cases}
\end{equation}
with the constant $\Lambda$. Here $u$ is the so-called \emph{Stokes stream function} and $u^+:=\max\{u,0\}$, $u^-:=\max\{-u,0\}$. We would like to mention that the term $-x_2$ in the free boundary conditions arises from the gravity force, which leads to the presence of stagnation point. This phenomenon is commonly encountered in hydrodynamics, and we will provide a detailed physical explanation of the issue in Section 1.3. Without loss of generality, we assume $\Lambda\leq0$, namely, $\Lambda=-\lambda^2$ for $\lambda\geq0$. In this setting, $\nabla u^\pm$ may degenerate at certain points on the free boundaries $\Gamma:=\partial\{u>0\}\cup\partial\{u<0\}$, leading to singularities. Furthermore, the nodal set $\{u=0\}$ may have positive Lebesgue measure, and we will specify some definitions as in \cite{PSV21} to help the classification for free boundary points. (Please see Figure \ref{1-0}.)

\begin{figure}[!h]
	\includegraphics[width=75mm]{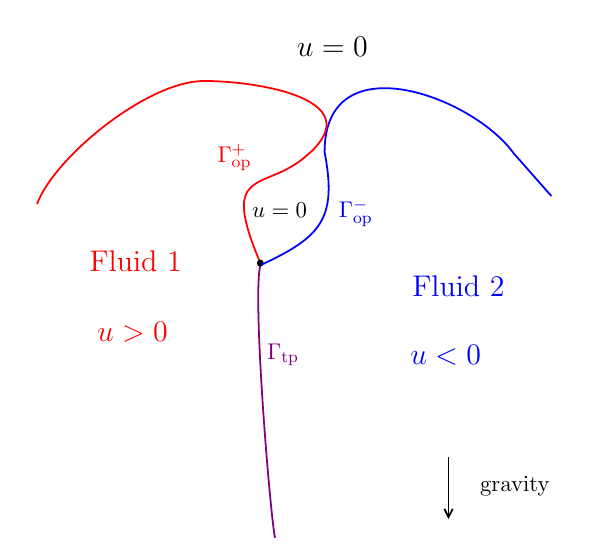}
	\caption{Two-phase free boundary problem}
	\label{1-0}
\end{figure}

\begin{define} \label{deffb}
	We classify the free boundary points for $u$ of the two-phase free boundary problem (\ref{eq1.1}) into two types as follows.
	
	(1) The one-phase free boundary points
	\begin{equation} \label{op+}
	\Gamma_{\rm op}^+:= \partial\{u>0\}\backslash\partial\{u<0\},
	\end{equation}
	\qquad\quad and
	\begin{equation} \label{op-}
	\Gamma_{\rm op}^-:= \partial\{u<0\}\backslash\partial\{u>0\}.
	\end{equation}
	
	(2) The two-phase free boundary points
	\begin{equation} \label{tp}
	\Gamma_{\rm tp}:= \partial\{u>0\}\cap\partial\{u<0\}.
	\end{equation}
\end{define}

\begin{rem}
	Definition \ref{deffb} implies that if $x^0\in\Gamma_{\rm op}^+$ (resp. $x^0\in\Gamma_{\rm op}^-$), then $u(x^0)=0$ and there is $r_0>0$ depending on $x^0$ such that for any $r<r_0$, $u\geq0$ (resp. $u\leq0$) in $B_r(x^0)$. If $x^0\in\Gamma_{\rm tp}$, then $u$ has to change its sign in $B_r(x^0)$ for any $r>0$.
\end{rem}

In particular, the connection of the one-phase free boundaries $\Gamma_{\rm op}^\pm$ and two-phase free boundary $\Gamma_{\rm tp}$ is always interesting point, called \emph{branch point}, and we give the definition as follows. (See Figure \ref{1-0} for example.)

\begin{define} \label{defbp}
	We say a point $x^0\in\Gamma_{\rm tp}$ is a branch point of $u$, provided that the Lebesgue measure $|\{u=0\}\cap B_r(x^0)|>0$ for any $r>0$. And we denote $\Gamma_{\rm bp}$ as the set of branch points.
\end{define}

Our motivation is reminiscent of the previous studies on stagnation points (where $|\nabla u|=0$) for one-phase gravity water waves, which is referred to as \emph{Stokes conjecture} about the shape profile of the wave at the crest of the interface between the air and water. In 1880, G. G. Stokes studied the free surface of an incompressible inviscid fluid under the influence of gravity in two dimensions, traveling in permanent form with a constant speed, and the formation of a $120^\circ$ corner at the crests of the interface has been conjectured. In the significant work \cite{VW11}, V$\check{a}$rv$\check{a}$ruc$\check{a}$ and Weiss skillfully tackled the intricate problem of singularity analysis through the utilization of a monotonicity formula and frequency formula, and unraveled the shape of free boundary at stagnation point. They concluded that the asymptotics at any stagnation point in one-phase water wave without vorticity is given by the "Stokes corner flow" where the surface has a corner of $120^\circ$, which gave a perfect and rigorous answer to the Stokes conjecture. However, the two-phase flow possesses some new and interesting phenomena that $|\nabla u^\pm|$ may not vanish simultaneously, and each phase may interact with the other side, which require us to consider both the one-sided singularity analysis and the overall interaction structure.

\subsection{Physical background}
A class of two-phase incompressible fluids concerns the motion of interface separating two inviscid, incompressible, irrotational fluids from a region of zero density (for example, the air) under the influence of gravity in 2-dimensional domain $D$. It is assumed that the velocity field of the fluids are $U^\pm$, the densities of the fluids are $\rho^\pm$, the pressure of the fluids are $P^\pm$, the two fluids occupy the regions $\Omega^\pm$ respectively, and the gravity field is $-g \bm e_2$ with $\bm e_2=(0,1)$ and $g$ the gravitational acceleration. When surface tension is zero, the governing equations of the inviscid irrotational fluids are described by
\begin{equation*}
\begin{cases}
\nabla\cdot(\rho^\pm U^\pm)=0 \qquad\qquad\qquad\quad\quad\ \ \text{in} \quad \Omega^\pm, \\
\rho^\pm(U^\pm\cdot\nabla)U^\pm + \nabla P^\pm + g \bm e_2 = 0 \quad \text{in} \quad \Omega^\pm, \\
\nabla\times U^\pm = 0 \qquad\qquad\qquad\qquad\qquad \text{in} \quad \Omega^\pm,
\end{cases}
\end{equation*}
where $U^+=(U_1^+(x,y),U_2^+(x,y))$ in fluid 1 and $U^-=(U_1^-(x,y),U_2^-(x,y))$ in fluid 2 for $(x,y)\in\Omega^\pm$. The fluids and the air are separated by the unknown interfaces $\partial\Omega^\pm$, named the free boundaries. See Figure \ref{1-1}.

\begin{figure}[!h]
	\includegraphics[width=75mm]{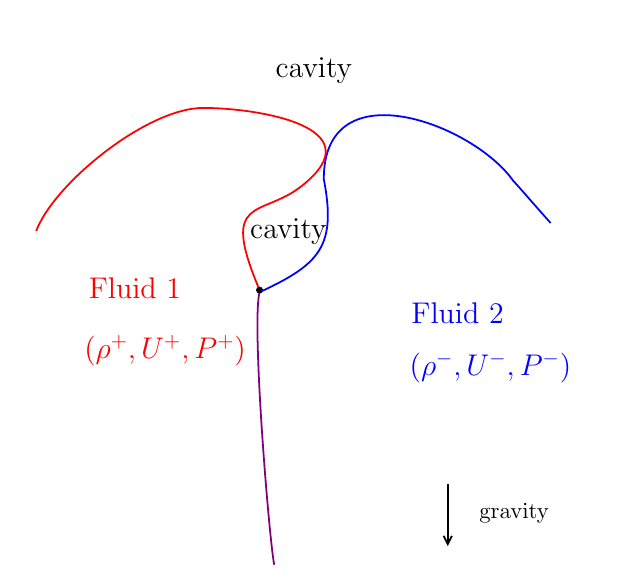}
	\caption{Two-phase fluid model with free boundaries}
	\label{1-1}
\end{figure}

It is straightforward from the incompressbility condition that there are two stream functions $\psi^\pm(x,y)$ in the fluid field such that
$$U_1^\pm=\frac{\partial_y \psi^\pm}{\sqrt{\rho^\pm}} \quad \text{and} \quad U_2^\pm=-\frac{\partial_x \psi^\pm}{\sqrt{\rho^\pm}},$$
where $\partial_x$ and $\partial_y$ denotes the partial derivatives with respect to $x$ and $y$, and satisfying
$$\Delta \psi^\pm =0 \quad \text{in} \quad \Omega^\pm$$
respectively in fluid 1 and fluid 2. Meanwhile, the kinematic boundary condition that the same particles always form the free surface is equivalent to the fact that $\psi$ is a constant on the free boundary. Without loss of generality, we impose $\psi=0$ on the free surface, and we can define
\begin{equation*}
\psi(x,y)=
\begin{cases}
\psi^+ \quad \text{in fluid field 1, $\Omega^+$,} \\
\psi^- \quad \text{in fluid field 2, $\Omega^-$,} \\
0, \quad\ \text{otherwise.}
\end{cases}
\end{equation*}
Hence $\{\psi>0\}$ and $\{\psi<0\}$ are the two fluid fields respectively, $\{\psi=0\}$ is the air region, and $\partial\{\psi>0\}\cup\partial\{\psi<0\}$ denotes the free surface. Furthermore, we can define the two-phase free boundary $\Gamma_{\rm tp}$ to be the interface between the two fluids, and the one-phase free boundary $\Gamma_{\rm op}^+$ (or $\Gamma_{\rm op}^-$) to be the interface between the cavity and fluid 1 (or fluid 2). See Figure \ref{1-0} in Section 1.1 for the classification of free boundaries.

On account of Bernoulli's law, we obtain that
$$\frac{P^\pm}{\rho^\pm} + \frac12|U^\pm|^2 +\frac{gy}{\rho^\pm} \text{ are constants.}$$
We assume the constants to be $\mathcal{B}^\pm$ respectively in $\Omega^\pm$. Plugging the expression of $U^\pm$ into it, we have
\begin{equation} \label{ber}
\frac{|\nabla \psi^+|^2}{2} + P^+ + gy =\rho^+\mathcal{B}^+ \quad \text{and} \quad \frac{|\nabla \psi^-|^2}{2} + P^- + gy =\rho^-\mathcal{B}^-
\end{equation}
in $\Omega^+$ and $\Omega^-$ respectively. Moreover, the pressure is assumed to be continuous across the two-phase free boundary $\Gamma_{\rm tp}$ (fluid-fluid interface), and they are equal to a constant pressure denoted as $P_0$ on the one-phase free boundaries $\Gamma_{\rm op}^\pm$ (fluid-air interface). Hence, we have
$$P^+=P^- \quad \text{on} \quad \Gamma_{\rm tp} \quad \text{and} \quad  P^\pm=P_0\quad \text{on} \quad \Gamma_{\rm op}^\pm.$$
This together with (\ref{ber}) implies that
$$|\nabla \psi^\pm|^2 = 2(\rho^\pm\mathcal{B}^\pm - P_0 - gy) \quad \text{on} \quad \Gamma_{\rm op}^\pm$$
and
$$|\nabla \psi^+|^2 - |\nabla \psi^-|^2 = 2(\rho^+\mathcal{B}^+ - \rho^-\mathcal{B}^-) \quad \text{on} \quad \Gamma_{\rm tp}.$$
Define the parameters $\lambda^+$ and $\lambda^-$ as
$$\lambda^+=2(\rho^+\mathcal{B}^+ - P_0) \quad \text{and} \quad \lambda^-=2(\rho^-\mathcal{B}^- - P_0),$$
and we have
$$\Lambda:=\lambda^+ - \lambda^- = 2(\rho^+\mathcal{B}^+ - \rho^-\mathcal{B}^-).$$
In fact, $\frac12\lambda^\pm$ represents the kinetic energy of the fluids per unit volumn on their one-phase free boundaries, and $\frac12(\lambda^+-\lambda^-)$ means the jump of the kinetic energy per unit volumn across the two-phase free boundary. Consequently, we obtain the following transition conditions on the free boundaries,
\begin{equation*}
\begin{cases}
|\nabla \psi^+|^2 = \lambda^+ - 2gy \qquad\qquad\qquad\quad \text{on} \quad \Gamma_{\rm op}^+, \\
|\nabla \psi^-|^2 = \lambda^- - 2gy \qquad\qquad\qquad\quad \text{on} \quad \Gamma_{\rm op}^-, \\
|\nabla \psi^+|^2 - |\nabla \psi^-|^2 = \lambda^+ - \lambda^-=\Lambda \quad \text{on} \quad \Gamma_{\rm tp}.
\end{cases}
\end{equation*}
It is natural to assume throughout the rest of the paper that $$\psi^+\equiv0 \text{ in } \{y\geq\lambda^+/(2g)\} \quad\text{and}\quad \psi^-\equiv0 \text{ in } \{y\geq\lambda^-/(2g)\}.$$
Suppose without loss of generality that $\Lambda\leq0$ and set $\lambda^2:=-\Lambda=\lambda^- - \lambda^+\geq0$. For the sake of simplicity, transform the coordinate $(x_1,x_2)=(2gx,2gy-\lambda^+)$ and denote $u(x_1,x_2)=2g\psi(x,y)=2g\psi\left( \frac{x_1}{2g},\frac{x_2+\lambda^+}{2g} \right)$, and we have
\begin{equation} \label{eq1.2}
\begin{cases}
|\nabla u^+|^2 = -x_2 \qquad\quad\quad\quad\; \text{on} \quad \Gamma_{\rm op}^+, \\
|\nabla u^-|^2 = -x_2 + \lambda^2 \quad\quad\ \ \ \text{on} \quad \Gamma_{\rm op}^-, \\
|\nabla u^+|^2 - |\nabla u^-|^2 = -\lambda^2 \quad \text{on} \quad \Gamma_{\rm tp}.
\end{cases}
\end{equation}
Thus we obtain the equation (\ref{eq1.1}) with $\Lambda=-\lambda^2\leq0$. See Figure \ref{impinge} for an intuitive physical illustration, which describes two flows impinging in a nozzle and falling back by gravity.

\begin{figure}[!h]
	\includegraphics[width=90mm]{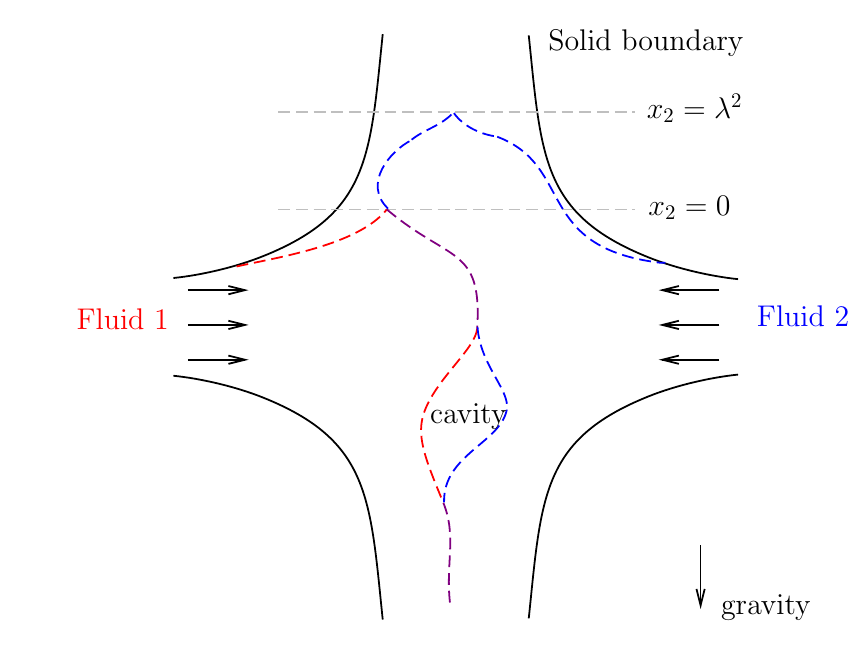}
	\caption{The impinging jet flows with gravity through a nozzle.}
	\label{impinge}
\end{figure}

\subsection{Variational approach}

The solutions to water wave problems are relative to the critical points of some energy functional, and we could discuss the variation of the energy. Alt, Caffarelli and Friedman started the pioneer research on local minimizers (where the domain variation of the energy must vanish) in a domain $D\subset\mathbb{R}^n$ for the so-called \emph{ACF-functional},
\begin{equation} \label{Jacf}
\mathcal{J}_{acf}(u;D)=\int_D \left( |\nabla u|^2 + \lambda_1^2(x) \chi_{\{u>0\}} + \lambda_2^2(x) \chi_{\{u<0\}} +\lambda_0^2(x) \chi_{\{u=0\}} \right) dx
\end{equation}
with some given non-negative functions $\lambda_i(x)\in C^{0,\alpha}(D), i=0,1,2$ for
$$u\in\mathcal{K}:=\left\{ u\in W^{1,2}(D) | \ u=g \ \text{on}\ \partial D  \right\},$$
where $g\in W^{1,2}(\partial D)$ is the fixed boundary function and $\chi_S$ denotes the characteristic function of the set $S$, $dx=dx_1 dx_2$. The regularity of local minimizers and free boundaries with two-phases was first addressed in their epoch-marking works \cite{ACF84}\cite{C87}\cite{C88} and \cite{C89}, and the series of works also established the significant criteria "Flatness implies $C^{1,\alpha}$" for free boundaries. One of the key points in the two-phase problem is the nontrivial nodal set $\{u=0\}$, which affects the behavior of the free boundaries. This cavity introduces a new element in the analysis of the free boundary, switching from one-phase to two-phase at the so-called branch points, at which the zero level set looks like a cusp. Alt, Caffarelli and Friedman avoided the discussion of cavity and branch point by setting the constant parameters $\lambda_i(x):=\lambda_i (i=0,1,2)$ and $\lambda_0=\min\{\lambda_1,\lambda_2\}$ (c.f. \cite{ACF84}, Chapter 6) and focused on the two-phase free boundary points to derive the $C^{1,\alpha}$ regularity of the free boundary when the gradient of $u^\pm$ do not vanish. Moreover, the rule of "flatness implies $C^{1,\alpha}$" is also obtained in \cite{SFS14} for two-phase viscosity solutions (which are the closest sense of solutions to minimizers). Roughly speaking, "flatness implies $C^{1,\alpha}$" means that if the free boundary is in the neighborhood of a certain point sufficiently close to a plane, then it has to be a $C^{1,\alpha}$ surface in a certain smaller neighborhood of that point. Recently, De Philippis, Spolaor and Velichkov carried on a more interesting investigation about the regularity around the branch point when the constant parameters $0\leq\lambda_0<\min(\lambda_1,\lambda_2)$ and $|\{u=0\}|\neq0$ in the celebrated work \cite{PSV21}, which filled the gap in the two-phase research with cavity, but required that both $|\nabla u^\pm|$ should be positive. This regularity result on the free boundary was generalized to the 3-dimensional axisymmetric two-phase flow in \cite{DJ24}. Nevertheless, the singularity of the free boundary point where at least one of $|\nabla u^\pm|$ vanishes remains unknown, which forms a part of our motivation of this paper.

A perplexing issue, untouched in the literature, concerns the singular free boundary points in two-phase problem. A stagnation point is one at which the velocity field is zero, and such free boundary point is always not "flat". The brilliant works \cite{VW11} and \cite{VW12} are well established paradigms for Stokes conjecture, initially come up in \cite{S80}, Appendix B. Inspired by the theory for one-phase gravity water wave, we investigate the stagnation points in two-phase incompressible inviscid fluids in gravity field. In particular, our results also imply that in general, the velocity field may not vanish simultaneously at a point for the two-phase interface. More precisely, the result of "Stokes corner flow" is expected to hold for the fluid with vanishing velocity at the free boundary point, while the other phase may have smooth free surface near this point, and there might be interaction between the two free surfaces. We also discuss a special case that both fluids have vanishing velocity field simultaneously at the stagnation point. Acted on by gravity, the two fluids are naturally restricted in a half-plane, which cannot be filled with two "Stokes corner flow" without overlapping. This result implies that if both velocities of the two fluids vanish at some free boundary point $x^0$, then the point $x^0$ must be a one-phase free boundary point.

We focus on the two-phase stagnation point, namely, the possible stagnation point on two-phase free boundary $\Gamma_{\rm tp}$. The main objective of the present paper is to work out the singularity and the regularity of the free surface near the two-phase stagnation point. Considering the effect of the gravity, we may set the functional
$$\mathcal{J}_{tp}(u;D)=\int_D \left( |\nabla u|^2 + (-x_2) \chi_{\{u>0\}} + (-x_2+\lambda^2) \chi_{\{u<0\}} \right) dx,$$
which is coincide with (\ref{Jacf}) when $\lambda_1(x)=-x_2$, $\lambda_2(x)=-x_2+\lambda^2$ and $\lambda_0(x)=0$ with $\lambda\geq0$. The reason for this setting is due to the transition conditions (\ref{eq1.2}) on the free boundaries. The free boundary problem in this paper writes

\begin{equation} \label{equation}
\begin{cases}
\Delta u = 0 \qquad\qquad\qquad\qquad \text{in} \quad D\cap\{u\neq0\}, \\
|\nabla u^+|^2 = -x_2 \qquad\qquad\quad\, \text{on} \quad D\cap\Gamma_{\rm op}^+, \\
|\nabla u^-|^2 = -x_2+\lambda^2 \qquad\quad \text{on} \quad D\cap\Gamma_{\rm op}^-, \\
|\nabla u^+|^2 - |\nabla u^-|^2 = -\lambda^2 \quad \text{on} \quad D\cap\Gamma_{\rm tp}
\end{cases}
\end{equation}
with $\lambda\geq0$. It should be noted that the minimizer $u$ of the energy functional $\mathcal{J}_{acf}$ not only satisfies the transition conditions in (\ref{equation}), but also satisfies the following additional conditions on the free boundaries,
\begin{equation} \label{equation2}
|\nabla u^+|^2\geq -x_2 \quad \text{on} \quad \partial\{u>0\}\cap D, \quad |\nabla u^-|^2\geq -x_2+\lambda^2 \quad \text{on} \quad \partial\{u<0\}\cap D.
\end{equation}
These additional restriction conditions also hold in our case, which will be verified in Appendix A.

\subsection{Classification and analysis on the stagnation point}
As we mentioned before, once the gravity effect is taken into account, physical intuition suggests that the singularity occurs near the possible stagnation points. The main purpose of this paper is to clarify the singular profile of the interfaces at the stagnation point on the free boundary. More precisely, we are interested in the stagnation points on $\Gamma_{\rm tp}$, since the one-phase stagnation points have already been thoroughly studied by V$\check{a}$rv$\check{a}$ruc$\check{a}$ and Weiss in \cite{VW11}. Notice that due to the transition condition $|\nabla u^+|^2 - |\nabla u^-|^2 = -\lambda^2\leq0$, $|\nabla u^\pm|$ may either degenerate simultaneously or not at one point on $\Gamma_{\rm tp}$, which gives a classification of \emph{"partial-degenerate case"} and \emph{"complete-degenerate case"} determined by the value of $\lambda$. In the partial-degenerate case where $\lambda>0$ and $|\nabla u^\pm|$ cannot vanish simultaneously, the stagnation point must be the branch point, defined as in Definition \ref{defbp}. In fact, the $C^{1,\alpha}$ regularity of the "pure" two-phase free boundary $\Gamma_{\rm tp}\backslash\Gamma_{\rm bp}$ when $|\nabla u^-(x^0)|\neq0$ was established in the previous seminal works by Alt, Caffarelli and Friedman in 1980s, and later by De Silva, Ferrari and Salsa through a new approach in \cite{SFS14} in 2014. However, in the complete-degenerate case where $\lambda=0$ and $|\nabla u^+|=|\nabla u^-|=0$ at the stagnation point, the aforementioned results do not apply for the "pure" two-phase free boundary points. Hence, in this paper, we restrict our investigation to the stagnation point $x^0\in\Gamma_{\rm bp}$ when $\lambda>0$, and the stagnation point $x^0\in\Gamma_{\rm tp}$ when $\lambda=0$.

We will discuss the two subcases of partial-degenerate case $\lambda>0$ (Case A), and complete-degenerate case $\lambda=0$ (Case B) in Section 2 and Section 3 respectively, and here we give a brief analysis on each case.

\emph{Case A.} (Partial-degenerate case, $\lambda>0$) Recalling the additional free boundary conditions (\ref{equation2}) that
$$|\nabla u^+|^2\geq-x_2 \quad \text{on} \quad \partial\{u>0\} , \qquad |\nabla u^-|^2\geq-x_2+\lambda^2 \quad \text{on} \quad \partial\{u<0\},$$
it is straightforward to deduce that $|\nabla u^+|$ will vanish in $\{x_2\geq0\}$ at a height lower than $|\nabla u^-|$ when $\lambda>0$, and we concentrate only on the stagnation point $x^0=(x_1^0,0)$ such that $|\nabla u^+(x^0)|=0$. The assumption that $x^0$ lies on $\{x_2=0\}$ arises from the fact that $\Gamma_{\rm op}^+$ vanishes in $\{x_2>0\}$ as $|\nabla u^+|^2=-x_2$ on $\Gamma_{\rm op}^+$. This implies that the free boundary $\partial\{u>0\}=\Gamma_{\rm tp}$ in $\{x_2>0\}$ has no branch point. This situation aligns with the case discussed in \cite{ACF84} and \cite{SFS14} when $\lambda>0$, since the transition condition
$$|\nabla u^+|^2 - |\nabla u^-|^2 = -\lambda^2 \quad \text{on} \quad \Gamma_{\rm tp}$$
as noted in (\ref{equation}) ensures that $|\nabla u^-|$ is non-degenerate.

\emph{Case B.} (Complete-degenerate case, $\lambda=0$) In this case the stagnation point $x^0$ also lies in $\{x_2\geq0\}$ as explained in Case A. Furthermore, we consider the stagnation point $x^0=(x_1^0,0)$, since it is reasonable to assume that $u\equiv0$ in $\{x_2\geq0\}$ in this case, which corresponds to the physical situation that the upper half plane is occupied by the air.

Summarily, it is valid to consider the stagnation point $x^0$ to the positive phase $u^+$ on $\{x_2=0\}$, and the two classifications of $\lambda>0$ and $\lambda=0$ may lead to different asymptotic behaviors, which exhibit new features in two-phase free boundary problem compared to the one-phase water wave. In Case A, we conclude that there is a slight interaction between the degenerate positive phase $u^+$ and the non-degenerate negative phase $u^-$. This interaction can disrupt the symmetry of the possible Stokes corner if the two free boundaries $\partial\{u>0\}$ and $\partial\{u<0\}$ are closely attached. Additionally, it may restrict the deflection angle of the non-degenerate smooth curve within a certain range. In Case B, we observe that the two fluids push against each other when they stick together, preventing them from attaining a common maximal height at one point simultaneously. See Table \ref{table1} for the brief classifications and results.

\begin{table}[h]
	\caption{Classification and Conclusion}\label{table1}
	\centering
	\begin{tabular}{|m{3em}<{\centering}|m{6em}<{\centering}|m{6em}<{\centering}|m{6em}<{\centering}|m{6em}<{\centering}|m{6em}<{\centering}|}
		\hline
		& \multicolumn{2}{c|}{One-phase case: $x^0\in\Gamma_{\rm op}$} & \multicolumn{3}{c|}{Two-phase case: $x^0\in\Gamma_{\rm tp}$} \\
		\hline
		& Type 1 & Type 2 & Type 3 & Type 4 & Type 5 \\
		\hline
		Classifi-cation & Non-degenerate case: \small{$|\nabla u(x^0)|\neq0$} & Degenerate case: \small{$|\nabla u(x^0)|=0$} & Non-degenerate case: \small{$|\nabla u^\pm(x^0)|\neq0$} & Partial-degenerate case: \small{$|\nabla u^+(x^0)|=0$ $|\nabla u^-(x^0)|\neq0$} & Complete-degenerate case: \small{$|\nabla u^\pm(x^0)|=0$} \\
		\hline
		Free boundary near $x^0$ & $C^{1,\alpha}$-graph & Stokes corner & $\partial\{u>0\}\cup$ $\partial\{u<0\}$: respectively $C^{1,\alpha}$-graphs & $\partial\{u>0\}$: Stokes(-type) corner; $\partial\{u<0\}: C^{1,\alpha}$-graph & Impossible; each stagnation point $x^0$ must be a one-phase point \\
		\hline
		Works & Seminal works by Alt-Caffarelli, 1980s. & V$\check{a}$rv$\check{a}$ruc$\check{a}$-Weiss, \emph{Acta Math.}, 2011. & De Philippis-Spolaor-Velichkov, \emph{Invent. Math.}, 2021. & Theorem A in this paper & Theorem B in this paper  \\
		\hline
	\end{tabular}
\end{table}

\subsection{Main results and the organization of this paper}

As explained in Section 1.5, we suppose that the stagnation point of the fluid 1 is on $x_1$-axis, denoted as $x^0=(x_1^0,0)$. We have given the complete classification and the corresponding conclusion both on regular profile and singular profile of the free boundaries in the gravity field in Table \ref{table1}, presenting an overview for the behavior of all free boundary points combined with previous seminal works studying regular free boundary points and one-phase stagnation points. See also Figure \ref{F0} for an illustration.

\begin{figure}[!h]
	\includegraphics[width=175mm]{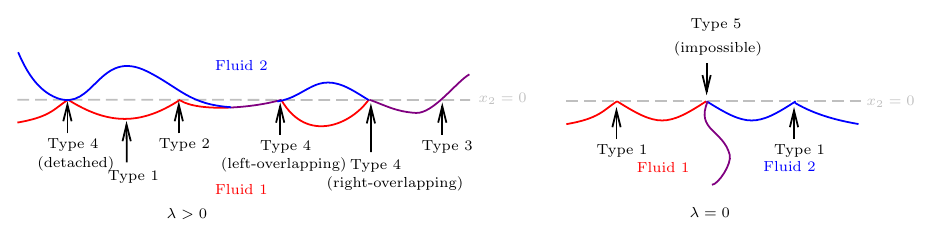}
	\caption{Classification of free boundary points}
	\label{F0}
\end{figure}

\emph{One-phase case.} $x^0\in\Gamma_{\rm op}^+$. In this case $x^0$ can be locally seen as a free boundary point for a one-phase free boundary problem, which is related closely to the Stokes conjecture solved in the celebrated work \cite{VW11}. It is remarkable that the discussion on stagnation points for $u^-$ on $\Gamma_{\rm op}^-\cap\{x_2=\lambda\}$ is quite similar, which is not under consideration here.

\emph{Two-phase case.} We will focus on the local version of the free boundary problem (\ref{eq1.1}), namely, we study the property of the free boundary in a small neighborhood of the possible stagnation point. It is remarkable that the one-phase free boundary condition on $\Gamma_{\rm op}^+$ implies that $\Gamma_{\rm op}^+$ must be contained in the lower half-plane $\{x_2\leq0\}$, and thus we suppose that the fluid of $u^+$ attains its maximal height locally at the stagnation $x^0=(x_1^0,0)$. Namely, we assume that in a small neighborhood $B_R(x^0)$ of $x^0$, $\{u(x)>0\}\cap B_R(x^0)\subset\{x_2<0\}$, which means that the fluid 1 near the stagnation point $x^0$ lies below the height $x_2=0$. Hence we suppose locally that in a small neighborhood of $x^0$,
$$u\leq0 \quad \text{in} \quad B_R(x^0)\cap\{x_2\geq0\} \quad \text{if} \quad \lambda>0,$$
and
$$u\equiv0 \quad \text{in} \quad B_R(x^0)\cap\{x_2\geq0\} \quad \text{if} \quad \lambda=0$$
for some $R>0$. Here we will sketch the classification of $x^0$ here before the statement of our main results.

\emph{Case A. (Partial-degenerate case, $\lambda>0$)} $x^0\in\Gamma_{\rm tp}$ and due to the transition condition in (\ref{equation}), only one of the two fluids has degenerate velocity at $x^0$, namely,
$$ |\nabla u^+(x^0)|^2=0 \quad \text{and} \quad |\nabla u^-(x^0)|^2=\lambda>0 \quad \text{with} \quad \lambda>0.$$
Here $u^-$ is non-degenerate near $x^0=(x_1^0,0)$. It is remarkable that the "pure" two-phase free boundary $\Gamma_{\rm tp}\backslash\Gamma_{\rm bp}$ is locally equal to $\partial\{u<0\}$ (and is also equal to $\partial\{u>0\}$) due to the fact that $\Gamma_{\rm op}^+=\varnothing$ in $\{x_2>0\}$, whose regularity has been obtained by \cite{ACF84} and \cite{SFS14} when $|\nabla u^-|$ is strictly positive. Hence we assume $x^0\in\Gamma_{\rm bp}$, and it can be divided into two further subcases.

\emph{Subcase A-1. (Detached case)} In this case, the free boundaries of $u^+$ and $u^-$ do not intersect in any small neighborhood $B_r(x^0)$ except at $x^0$, namely,
$$\left(\partial\{u>0\}\cap\partial\{u<0\}\cap B_r(x^0)\right)\backslash\{x^0\}=\varnothing$$
for any small enough $r>0$. Notice that we have
\begin{equation*}
\begin{cases}
|\nabla u^+(x)|^2 = -x_2 \qquad\quad\, \text{on} \quad \left( \partial\{u>0\}\cap B_r(x^0) \right) \backslash \{x^0\}, \\
|\nabla u^-(x)|^2 = -x_2+\lambda^2 \quad \text{on} \quad \left( \partial\{u<0\}\cap B_r(x^0) \right) \backslash \{x^0\}.
\end{cases}
\end{equation*}

\emph{Subcase A-2. (Overlapping case)} In this subcase, $u^+$ and $u^-$ share a common free boundary near $x^0$, namely,
$$\left(\partial\{u>0\}\cap\partial\{u<0\}\cap B_r(x^0)\right)\backslash\{x^0\}\neq\varnothing$$
for small enough $r$. It should be noted that in one-phase case, the free boundary condition $|\nabla u|^2=-x_2$ on $\partial\{u>0\}$ helps us to solve an ODE for the scaled solution of $u$, while in the two-phase overlapping situation here, the transition condition
$$|\nabla u^+(x)|^2 - |\nabla u^-(x)|^2 = -\lambda^2 \quad \text{on} \quad \partial\{u>0\}\cap\partial\{u<0\}\cap B_r(x^0)$$
brings some technical difficulties in our analysis since each phase scales differently when taking on a blow-up process.

\emph{Case B. (Complete-degenerate case, $\lambda=0$)} $x^0\in\Gamma_{\rm tp}$ and both velocities of the two fluids vanish at $x^0$, namely,
$$ |\nabla u^+(x^0)|=|\nabla u^-(x^0)|=\lambda=0.$$
We will prove in Section 3 that this case cannot happen for the two-phase free boundary point $x^0\in\Gamma_{\rm tp}$. In other words, in the case $\lambda=0$ the two-phase free boundary $\Gamma_{\rm tp}$ is showed regular, and the stagnation point can only be on the one-phase free boundaries $\Gamma_{\rm op}^\pm$.

\begin{rem}
	We would like to emphasize that without gravity effects there would be no singularity on the free boundary, since the gradient of $u^\pm$ will not vanish, and the free boundary is proved to be $C^{1,\alpha}$ in \cite{PSV21}. The regular profile in \cite{PSV21} without gravity can be seen in Figure \ref{F2}, and an overview of the possible stagnation points and their singular profiles in our situation with gravity is depicted in Figure \ref{F1}.
	\begin{figure}[!h]
		\includegraphics[width=110mm]{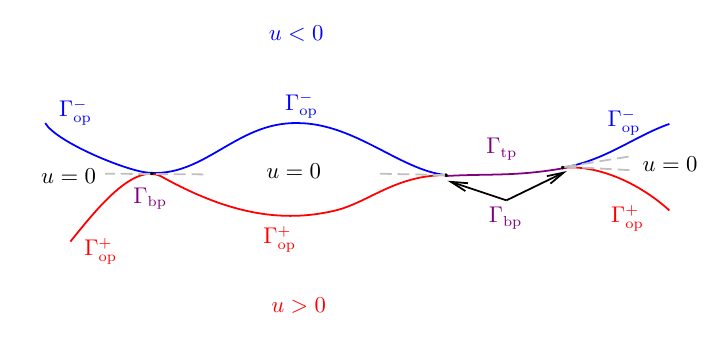}
		\caption{Regular profile of two-phase fluid without gravity.}
		\label{F2}
	\end{figure}
	\begin{figure}[!h]
		\includegraphics[width=155mm]{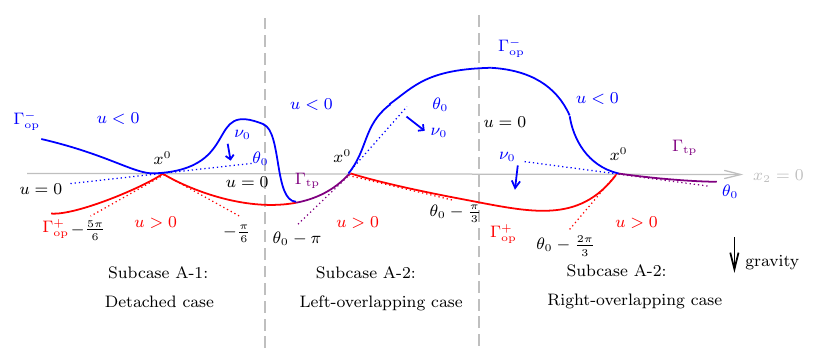}
		\caption{Regular and singular profile of two-phase fluid with gravity.}
		\label{F1}
	\end{figure}
\end{rem}

We start with the variational solution (or the weak solution) $u$ of the two-phase free boundary problem (\ref{eq1.1}). The definitions of the solutions are contained in Appendix B for the sake of completeness.

To obtain the main result we will carry on a blow-up process at the stagnation point $x^0$ to help the singular and regular analysis, which is closely related to the decay rate of $u^\pm$ when approaching the free surface. The Bernstein condition
$$|\nabla u^+|^2\leq Cx_2^- = C\max\{-x_2,0\} \quad \text{in a neighborhood of $x^0$ for some constant $C>0$}$$
implies that $|\nabla u^+|$ decays no slower than $\frac12$-order polynomial rate, and $u^+$ decays no slower than $\frac32$-order polynomial rate. In the partial-degenerate case, the non-degenerate value of
$$|\nabla u^-(x^0)|^2=\lambda>0$$
indicates that the stream function $u^-$ on the interface is expected to decay at a linear rate. In the complete-degenerate case, if we assume that the Bernstein condition holds for both $u^+$ and $u^-$, namely
$$|\nabla u|^2\leq Cx_2^- = C\max\{-x_2,0\} \quad \text{in a neighborhood of $x^0$},$$
then the polynomial decay rates of $u^\pm$ are at most of $\frac32$-order.

Now we state our main theorem for $\lambda>0$, namely the aforementioned Case A. The main results can be divided into two parts: the asymptotics of scaled solution, and the asymptotic behavior of the free surface near the stagnation point $x^0$.

\begin{thmA-1} (Partial-degenerate case, $\lambda>0$)
	Let $u$ be a variational solution of (\ref{eq1.1}) in $D$ with $\lambda>0$, satisfying $|\nabla u^+|^2\leq x_2^-=\max\{-x_2,0\}$ in $D$, where $x^0=(x_1^0,0)\in\Gamma_{\rm bp}$ is a stagnation point of fluid 1 such that $\nabla u^+(x^0)=0$. Let $\{u=0\}$ has locally only finitely many connected components. Then,
	
	{\bf (1) (Asymptotics of $u^-$)} The scaled solution of $u^-$ at $x^0$ converges to a half-plane solution, namely, there exists a unique non-positive function $U_0(x)=-\lambda(x\cdot\nu_0)^-$ with a unique unit vector $\nu_0=\nu_{0,x^0}=(\nu_1^0,\nu_2^0)$ depending only on $x^0$, such that
	$$\frac{u^-(x^0+rx)}{r}\rightarrow -U_0(x)=\lambda(x\cdot\nu_0)^- \quad \text{as} \quad r\rightarrow0+$$
	strongly in $W_{\rm{loc}}^{1,2}(\mathbb{R}^2)$ and locally uniformly in $\mathbb{R}^2$. Moreover, the deflection angle $\theta_0:=\arctan\frac{\nu_1^0}{-\nu_2^0}$ also depends only on $x^0$.
	
	{\bf (2) (Asymptotics of $u^+$)} Meanwhile, suppose furthermore that $u$ is a weak solution of (\ref{eq1.1}) in $D$, then there are only two subcases of the scaled solution of $u^+$ at $x^0$.
	
	(i) (The detached case) If $\left(\partial\{u>0\}\cap\partial\{u<0\}\cap B_r(x^0)\right)$ $\backslash\{x^0\}=\varnothing$ for some small enough $r$, which means that $\Gamma_{\rm tp}$ is an isolated point in $B_r(x^0)$, then the scaled solution of $u^+$ at $x^0$ converges to the Stokes corner flow $V_0(x)$, namely,
	\begin{equation*}
	\frac{u^+(x^0+rx)}{r^{3/2}}\rightarrow V_0(x)=
	\begin{cases}
	\frac{\sqrt{2}}3 \rho^{3/2} \cos \left( \frac32\theta+\frac{3\pi}{4} \right), \quad -\frac{5\pi}{6}<\theta<-\frac{\pi}{6}, \\
	0, \qquad\qquad\qquad\qquad\quad\ \text{otherwise,}
	\end{cases}
	\end{equation*}
	as $r\rightarrow0+$ strongly in $W_{\rm{loc}}^{1,2}(\mathbb{R}^2)$ and locally uniformly in $\mathbb{R}^2$, where $x=(\rho\cos\theta, \rho\sin\theta)$. In this case, the deflection angle
	$\theta_0:=\arctan\frac{\nu_1^0}{-\nu_2^0}\in(-\frac{\pi}{6},\frac{\pi}{6})$. Furthermore, the free boundary $\partial\{u>0\}$ is the union of two $C^1$-graphs in $B_r(x^0)$. (See Figure \ref{corner1-1g}.)
	\begin{figure}[!h] 
		\includegraphics[width=90mm]{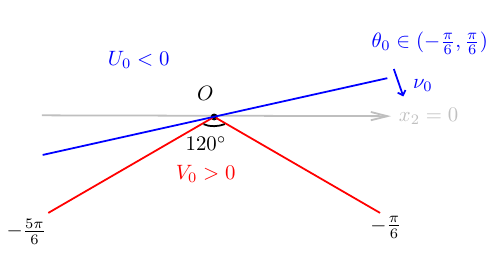}
		\caption{The blow-up limit of $u$ in detached case.}
		\label{corner1-1g}
	\end{figure}
	
	(ii) (The overlapping case) If $\left(\partial\{u>0\}\cap\partial\{u<0\}\cap B_r(x^0)\right)$ $\backslash\{x^0\}\neq\varnothing$ for any small enough $r$, which means that $\Gamma_{\rm tp}$ is not an isolated point in $B_r(x^0)$, then the scaled solution of $u^+$ at $x^0$ converges to the Stokes-type corner flow $V_0(x)$, namely, in the left-overlapping subcase,
	\begin{equation*}
	\frac{u^+(x^0+rx)}{r^{3/2}}\rightarrow V_0(x)=
	\begin{cases}
	\frac23 \sqrt{-\sin\left( \theta_0-\frac{\pi}{3} \right)} \rho^{3/2} \cos \left( \frac32\theta-\frac32\theta_0+\pi \right), \quad \theta_0-\pi<\theta<\theta_0-\frac{\pi}{3}, \\
	0, \qquad\qquad\qquad\qquad\qquad\qquad\qquad\qquad\quad\ \ \text{otherwise},
	\end{cases}
	\end{equation*}
	as $r\rightarrow0+$ strongly in $W_{\rm{loc}}^{1,2}(\mathbb{R}^2)$ and locally uniformly in $\mathbb{R}^2$, and in the right-overlapping subcase,
	\begin{equation*}
	\frac{u^+(x^0+rx)}{r^{3/2}}\rightarrow V_0(x)=
	\begin{cases}
	\frac23 \sqrt{-\sin\left( \theta_0-\frac{2\pi}{3} \right)} \rho^{3/2} \cos \left( \frac32\theta-\frac32\theta_0+\frac{\pi}{2} \right), \quad \theta_0-\frac{2\pi}{3}<\theta<\theta_0, \\
	0, \qquad\qquad\qquad\qquad\qquad\qquad\qquad\qquad\qquad \text{otherwise},
	\end{cases}
	\end{equation*}
	as $r\rightarrow0+$ strongly in $W_{\rm{loc}}^{1,2}(\mathbb{R}^2)$ and locally uniformly in $\mathbb{R}^2$, where $x=(\rho\cos\theta, \rho\sin\theta)$. In these two overlapping subcases, the deflection angle $\theta_0:=\arctan\frac{\nu_1^0}{-\nu_2^0}\in[-\frac{\pi}{3},\frac{\pi}{3}]$ is uniquely determined by $\nu_0$. (See Figure \ref{corner1-2g}.)
	
	Furthermore, the free boundary $\partial\{u>0\}$ is the union of two $C^1$-graphs in $B_r(x^0)$.
	\begin{figure}[!h] 
		\includegraphics[width=160mm]{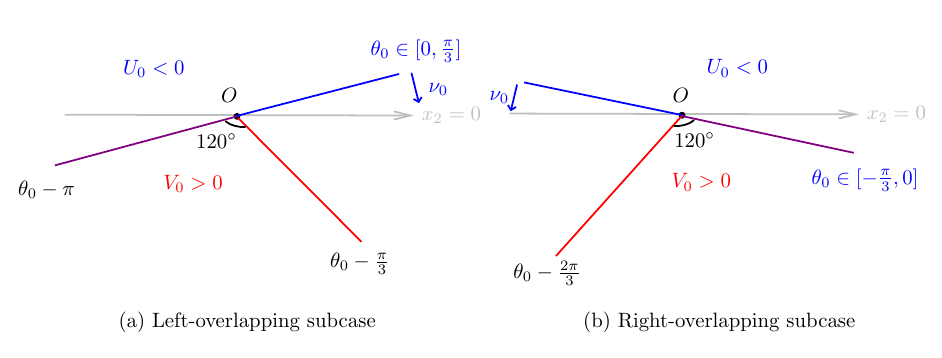}
		\caption{The blow-up limit of $u$ in overlapping case.}
		\label{corner1-2g}
	\end{figure}
	
\end{thmA-1}

\begin{rem}
	We use the notation $-U_0$ and $V_0$ in Theorem A-1 to denote the blow-up limit of $u^-$ and $u^+$ at $x^0$ respectively. In fact, we will prove in the subsequent sections that $U_0$ and $V_0$ have disjoint support, hence the function
	\begin{equation*}
	u_0(x):=
	\begin{cases}
	U_0(x) \quad \text{in} \quad \{x\in\mathbb{R}^2|U_0(x)<0\}, \\
	V_0(x) \quad \text{in} \quad \{x\in\mathbb{R}^2|V_0(x)>0\}, \\
	0, \qquad\ \text{otherwise,}
	\end{cases}
	\end{equation*}
	is well-defined, and we can denote
	$$u_0^-:=\max\{-u_0,0\}=-U_0(x) \quad \text{and} \quad u_0^+:=\max\{u_0,0\}=V_0(x).$$
\end{rem}

\begin{rem}
	The proof of the asymptotic profile of the two-phase fluid is based on the blow-up analysis, monotonicity formula and frequency formula. However, the heuristic behind the proof of Theorem A-1 is that $u^\pm$ possess different decay rates near the stagnation point, leading to different underlying scaling orders. Consequently, the analysis of $u^+$ and $u^-$ must be conducted respectively. In dealing with $u^-$, which decays slower than $u^+$ at linear rate and is proved to be the primary part in $u = u^+ - u^-$, we directly apply the monotonicity formula for the two-phase solution $u$. This approach is valid because the scaled version of $u^+$, under the linear scaling $\frac{u^+(x^0+rx)}{r}$, is approaching $0$ and cannot be seen in the blow-up process. Nevertheless, when handling $u^+$, some technical adjustments to the monotonicity formula are required since $u^-$ will blow up to infinity under the $\frac32$-order polynomial scaling $\frac{u^-(x^0+rx)}{r^{3/2}}$. To overcome this difficulty, we will pick some special test functions to eliminate the terms of $u^-$ in the monotonicity formula and the frequency formula, allowing us to focus only on the behavior of $u^+$.
\end{rem}

\begin{rem}
	To classify the singular profiles at the stagnation point on the two-phase free boundary, we utilize the tool of frequency formula, which was originally used in \cite{A00} for Q-valued harmonic functions and was later developed in \cite{VW11} to investigate the singularity at the stagnation point of one-phase gravity water wave. Notice that Garcia, V$\check{a}$rv$\check{a}$ruc$\check{a}$ and Weiss said in \cite{GVW16} that,
	
	... This is not a complete surprise as there are hitherto no known frequency formulas for two-phase Stefan problems, and possibly the elliptic system is more akin to that group of problems.
	
	They claimed that there had been no frequency formula for two-phase fluids yet. Fortunately, in our case, the free surface for $u^-$ is smooth and to investigate the profile of free surface of $u^+$, we can establish the frequency formula only involving the positive phase $u^+$.
\end{rem}

In the following, we will discover the profile of the free boundaries near the possible stagnation point $x^0=(x_1^0,0)$ of fluid 1, and we assume that $\partial\{u>0\}$ is locally an injective curve, and give the possible profiles for $u^+$ of the free boundaries close to $x^0$ as in Theorem A-2.

\begin{thmA-2} (Partial-degenerate case, $\lambda>0$)
	Let $u$ be a variational solution of (\ref{eq1.1}) in $D$ with $\lambda>0$, satisfying $|\nabla u^+|^2\leq x_2^-=\max\{-x_2,0\}$ in $D$, where $x^0=(x_1^0,0)\in\Gamma_{\rm bp}$ is a stagnation point of fluid 1 such that $\nabla u^+(x^0)=0$. Then at the stagnation point $x^0$,
	
	{\bf (1) (Regular profile of $\partial\{u<0\}$)} the free boundary of $\partial\{u<0\}$ is locally a $C^{1,\alpha}$-curve in a small neighborhood of $x^0$.
	
	{\bf (2) (Singular profile of $\partial\{u>0\}$)} Suppose furthermore that $u$ is a weak solution of (\ref{eq1.1}) in $D$, and the free boundary $\partial\{u>0\}$ is a continuous injective curve $\sigma(t)=(\sigma_1(t),\sigma_2(t))$ with $t\in(-1,1)$ such that $\sigma(0)=x^0$, then there are only two cases.
	
	In the detached case, $\partial\{u>0\}$ is the union of two $C^1$-graphs in a small neighborhood $B_r(x^0)$ with $r<R$ of functions
	$$\eta_1:(x_1^0-\delta, x_1^0]\rightarrow\mathbb{R} \quad \text{and} \quad \eta_2:[x_1^0,x_1^0+\delta)\rightarrow\mathbb{R}$$
	which are both continuously differentiable up to $x^0$ and satisfy
	$$\eta_1'(x_1^0-)=\frac{1}{\sqrt{3}} \quad \text{and} \quad \eta_2'(x_1^0+)=-\frac{1}{\sqrt{3}}.$$
	(See Figure \ref{corner2-1g}.)
	\begin{figure}[!h]
		\includegraphics[width=95mm]{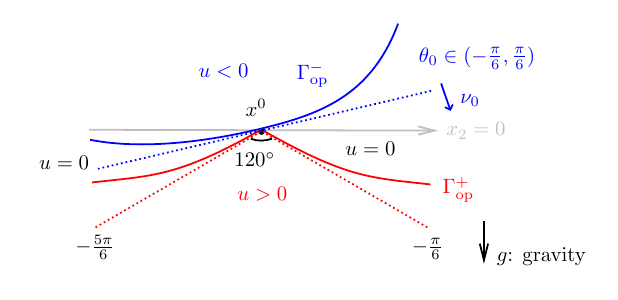}
		\caption{The asymptotic profile of the free boundaries in detached case}
		\label{corner2-1g}
	\end{figure}
	
	In the overlapping case, $\partial\{u>0\}$ is the union of two $C^1$-graphs in a small neighborhood $B_r(x^0)$ with $r<R$ of functions
	$$\eta_1:(x_1^0-\delta, x_1^0]\rightarrow\mathbb{R} \quad \text{and} \quad \eta_2:[x_1^0,x_1^0+\delta)\rightarrow\mathbb{R}$$
	which are both continuously differentiable up to $x^0$ and satisfy
	\begin{equation*}
	\begin{cases}
	\eta_1'(x_1^0-)=\tan\theta_0 \quad \text{and} \quad \eta_2'(x_1^0+)=\tan\left(\theta_0-\frac{\pi}{3}\right) \quad\ \text{in left-overlapping case,} \\
	\eta_1'(x_1^0-)=\tan\left( \theta_0-\frac{2\pi}{3} \right) \quad \text{and} \quad \eta_2'(x_1^0+)=\tan\theta_0 \quad \text{in right-overlapping case,}
	\end{cases}
	\end{equation*}
	where $\theta_0$ is uniquely determined by $\nu_0$. (See Figure \ref{corner2-2g}.)
	\begin{figure}[!h]
		\includegraphics[width=160mm]{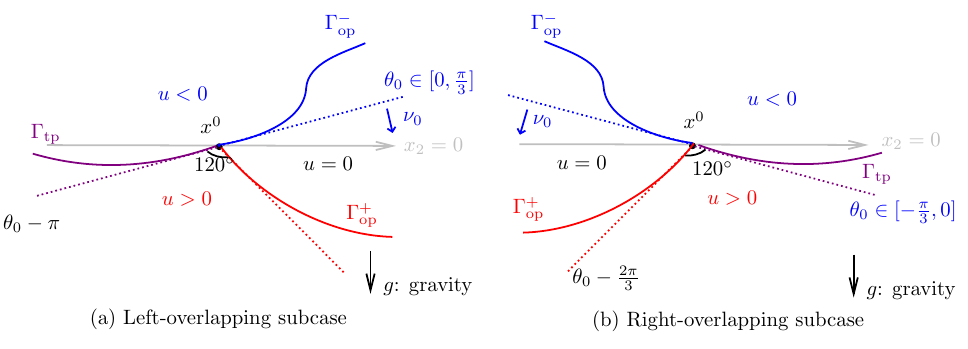}
		\caption{The asymptotic profile of the free boundaries in overlapping case}
		\label{corner2-2g}
	\end{figure}
\end{thmA-2}

\begin{rem}
	Our results coincide with the regularity result of the free boundary for one-phase flow in \cite{AC81}\cite{C87}\cite{C88} and \cite{C89} once $u^+\equiv0$, and also coincide with the singularity analysis of the stagnation point for one-phase gravity water wave in \cite{VW11} once $u^-\equiv0$.
\end{rem}

\begin{rem}
	In the detached case, the two fluids contact at only one point $\Gamma_{\rm bp}$, and the two-phase free boundary $\Gamma_{\rm tp}$ is an isolated point. The singular profile of the positive phase in this case is a symmetric Stokes corner, and the interaction of the two phases can be ignored. However, in the overlapping case, the two fluids stick together and push against each other on a common section of the two-phase free boundary, which breaks the symmetry of the asymptotic Stokes corner of $\partial\{u>0\}$.
\end{rem}

\begin{rem}
	The convexity of a periodic Stokes wave of extreme form is established by Plotnikov and Toland in \cite{PT04}. Hence, the profile of the Stokes-type corner with the angle of $120^\circ$ for the free boundary of $u^+$ at the stagnation point $x^0$ can be hence confirmed convex.
\end{rem}

\begin{rem}
	The classification of the detached case and the overlapping case is complete because the stagnation point $x^0\in\Gamma_{\rm bp}$ is either an isolated point of $\Gamma_{\rm tp}$ or not. We give a further explanation about the division of left-overlapping and right-overlapping. Naturally in a small neighborhood $B_R(x^0)$ of $x^0$, the free boundary $\partial\{u>0\}\cap B_R(x^0)$ falls in the third quadrant $\mathcal{Q}_3=\{(x_1,x_2)\ |\ x_1<x_1^0,x_2\leq0\}$ and the fourth quadrant $\mathcal{Q}_4=\{(x_1,x_2)\ |\ x_1>x_1^0,x_2\leq0\}$ since it has an angle of $120^\circ$ asymptotically in $\{x_2<0\}\cap B_R(x^0)$. Hence, if it shares a common boundary with $\partial\{u<0\}$ in $B_R(x^0)\backslash\{x^0\}$, which is nearly a half-plane, then the overlapping part $\Gamma_{\rm tp}$ falls either in the third quadrant, namely the left-overlapping case, or in the fourth quadrant, namely the right-overlapping case. These two subcases have no essential distinction. See Figure \ref{lr} for an illustration of either scenario.
\end{rem}

\begin{figure}[!h]
	\includegraphics[width=160mm]{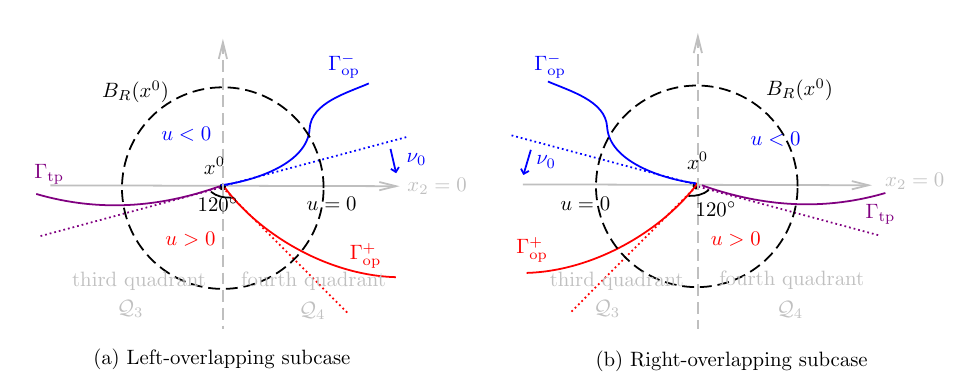}
	\caption{The division of left-overlapping and right-overlapping}
	\label{lr}
\end{figure}

\begin{rem}
	The positive phase $u^+$ occupies asymptotically a corner of $120^\circ$ acted on by gravity, which also has influence on the negative phase $u^-$. In the one-phase case when $u_0^+\equiv0$, the deflection angle $\theta_0$ of $u^-$ at regular free boundary point falls in $[-\pi,\pi)$, while in the two-phase situation,
	\begin{equation*}
	\theta_0\in
	\begin{cases}
	(-\frac{\pi}{6},\frac{\pi}{6}) \quad \text{in the detached case}, \\
	[-\frac{\pi}{3},\frac{\pi}{3}] \quad\, \text{in the overlapping case}.
	\end{cases}
	\end{equation*}
\end{rem}

Now we come to the case $\lambda=0$, namely the aforementioned Case B. In this case, $x^0=(x_1^0,0)$ is a stagnation point for both fluid 1 and fluid 2, namely, $|\nabla u^+(x^0)|=|\nabla u^-(x^0)|=0$. Naturally, we are interested in such points on $\Gamma_{\rm tp}$. However, the following Theorem B says that there is no such stagnation point on two-phase free boundary $\Gamma_{\rm tp}$, and the complete degenerate stagnation points must lie on the one-phase free boundaries $\Gamma_{\rm op}^\pm$.

\begin{thmB} (Complete-degenerate case, $\lambda=0$)
	Let $u$ be a weak solution of (\ref{eq1.1}) in $D$ with $\lambda=0$, and $x^0=(x_1^0,0)$ be a stagnation point on the free boundaries $\partial\{u>0\}\cup\partial\{u<0\}$ such that the following conditions hold:
	
	(1) $|\nabla u^+(x^0)|=|\nabla u^-(x^0)|=0$;
	
	(2) The strong Bernstein estimate holds that $|\nabla u|^2\leq x_2^-=\max\{-x_2,0\}$ in $D$;
	
	(3) $\{u=0\}$ has locally only finitely many connected components.
	
	Then such stagnation point $x^0$ satisfying (1)-(3) must be a one-phase free boundary point on $\Gamma_{\rm op}^+$ or $\Gamma_{\rm op}^-$. The scaled solution of $u$ converges to a nonnegative or nonpositive function
	\begin{equation*}
	\frac{u(x^0+rx)}{r^{3/2}}\rightarrow u_0(x)=
	\begin{cases}
	\pm\frac{\sqrt{2}}3 \rho^{3/2} \cos \left( \frac32\theta+\frac{3\pi}{4} \right), \quad -\frac{5\pi}{6}<\theta<-\frac{\pi}{6}, \\
	0, \qquad\qquad\qquad\qquad\quad\quad \text{otherwise,}
	\end{cases}
	\end{equation*}
	as $r\rightarrow0+$ strongly in $W_{\rm{loc}}^{1,2}(\mathbb{R}^2)$ and locally uniformly in $\mathbb{R}^2$, where $x=(\rho\cos\theta, \rho\sin\theta)$.
	
	Furthermore, suppose that the free boundary $\partial\{u>0\}$ or $\partial\{u<0\}$ is a continuous injective curve $\sigma(t)=(\sigma_1(t),\sigma_2(t))$ such that $\sigma(0)=x^0=(x^1_0,0)$. Then the free boundary $\partial\{u>0\}$ or $\partial\{u<0\}$ is the union of two $C^1$-graphs in a small neighborhood $B_r(x^0)$ of functions
	$$\eta_1:(x_1^0-\delta, x_1^0]\rightarrow\mathbb{R} \quad \text{and} \quad \eta_2:[x_1^0,x_1^0+\delta)\rightarrow\mathbb{R}$$
	which are both continuously differentiable up to $x^0$ and satisfy
	$$\eta_1'(x_1^0-)=\frac{1}{\sqrt{3}} \quad \text{and} \quad \eta_2'(x_1^0+)=-\frac{1}{\sqrt{3}}.$$
	See Figure \ref{corner3}.
	\begin{figure}[!h] 
		\includegraphics[width=130mm]{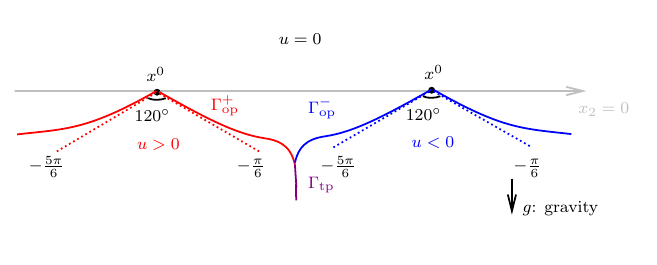}
		\caption{Wave profile for $\lambda=0$.}
		\label{corner3}
	\end{figure}
\end{thmB}

\begin{rem}
	Theorem B coincides with the result in one-phase case. In fact, if $u^-\equiv0$ in the fluid domain, then the stagnation point is naturally supposed to be on $\partial\{u>0\}$, and the singular profile of $\partial\{u>0\}$ at $x^0$ is asymptotically a symmetric $120^\circ$ Stokes corner.
\end{rem}

\begin{rem}
	It is straightforward to deduce that if $u$ satisfies the conditions in Theorem B when $\lambda=0$, then $|\nabla u^\pm|$ will not vanish on its two-phase free boundary $\Gamma_{\rm tp}$, which implies that $\Gamma_{\rm tp}$ is $C^{1,\alpha}$ smooth by \cite{PSV21}, and can be bootstrapped to higher regularity.
\end{rem}

Theorem B excludes the possible case in Figure \ref{corner3-1} for $\lambda=0$, where $x^0$ is a two-phase stagnation point, and $u^+$ and $u^-$ share a corner of $120^\circ$ asymptotically. It shows that the complete-degenerate case $\lambda=0$ is quite different from the partial-degenerate case $\lambda>0$ since the stagnation point can only be on the one-phase free boundary. We proceed by way of contradiction and suppose $x^0\in\Gamma_{\rm tp}$ is a stagnation point. The assumption of strong Bernstein estimate for both $u^\pm$ shows that neither of $u^\pm$ could have cusp behavior, while the monotonicity formula for two-phase fluid implies that each connector component for the blow-up $u_0^\pm$ should be a corner of $120^{\circ}$. These two facts are not compatible, hence give the desired conclusion. Once we exclude the case that the stagnation point $x^0\in\Gamma_{\rm tp}$ and derive that $x^0\in\Gamma_{\rm op}^\pm$, the singular profile given in Theorem B is a straightforward consequence according to the one-phase work \cite{VW11}.

\begin{figure}[!h] 
	\includegraphics[width=100mm]{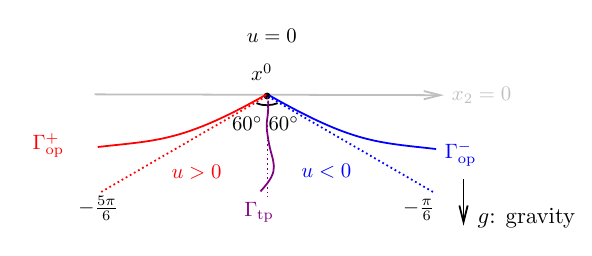}
	\caption{An excluded singular profile.}
	\label{corner3-1}
\end{figure}

Before closing this introduction, we will give an overview of the organization of this paper. In Section 2, we consider the case A where $\lambda>0$. We give two monotonicity formulas for both $u^+$ and $u^-$ in Section 2.1 to study blow-up limits. In this case, the free boundary of $u^-$ can be proved to be $C^{1,\alpha}$-smooth, and it remains to analyze the singularity of the free boundary $\partial\{u>0\}$ for $u^+$. Furthermore, we get the Lebesgue densities for $u^+$ in Section 2.2, leading the possible asymptotic singular profiles, such as Stokes-type corner, cusp, horizontal flatness of the free boundary of $u^+$ near the stagnation point. In Section 2.3-2.5 we exclude the cases with cusps and horizontally flat singularities, and obtain the consequence of Theorem A-1 and Theorem A-2. In Section 3, we consider the case B where $\lambda=0$, using monotonicity formula and the strong Bernstein estimate to prove Theorem B. The Appendix is a compensate for some definitions, and the concentration of compactness in frequency formula.

\section{The partial-degenerate case}

In this section, we consider the partial-degenerate case $\lambda>0$, namely $|\nabla u^+(x^0)|=0$ and $|\nabla u^-(x^0)|>0$ for the stagnation point $x^0=(x_1^0,0)\in\Gamma_{\rm tp}$, which means that the velocity of the fluid 1 is degenerate at $x^0$, while the velocity of the fluid 2 does not degenerate at the two-phase interface point $x^0$. We will use the tool of monotonicity formula to make a blow-up analysis at $x^0$ for $u^\pm$ respectively, which is substantially different from the one-phase work \cite{VW11} since the convergence rates of $u^\pm$ approaching $0$ are different. We will deal with $u^-$ first, obtaining the regularity of $\partial\{u<0\}$ near $x^0$, and then handle $u^+$ and get its weighted density at $x^0$, roughly $\frac{|\{u>0\}\cap B_r(x^0)|}{|B_r(x^0)|}$ for small enough $r$. It is then possible to exclude cusps and horizontally flat singularities for $u^+$ by strong Bernstein assumption and frequency formula. All those discussion leads to Theorem A-1 and A-2.

Unless stated otherwise, in this section we suppose the stagnation point $x^0=(x_1^0,0)\in\Gamma_{\rm bp}$ is of fluid 1 such that $|\nabla u^+(x^0)|=0$. For clarity of exposition, we denote the set of \emph{branch stagnation points}
$$S^u_+:=\left\{ x=(x_1,0)\in \Gamma_{\rm bp} \ \Big| \ |\nabla u^+(x)|=0 \right\},$$
where $u$ is a variational solution of (\ref{eq1.1}) in $D$ with $\lambda>0$.

We first give the following assumptions, which is somewhat equivalent to the Bernstein estimate to the positive-phase fluid.

\begin{assume} \label{assume}
	(Bernstein estimate) There is a positive constant $C$ such that $|\nabla u^+|^2 \leq C x_2^-$ locally in $D$.
	
\end{assume}

The assumption \ref{assume} can also be found in \cite{VW11} for one-phase water waves as a version of Rayleigh-Taylor condition, providing some growth assumptions on $u$ to help us study the asymptotic behavior of $u$ near the stagnation point $x^0$. Here in two-phase case, we continue to require such assumption due to the fact that $u^+$ satisfies
$$\Delta (|\nabla u^+|^2)\geq0 \quad \text{in} \quad D\cap\{u>0\}$$
from direct calculation, which means that the pressures $P^+$ of the positive phase is a superharmonic function by means of Bernoulli's law. Under suitable boundary conditions we can construct barrier solution to deduce such growth condition for $u^+$. Similarly we can also deduce that
$$|\nabla u^-|^2 \leq C \quad \text{locally in $D$},$$
which is a natural byproduct implying that $u$ is Lipschitz.

\subsection{Monotonicity formula}

The tool of monotonicity formula is widely employed in blow-up analysis, which plays a key role in proving the homogeneity of the blow-up limit, indicating the regularity or singularity and helping the analysis on the explicit form of blow-up limit. It is a primary approach to study the regularity of free boundary, for example in \cite{AC81} and \cite{V23} for one-phase problem, in \cite{ACF84} and \cite{PSV21} for two-phase problem. To investigate the singular profile of the free boundary near the stagnation points, V$\check{a}$rv$\check{a}$ruc$\check{a}$ and Weiss introduced firstly the monotonicity formula in one-phase water wave with gravity in \cite{VW11}, and it was generated in the case with vorticity in \cite{VW12}, the axisymmetric case in \cite{VW14}, and the compressible case in \cite{DY}.

In what follows we will construct two monotonicity formulas respectively for $u^-$ and $u^+$, which correspond to their different decay rates as analysed in Section 1.5.

\subsubsection{Monotonicity formula for $u^-$}

In the following context we recall some known results for $u^-$ near $x^0$, where $|\nabla u^-(x^0)|>0$, referred to \cite{ACF84},\cite{PSV21} and \cite{S11} and references therein. Notice that the Bernstein estimate
$$|\nabla u^+|^2 \leq Cx_2^- \quad \text{locally in $D$}$$
implies that $u^+$ decays no slower than $\frac32$-order polynomial rate near $x^0$, and
$$|\nabla u^-|^2 \leq C \quad \text{locally in $D$}$$
implies that $u^-$ decays no slower than a linear function near $x^0$. Hence, $u^-$ plays a primary part in $u=u^+ - u^-$ comparing the convergence rates of $u^\pm$, and we construct the monotonicity formula in Proposition \ref{mono1} for $u^-$, which involves terms of $u^+$ but have no bad influence when calculating $u^-$ since these terms approach $0$.

Before we give the precise statement, we first define the following energy functionals in $B_r(x^0)$,
\begin{equation}
I(r)=I_{x^0,u}(r)=r^{-2}\int_{B_r(x^0)}\left( |\nabla u|^2 + (-x_2)\chi_{\{u>0\}} + (-x_2+\lambda^2)\chi_{\{u<0\}} \right)dx,
\end{equation}
and the $L^2$-boundary energy,
\begin{equation}
J(r)=J_{x^0,u}(r)=r^{-3}\int_{\partial B_r(x^0)}  u^2 d\mathcal{S},
\end{equation}
where $d\mathcal{S}=d\mathcal{H}^1$ with $\mathcal{H}^1$ denotes the $1$-dimensional Hausdorff measure, and the so-called \emph{Weiss boundary adjusted energy}
\begin{equation}
\begin{aligned}
M(r)&=M_{x^0,u}(r) \\
&=I(r)-J(r) \\
&=r^{-2}\int_{B_r(x^0)}\left( |\nabla u|^2 + (-x_2)\chi_{\{u>0\}} + (-x_2+\lambda^2)\chi_{\{u<0\}} \right)dx-r^{-3}\int_{\partial B_r(x^0)}  u^2 d\mathcal{S}.
\end{aligned}
\end{equation}
These definitions were first introduced by Weiss in \cite{VW11}, Theorem 3.5, which gave a one-phase monotonicity formula at non-degenerate free boundary point. In the two-phase case, we utilize this monotonicity formula to carry on the blow-up analysis for $u^-$, since the terms of $u^+$ is roughly disappeared in the scaling version of $M(r)$.

\begin{prop}(Monotonicity formula for $u$) \label{mono1}
	Let $u$ be a variational solution of (\ref{eq1.1}) in $D$ with $\lambda>0$ and $x^0=(x_1^0,0)\in\Gamma_{\rm tp}$.
	Then, for a.e. $r\in(0,R/2)$,
	\begin{equation*}
	\frac{d}{dr} M(r)=2r^{-2}\int_{\partial B_r(x^0)} \left( \nabla u\cdot\nu-\frac{u}r \right)^2 d\mathcal{S} - r^{-3}\int_{B_r(x^0)} x_2 \chi_{\{u\neq0\}} dx .
	\end{equation*}
\end{prop}

\begin{proof}
	Utilizing the identity in dimension 2 that
	$$\frac{d}{dr}\left( r^\alpha \int_{\partial B_r(x^0)} w^2 \right)d\mathcal{S}=(\alpha+1) r^{\alpha-1}\int_{\partial B_r(x^0)}w^2 d\mathcal{S} + r^\alpha\int_{\partial B_r(x^0)} 2w\nabla w\cdot\nu d\mathcal{S},$$
	where $\nu$ denotes the unit outer normal of $\partial B_r(x^0)$, and for any $w\in W^{1,2}(B_r(x^0))$ and any $\alpha\in\mathbb{R}$, we calculate the derivative of $M(r)$ directly,
	\begin{equation} \label{2.7}
	\begin{aligned}
	\frac{d}{dr}M(r) &= r^{-2}\int_{\partial B_r(x^0)} \left( |\nabla u|^2 - x_2\chi_{\{u>0\}} + (-x_2+\lambda^2)\chi_{\{u<0\}} \right)d\mathcal{S} \\
	&\quad - 2r^{-3}\int_{B_r(x^0)}\left( |\nabla u|^2 - x_2\chi_{\{u>0\}} + (-x_2+\lambda^2)\chi_{\{u<0\}} \right)dx \\
	&\quad +2 r^{-4}\int_{\partial B_r(x^0)} u^2 d\mathcal{S} -2r^{-3}\int_{\partial B_r(x^0)} u\nabla u\cdot\nu d\mathcal{S} \\
	&=r^{-2}\int_{\partial B_r(x^0)} \left( |\nabla u|^2 - x_2\chi_{\{u>0\}} + (-x_2+\lambda^2)\chi_{\{u<0\}} \right)d\mathcal{S} \\
	&\quad + 2r^{-3}\int_{B_r(x^0)} x_2\chi_{\{u>0\}} + (x_2-\lambda^2)\chi_{\{u<0\}} dx \\
	&\quad -4r^{-3}\int_{\partial B_r(x^0)} u\nabla u\cdot\nu d\mathcal{S} + 2 r^{-4}\int_{\partial B_r(x^0)} u^2 d\mathcal{S}.
	\end{aligned}
	\end{equation}
	Notice that in the last equality we have used the fact
	$$\int_{B_r(x^0)} |\nabla u|^2 dx = \int_{\partial B_r(x^0)} u\nabla u\cdot\nu d\mathcal{S},$$
	which is approximated by
	$$\int_{B_r(x^0)} \nabla u^\pm \cdot \nabla(\max\{u^\pm - \epsilon,0\}^{1+\epsilon}) dx = \int_{\partial B_r(x^0)} \max\{u^\pm - \epsilon,0\}^{1+\epsilon}\nabla u^\pm\cdot\nu d\mathcal{S}$$
	while $\epsilon\rightarrow0$.
	
	Now for small $\kappa$ and $\eta_\kappa(t;r):=\max\{0,\min\{1,\frac{r-t}{\kappa}\}\}$, we take after approximation $\bm \phi_\kappa(x)=\eta_\kappa(|x-x^0|;r)(x-x^0)\in W_0^{1,2}(B_r(x^0);\mathbb{R}^2)$ as a test function in the definition of the variational solution $u$. We obtain
	\begin{equation*}
	\begin{aligned}
	0 &= \int_{D} \left( |\nabla u|^2 - x_2\chi_{\{u>0\}} +(-x_2+\lambda^2)\chi_{\{u<0\}} \right) \left( 2\eta_\kappa(|x-x^0|;r)+\eta_\kappa'(|x-x^0|;r)|x-x^0| \right)dx \\
	& \quad -2\int_{D} \left( |\nabla u|^2 \eta_\kappa(|x-x^0|;r) + \left( \nabla u \cdot \frac{x-x^0}{|x-x^0|} \right)^2 \eta_\kappa'(|x-x^0|;r)|x-x^0| \right)dx \\
	& \quad - \int_{D} x_2\eta_\kappa(|x-x^0|;r) \chi_{\{u\neq0\}} dx.
	\end{aligned}
	\end{equation*}
	Passing the limit $\kappa\rightarrow0$, we have $\eta_\kappa\rightarrow \chi_{B_r(x^0)}$ pointwise and it follows from the fact that $\int_{D} \eta_\kappa'(|x-x^0|;r)f(x)dx \rightarrow \int_{\partial B_r(x^0)} -f(x)d\mathcal{S}$ for any $f(x)\in L^2(D)$, one gets
	\begin{equation*}
	\begin{aligned}
	0&=2\int_{B_r(x^0)}\left( |\nabla u|^2 - x_2\chi_{\{u>0\}} + (-x_2+\lambda^2)\chi_{\{u<0\}} \right)dx \\
	& \quad -r\int_{\partial B_r(x^0)} \left( |\nabla u|^2 - x_2\chi_{\{u>0\}} + (-x_2+\lambda^2)\chi_{\{u<0\}} \right)d\mathcal{S} \\
	& \quad +2r\int_{\partial B_r(x^0)} (\nabla u\cdot\nu)^2 d\mathcal{S} -2\int_{B_r(x^0)}|\nabla u|^2 dx -\int_{B_r(x^0)}  x_2\left( \chi_{\{u>0\}} + \chi_{\{u<0\}} \right) dx \\
	&= 2r\int_{\partial B_r(x^0)}(\nabla u\cdot\nu)^2 d\mathcal{S} -r\int_{\partial B_r(x^0)} \left( |\nabla u|^2 - x_2\chi_{\{u>0\}} + (-x_2+\lambda^2)\chi_{\{u<0\}} \right)d\mathcal{S} \\
	& \quad -\int_{B_r(x^0)} 3x_2\chi_{\{u>0\}} dx + \int_{B_r(x^0)} (-3 x_2+2\lambda^2)\chi_{\{u<0\}} dx.
	\end{aligned}
	\end{equation*}
	Plugging this into (\ref{2.7}), we obtain that for a.e. $r\in(0,R/2)$,
	\begin{equation*}
	\begin{aligned}
	\frac{d}{dr}M(r)&=\frac{2}{r^2}\int_{\partial B_r(x^0)} (\nabla u\cdot\nu)^2 d\mathcal{S} -\frac4{r^3}\int_{\partial B_r(x^0)} u\nabla u\cdot\nu d\mathcal{S} +\frac2{r^4}\int_{\partial B_r(x^0)} u^2 d\mathcal{S} \\
	&\quad - r^{-3}\int_{B_r(x^0)} x_2\left( \chi_{\{u>0\}} + \chi_{\{u<0\}} \right) dx \\
	&= 2r^{-2}\int_{\partial B_r(x^0)} \left( \nabla u\cdot\nu -\frac{u}r \right)^2 d\mathcal{S} - r^{-3}\int_{B_r(x^0)} x_2\left( \chi_{\{u>0\}} + \chi_{\{u<0\}} \right) dx.
	\end{aligned}
	\end{equation*}
\end{proof}

\subsubsection{Blow-up analysis and regularity for $\partial\{u<0\}$}

We will carry on a blow-up process around the stagnation point $x^0=(x_1^0,0)\in S^u_+$, reducing the problem to the analysis of blow-up limits. Consider the blow-up sequence
\begin{equation} \label{Uk}
U_k=\frac{u(x^0+r_k x)}{r_k}
\end{equation}
for $r_k\rightarrow0+$ and $x\in B_{R/r_k}(0)$ such that $x^0+r_k x\in B_R(x^0)\Subset D$. Lemma \ref{lem1} in the following gives the unique form of the blow-up limit $U_0:=\lim_{r_k\rightarrow0+} U_k$, namely the half-plane function $U_0=-\gamma_0(x\cdot\nu_0)^-$ for some unit vector $\nu_0$, where $\gamma_0$ and $\nu_0$ depend only on $x^0$ but not determined uniquely yet. In fact, the uniqueness of the constant $\gamma_0$ and the direction $\nu_0$ will be deduced utilizing Proposition \ref{ep}, which are crucial to verify the $C^{1,\alpha}$ regularity of $\partial\{u<0\}$.

\begin{lem} \label{lem1}
	Let $u$ be a variational solution of (\ref{eq1.1}) in $D$ with $\lambda>0$ satisfying Assumption \ref{assume}, and $x^0\in S^u_+$. Let $U_k$ be a blow-up sequence of $u$ at $x^0$, defined as in (\ref{Uk}), that converges weakly in $W_{\rm{loc}}^{1,2}(\mathbb{R}^2)$ to a blow-up limit $U_0$. Then $U_k$ converges strongly to $U_0$ in $W_{\rm{loc}}^{1,2}(\mathbb{R}^2)$, and $U_0$ is a homogeneous function of degree 1. Moreover, $U_0=-\gamma_0(x\cdot\nu_0)^-$ for some $\gamma_0\geq\lambda$ and $\nu_0\in\partial B_1$.
\end{lem}

\begin{proof}
	The strong convergence from $U_k$ to $U_0$ in $W_{\rm{loc}}^{1,2}(\mathbb{R}^2)$ comes straightly from the fact that
	$$\limsup_{r_k\rightarrow0+}\int_{B_{R/r_k}(0)}|\nabla U_k|^2\eta dx \leq \int_{\mathbb{R}^2}|\nabla U_0|^2\eta dx$$
	for any $\eta\in C_0^1(\mathbb{R}^2)$. In fact,
	$$\int_{B_{R/r_k}(0)}|\nabla U_k|^2\eta dx = \int_{B_{R/r_k}(0)} -U_k \nabla U_k\cdot\nabla\eta dx,$$
	and the right-hand side converges to
	$$\int_{\mathbb{R}^2} -U_0\nabla U_0\cdot\nabla\eta dx = \int_{\mathbb{R}^2}|\nabla U_0|^2\eta dx.$$
	Moreover, the harmonicity of $U_k$ in $\{U_0\neq0\}$ implies that $U_0$ is also harmonic in $\{U_0\neq0\}$. On the other hand, the monotonicity formula in Proposition \ref{mono1} gives that for all $0<\rho<\sigma<\infty$ and $r_k>0$,
	\begin{equation*}
	\begin{aligned}
	M(r_k\sigma)-M(r_k\rho)&=\int_{r_k\rho}^{r_k\sigma} \left[ 2 r^{-2}\int_{\partial B_r(x^0)}\left( \nabla u\cdot\nu-\frac{u}{r}  \right)^2d\mathcal{S} -r^{-3}\int_{B_r(x^0)}x_2\left( \chi_{\{u>0\}}+\chi_{\{u<0\}} \right)dx \right]dr \\
	&=2\int_{\rho}^{\sigma} r^{-4} \int_{\partial B_r}\left( \nabla U_k\cdot x - U_k \right)^2 d\mathcal{S} dr - 2\int_{r_k\rho}^{r_k\sigma} r^{-3}\int_{B_r(x^0)} x_2\chi_{\{u\neq0\}} dx dr \\
	&=2\int_{B_\sigma(0)\backslash B_\rho(0)} |x|^{-4} \left( \nabla U_k\cdot x - U_k \right)^2 dx + o(1),
	\end{aligned}
	\end{equation*}
	where $o(1)$ is the infinitesimal as $r_k\rightarrow0+$ since $|x_2|\leq r$ and $\chi_{\{u\neq0\}}\leq1$ in $B_r(x^0)$. The left-hand side of the above equation converges to $0$ as $r_k\rightarrow0+$, which yields the desired homogeneity of $U_0$.
	
	Now we can write in polar coordinates that $U_0(x)=U_0(\rho,\theta)=\rho U(\theta)$ with $x_1=\rho\cos\theta$, $x_2=\rho\sin\theta$. In particular, $U(\theta)$ is an eigenfunction of the spherical Laplacian $\Delta_{\mathbb{S}}=\partial_{\theta\theta}$ on the spherical sets $\{u\neq0\}\cap\partial B_1$, corresponding to the eigenvalue $1$ on the unit sphere $\partial B_1$ centered at $0$ in $\mathbb{R}^2$. In other words,
	$$-\Delta_{\mathbb{S}} U = U \quad \text{in} \quad \{U(\theta)\neq0\}\cap\partial B_1.$$
	Recalling the assumption $|\nabla u^+(x)|^2\leq C x_2^-$ locally in $D$, we have that $u^+(y)\leq C \sqrt{y_2^-} r$ in a neighborhood $B_r(x^0)$ of $x^0=(x_1^0,0)$, which leads to the fact that $U_k(x)\leq C\sqrt{r_k x_2^-}$ in $B_1$ and $U_0^+(x)\equiv0$ in $B_1$. Hence, $U_0^+(x)\equiv0$ in $\mathbb{R}^2$ by the homogeneity of $U_0$, which implies that $U(\theta)\leq0$.
	
	Since the 1-eigenspace in dimension 2 contains only linear functions, one easily deduce that there exists some $\nu_0\in\partial B_1$, such that $U_0$ is of the form
	$$U_0(x)= - \gamma_0(x\cdot\nu_0)^- \quad \text{with} \quad \gamma_0>0,$$
	or
	$$U_0(x)\equiv0,$$
	or
	$$U_0(x)=-\gamma_0|x\cdot\nu_0| \quad \text{with} \quad \gamma_0>0.$$
	The last two cases can be excluded using Lemma 4.3 in \cite{VW11} and Proposition \ref{prop1} in the subsequent context, see Remark \ref{rem1} below.
	
	Moreover, utilizing the fact that $U_k\rightarrow U_0$ strongly in $W_{\rm{loc}}^{1,2}(\mathbb{R}^2)$ and $\chi_{\{U_k<0\}}\rightarrow \chi_{\{U_0<0\}}$ strongly in $L_{\rm{loc}}^1(\mathbb{R}^2)$, the additional free boundary condition that $|\nabla u^-|^2\geq -x_2+\lambda^2$ on $\partial\{u<0\}$ gives that
	$$|\nabla U_0^-(x)|\geq\lambda \quad \text{on} \quad \partial\{U_0<0\}.$$
	This together with the fact that $U_0^+\equiv0$ gives that
	$$U_0(x)=-\gamma_0(x\cdot\nu_0)^- \quad \text{for some $\nu_0\in \partial B_1$}$$
	with $\gamma_0\geq\lambda>0$.
\end{proof}

\begin{rem} \label{rem1}
	Lemma \ref{lem1} is compatible with the result by De Silva and Savin in \cite{SS19} that the blow-up should be either a two plane solution $U_0(x)=\beta_0(x\cdot\nu_0)^+ - \gamma_0(x\cdot\nu_0)^-$ with unique $\beta_0, \gamma_0>0$ and $\nu_0\in\partial B_1$ (which should be the blow-up limit at the non-degenerate free boundary point $y^0$ where $|\nabla u^\pm(y^0)|\neq0$), or $U_0^+\equiv0$ (which is exactly the blow-up limit at the stagnation point $x^0$). We exclude the cases $U_0\equiv0$ and $U_0 = -\gamma_0|x\cdot\nu_0|$ for $\gamma_0>0$ with the help of Lemma 4.3 in \cite{VW11} and Proposition \ref{prop1} in this paper. In fact, these two cases imply that either $\chi_{\{U_k<0\}}\rightarrow0$ in $L_{\rm{loc}}^1(\mathbb{R}^2)$, which can be excluded by way of the method in \cite{VW11}, or that $\chi_{\{u_k^+>0\}}\rightarrow0$ in $L_{\rm{loc}}^1(\mathbb{R}^2)$, were $u_k^+:=\frac{u(x^0+r_kx)}{r_k^{3/2}}$. However, the proof of Proposition \ref{prop1}, which is independent of the property of $u^-$, shows that the limit $\chi_{\{u>0\}}\rightarrow0$ in $L^1_{\rm{loc}}(\mathbb{R}^2)$ is invalid.
\end{rem}

With the explicit form of blow-ups for $u^-$, we can follow the steps in \cite{PSV21} to study the $\epsilon$-regularity for $u^-$, namely the flatness decay for $\partial\{u<0\}$. De Philippis, Spolaor and Velichkov in \cite{PSV21} reduced the following Proposition \ref{ep} to two key ingredients, partial boundary Harnack inequality and analysis of the linearized problem, giving a groundbreaking proof of the regularity of the two-phase free boundaries in any dimensions. The method is found applicable to several kinds of problems, such as in \cite{MTV23} for solutions of a free boundary system arising in shape optimization problems. It can also be employed in our case for $u^-$ since $|\nabla u^-(x^0)|\geq\lambda>0$.

\begin{prop}($\epsilon$-regularity for $u^-$) \label{ep}
	Let $u$ be a variational solution of (\ref{eq1.1}) in $D$ with $\lambda>0$ and $x^0\in S^u_+$. Let $\gamma\geq\lambda>0$. Then there exist $\epsilon_0>0$ and $C>0$ such that if the blow-up sequence $U_k^-(x)$ satisfies
	$$-\gamma(x\cdot\nu-\epsilon)^-\leq -U_k^-(x) \leq -\gamma(x
	\cdot\nu+\epsilon)^- \quad \text{in} \quad B_1$$
	for every $\epsilon<\epsilon_0$, then there are a constant $\tilde{\gamma}>0$ and a unit vector $\tilde{\nu}$ with $|\gamma - \tilde{\gamma}| + |\nu - \tilde{\nu}|\leq C\epsilon$ and a radius $\rho\in(0,1)$ such that
	$$-\tilde{\gamma}(x\cdot\tilde{\nu}-\epsilon/2)^-\leq -(U_k)_\rho^-(x) \leq -\tilde{\gamma}(x\cdot\tilde{\nu}+\epsilon/2)^- \quad \text{in} \quad B_1,$$
	where $(U_k)_\rho^-:=\frac{U_k^-(\rho x)}{\rho}=\frac{u^-(x^0+r_k\rho x)}{r_k\rho}$.
\end{prop}

We will give a sketch of the strategy to show the flatness decay, and omit the detailed proof here.

We use the general method of De Silva developed in \cite{S11}, which reduces the proof of flatness decay to two key ingredients: partial boundary Harnack lemma and the analysis of linearized problem. Roughly speaking, we consider a blow-up sequence $\{u_k\}$ that is $\epsilon_k$-close to $-\gamma_0(x\cdot\nu_0)^-$. Then let
$$w_k:=\frac{U_k^- - \gamma_0(x\cdot\nu_0)^-}{\gamma_0\epsilon_k}$$
be the linearizing sequence.

The first step is to show that $w_k$ converges to some H\"older function $w_\infty$, which is done by the partial boundary Harnack lemma for $u^-$ as in \cite{PSV21}. It says that if $u$ is $\epsilon$-flat to $-\gamma_0(x\cdot\nu_0)^-$ in $B_1$ for $\gamma_0>0$ and $\nu_0\in \partial B_1$, then in a smaller scale $B_{1/2}$ it is improved to $(1-c)\epsilon$-close for some $0<c<1$. Namely, if
$$-\gamma_0(x\cdot\nu_0-\epsilon)^-\leq u(x) \leq -\gamma_0(x\cdot\nu_0+\epsilon)^- \quad \text{in} \quad B_1,$$
then
$$-\gamma_0(x\cdot\nu_0-(1-c)\epsilon)^-\leq u(x) \leq -\gamma_0(x\cdot\nu_0+(1-c)\epsilon)^- \quad \text{in} \quad B_{1/2}.$$

The second step is to show that $w_\infty$ solves a PDE problem, i.e. the linearized problem:
\begin{equation*}
\begin{cases}
\Delta w_\infty = 0 \qquad\quad\quad \text{in} \quad B_1\cap\{x\cdot\nu_0<0\}, \\
\nabla w_\infty\cdot\nu_0=0 \qquad\ \text{on} \quad B_1\cap\{x\cdot\nu_0=0\}.
\end{cases}
\end{equation*}
The regularity of its solution $w_\infty$ implies the flatness decay of $\partial\{u_k<0\}$.

\begin{rem}
	We can follow the argument in \cite{PSV21}, Lemma 2.5 to check that $u$ is a viscosity solution of (\ref{eq1.1}), which is necessary in obtaining the partial boundary Harnack inequality and the linearized problem as in \cite{PSV21} and \cite{SFS14}.
\end{rem}

With the help of the flatness decay, it is straightforward to deduce the regularity of the partial free boundary $\partial\{u<0\}$. Additionally, Lemma 4.2 in \cite{PSV21} gives that $\gamma_0=\lambda$. Hence, we will obtain the following regular profile to the free boundary $\partial\{u<0\}$.

\begin{prop} \label{reg}
	Let $u$ be a variational solution of (\ref{eq1.1}) in $D$ with $\lambda>0$ and $x^0\in S^u_+$. Then $\gamma_0=\lambda_0$ and $\nu_0$ is uniquely determined by $x^0$. Furthermore, the free boundary $\partial\{u<0\}$ is a $C^{1,\alpha}$ curve in a small neighborhood of $x^0$ for some $\alpha\in(0,1)$.
\end{prop}

\subsubsection{Monotonicity formula for the positive phase}

In this subsection, we will deal with the asymptotic behavior of the positive phase $u^+$. Since the gradient of $u^+$ vanished at the stagnation point $x^0=(x_1^0,0)$, the blow-up limit possesses the singular behavior near the stagnation point $x^0$. The approach is quite different from the regular case for $u^-$, and we aim to make a blow-up analysis to study the singularity for $u^+$ at $x^0$, which relies on a monotonicity formula involving only $u^+$. The strategy borrows from the innovative ideas in \cite{VW11} to deal with the singular profile to one-phase free boundary problem.

Before delving into it, we give a key observation that once $u$ is fixed, then $u^+$ locally solves the problem
\begin{equation} \label{equation+}
\begin{cases}
\Delta V = 0 \qquad\quad\ \text{in} \quad D\cap\{u>0\}, \\
|\nabla V|^2 = -x_2 \quad \text{on} \quad D\cap\Gamma_{\rm op}^+
\end{cases}
\end{equation}
for fixed domain $D\cap\{u>0\}$, which helps extracting $u^+$ from $u$ in monotonicity formula. Noting that there is a good observation that the relationship between the variational solution to the two-phase problem and to the one-phase problem is as follows.

\begin{prop} \label{prop0}
	If u is a variational solution of (\ref{eq1.1}) in $D$, then $u^+$ is a (local) variational solution of (\ref{equation+}). That is, $u^+$ satisfies
	\begin{equation} \label{va1}
	\int_{D\cap\{u\geq0\}} \left( |\nabla u^+|^2 div\bm\phi - 2\nabla u^+ D\bm\phi \nabla u^+ -(\phi_2 + x_2 div\bm\phi )\chi_{\{u>0\}} \right)dx =0
	\end{equation}
	for any $\bm\phi=(\phi_1,\phi_2)\in W_0^{1,2}(D\cap\{u\geq0\};\mathbb{R}^2)$. Furthermore, $u^+$ satisfies the additional free boundary condition $|\nabla u^+|^2 \geq -x_2$ on $\partial\{u>0\}$.
\end{prop}

\begin{proof}
	We only prove the last part in the Proposition, since (\ref{va1}) is easily verified by taking $\bm\phi\in W_0^{1,2}(D\cap\{u\geq0\};\mathbb{R}^2)$ in Definition \ref{def1} and Definition \ref{def2}. Notice that the integral terms on $\Gamma_{\rm tp}$ vanish because $\bm\phi=0$ on it. Hence we go back to the definition of variational solutions for $u$ to prove the additional free boundary condition $|\nabla u^+|^2 \geq -x_2$ on $\partial\{u>0\}$.
	
	Denote $\nu_t$ the outer normal of $\partial\{u>t\}$ for small $t>0$, and take $\bm\phi_t=(\phi_{1,t},\phi_{2,t})\in W_0^{1,2}(D;\mathbb{R}^2)$ such that $\bm\phi_t\cdot\nu_t\leq0$ for any $t>0$. Let $u_{\epsilon,t}(y):=u^+(x) - u^-(y)$ with $y=(y_1,y_2)=x+\epsilon\bm\phi_t(x)$ for any $\epsilon>0$, which moves the positive part of $u$ inwards. That is, $\{y \ | \ u_{\epsilon,t}(y)>0\}\subset\{y \ | \ u(y)>0\}$. Hence we get
	\begin{equation*}
	\begin{aligned}
	0 &= \frac{d}{d\epsilon}\Big|_{\epsilon=0} \mathcal{J}(u(x+\epsilon\bm\phi_t(x))) \\
	&= \lim_{\epsilon\rightarrow0} \frac{\mathcal{J}(u(x+\epsilon\bm\phi_t(x))) - \mathcal{J}(u(x))}{\epsilon} \\
	&\leq \lim_{\epsilon\rightarrow0} \frac{1}{\epsilon} \Bigg[  \int_{D\cap\{u>0\}} \left( |\nabla u_{\epsilon,t}(y)|^2 -y_2  \right) dx + \int_{D\cap\{u<0\}} \left( |\nabla u|^2 + (-x_2+\lambda^2) \right) dx \\
	& \qquad + \int_{D\cap\{u_{\epsilon,t}<0<u\}} \left( |\nabla u_{\epsilon,t}(y)|^2 + (-y_2+\lambda^2) \right) dy - J(u) \Bigg] \\
	& = \int_{D\cap\{u>0\}} \left( -|\nabla u|^2 + 2\nabla u D\bm\phi_t \nabla u + div(x_2\bm\phi_t) \right) dx \\
	& = -\lim_{t\rightarrow0^+} \int_{D\cap\partial\{u>t\}} \left( |\nabla u^+|^2 + x_2 \right)\bm\phi_t\cdot\nu_t d\mathcal{S}.
	\end{aligned}
	\end{equation*}
	This yields $|\nabla u^+|^2\geq-x_2$ on $\partial\{u>0\}$ in weak sense.
\end{proof}

As in Section 2.1.1, we first define
\begin{equation}
I_+(r)=I_{+,x^0,u}(r)=r^{-3}\int_{B_r(x^0)}\left( |\nabla u^+|^2 + (-x_2)\chi_{\{u>0\}} \right)dx,
\end{equation}
and the $L^2$-boundary energy,
\begin{equation}
J_+(r)=J_{+,x^0,u}(r)=r^{-4}\int_{\partial B_r(x^0)}  (u^+)^2 d\mathcal{S},
\end{equation}
and
\begin{equation}
\begin{aligned}
M_+(r)&=M_{+,x^0,u}(r) \\
&=I_+(r)-\frac32 J_+(r) \\
&=r^{-3}\int_{B_r(x^0)}\left( |\nabla u^+|^2 + (-x_2)\chi_{\{u>0\}} \right)dx-\frac32 r^{-4}\int_{\partial B_r(x^0)}  (u^+)^2 d\mathcal{S}.
\end{aligned}
\end{equation}
This boundary adjusted energy $M_+(r)$ involves no terms of $u^-$, and we will utilize Proposition \ref{prop0} to prove that the derivative of $M_+(r)$ includes the integrand of a scalar multiple of $\left( \nabla u(x)\cdot(x-x^0)-\frac32 u(x) \right)^2$, which implies a different order of blow-up sequence from $u^-$ at $x^0$.

\begin{prop}(Monotonicity formula for $u^+$) \label{mono2}
	Let $u$ be a variational solution of (\ref{eq1.1}) in $D$ with $\lambda>0$ and $x^0\in S^u_+$.
	Then, for a.e. $r\in(0,R/2)$,
	\begin{equation*}
	\frac{d}{dr} M_+(r)=2r^{-3}\int_{\partial B_r(x^0)} \left( \nabla u^+\cdot\nu-\frac32\frac{u^+}r \right)^2 d\mathcal{S}.
	\end{equation*}
\end{prop}

\begin{proof}
	Notice that $u^+$ satisfies
	$$\int_{B_r(x^0)} \nabla u^+ \cdot \nabla(\max\{u^+-\epsilon,0\}^{1+\epsilon}) dx = \int_{\partial B_r(x^0)} \max\{u^+ - \epsilon,0\}^{1+\epsilon}\nabla u^+\cdot\nu d\mathcal{S}$$
	for any $\epsilon>0$. Take $\epsilon\rightarrow0$ and the approximation gives
	$$\int_{B_r(x^0)} |\nabla u^+|^2 dx = \int_{\partial B_r(x^0)} u^+\nabla u^+\cdot\nu d\mathcal{S}.$$
	Direct calculation of the derivative of $M_+(r)$ shows
	\begin{equation} \label{dr}
	\begin{aligned}
	\frac{d}{dr}M_+(r) &= r^{-3}\int_{\partial B_r(x^0)} \left( |\nabla u^+|^2 - x_2\chi_{\{u>0\}} \right)d\mathcal{S} - 3r^{-4}\int_{B_r(x^0)}\left( |\nabla u^+|^2 - x_2\chi_{\{u>0\}} \right)dx \\
	&\quad +\frac92 r^{-5}\int_{\partial B_r(x^0)} (u^+)^2 d\mathcal{S} -3r^{-4}\int_{\partial B_r(x^0)} u^+\nabla u^+\cdot\nu d\mathcal{S} \\
	&=r^{-3}\int_{\partial B_r(x^0)} \left( |\nabla u^+|^2 - x_2\chi_{\{u>0\}} \right)d\mathcal{S} + 3r^{-4}\int_{B_r(x^0)} x_2\chi_{\{u>0\}} dx \\
	&\quad -6r^{-4}\int_{\partial B_r(x^0)} u^+\nabla u^+\cdot\nu d\mathcal{S} + \frac92 r^{-5}\int_{\partial B_r(x^0)} (u^+)^2 d\mathcal{S}.
	\end{aligned}
	\end{equation}
	
	Now for small $\kappa$ and $\eta_\kappa(t;r):=\max\{0,\min\{1,\frac{r-t}{\kappa}\}\}$, we take after approximation $\bm\phi_\kappa(x)=\eta_\kappa(|x-x^0|;r)(x-x^0)\chi_{\{u\geq0\}}\in W_0^{1,2}(B_r(x^0)\cap\{u\geq0\};\mathbb{R}^2)$ as a test function in the definition of the variational solution $u^+$. We obtain
	\begin{equation*}
	\begin{aligned}
	0 &= \int_{D} \left( |\nabla u^+|^2 - x_2\chi_{\{u>0\}} \right) \left( 2\eta_\kappa(|x-x^0|;r)+\eta_\kappa'(|x-x^0|;r)|x-x^0| \right)dx \\
	& \quad -2\int_{D} \left( |\nabla u^+|^2 \eta_\kappa(|x-x^0|;r) + \left( \nabla u^+ \cdot \frac{x-x^0}{|x-x^0|} \right)^2 \eta_\kappa'(|x-x^0|;r)|x-x^0| \right)dx \\
	& \quad - \int_{D} x_2\eta_\kappa(|x-x^0|;r) \chi_{\{u>0\}} dx.
	\end{aligned}
	\end{equation*}
	Noticing that $\eta_\kappa$ converges pointwise to $\chi_{B_r(x^0)}$ as $\kappa\rightarrow0$ and satisfies
	$$\int_{D} \eta_\kappa'(|x-x^0|;r)f(x)\chi_{\{u\geq0\}}dx \rightarrow \int_{\partial B_r(x^0)\cap\{u\geq0\}} -f(x)d\mathcal{S}$$
	for any $f(x)\in L^2(D)$ as $\kappa\rightarrow0$, it comes to
	\begin{equation*}
	\begin{aligned}
	0&=2\int_{B_r(x^0)}\left( |\nabla u^+|^2 - x_2\chi_{\{u>0\}} \right)dx -r\int_{\partial B_r(x^0)} \left( |\nabla u^+|^2 - x_2\chi_{\{u>0\}} \right)d\mathcal{S} \\
	& \quad +2r\int_{\partial B_r(x^0)} (\nabla u^+\cdot\nu)^2 d\mathcal{S} -2\int_{B_r(x^0)}|\nabla u^+|^2 dx +\int_{B_r(x^0)}(-x_2)\chi_{\{u>0\}} dx \\
	&= 2r\int_{\partial B_r(x^0)}(\nabla u^+\cdot\nu)^2 d\mathcal{S} -r\int_{\partial B_r(x^0)} \left( |\nabla u^+|^2 - x_2\chi_{\{u>0\}} \right)d\mathcal{S} -\int_{B_r(x^0)} 3x_2\chi_{\{u>0\}} dx.
	\end{aligned}
	\end{equation*}
	Plugging this into the identity (\ref{dr}), we obtain that for a.e. $r\in(0,R/2)$,
	\begin{equation*}
	\begin{aligned}
	\frac{d}{dr}M_+(r)&=\frac{2}{r^3}\int_{\partial B_r(x^0)} (\nabla u^+\cdot\nu)^2 d\mathcal{S} -\frac6{r^4}\int_{\partial B_r(x^0)} u^+\nabla u^+\cdot\nu d\mathcal{S} +\frac9{2r^5}\int_{\partial B_r(x^0)} (u^+)^2 d\mathcal{S} \\
	&= 2r^{-3}\int_{\partial B_r(x^0)} \left( \nabla u^+\cdot\nu -\frac32\frac{u^+}r \right)^2 d\mathcal{S},
	\end{aligned}
	\end{equation*}
	which completes the proof.
\end{proof}

\subsubsection{Blow-up analysis for the positive phase}

Our analysis to $u^+$ also relies on scaling arguments. Combined with Lemma \ref{lem1}, we consider the blow-up sequence at $x^0$,
\begin{equation} \label{blowup}
u_k^+=\frac{u^+(x^0+r_k x)}{r_k^{3/2}}
\end{equation}
for $r_k\rightarrow0+$ and $x\in B_{R/r_k}(0)$ such that $x^0+r_k x\in B_R(x^0)\Subset D$.

Notice that $u_k^+$ and
$$u_k^-:=\frac{u^-(x^0+r_k x)}{r_k}$$
have disjoint support in $B_{R/r_k}(0)$, where $u_k^-=U_k^-$ in $B_{R/r_k}(0)$, satisfying the properties in Section 2.1.2.

\begin{lem} \label{lem2}
	Let $u$ be a variational solution of (\ref{eq1.1}) in $D$ with $\lambda>0$ satisfying Assumption \ref{assume} and $x^0=(x_1^0,0)\in S^u_+$. Let $u_k^+$ be a blow-up sequence of $u^+$ at $x^0$ that converges weakly in $W_{\rm{loc}}^{1,2}(\mathbb{R}^2)$ to a blow-up limit, noted as $u_0^+$. Then $u_k^+$ converges strongly to $u_0^+\geq0$ in $\mathbb{R}^2$. Moreover, $u_0^+$ is a homogeneous function of degree $\frac32$.
\end{lem}

\begin{proof}
	The convergence that $u_k^+\rightarrow u_0^+$ strongly in $W_{\rm{loc}}^{1,2}(\mathbb{R}^2)$ comes straightly from the fact that
	$$\limsup_{r_k\rightarrow0+}\int_{B_{R/r_k}(0)}|\nabla u_k^+|^2\eta dx \leq \int_{\mathbb{R}^2}|\nabla u_0^+|^2\eta dx$$
	for $\eta\in C_0^1(\mathbb{R}^2)$. In fact,
	$$\int_{B_{R/r_k}(0)}|\nabla u_k^+|^2\eta dx = \int_{B_{R/r_k}(0)} -u_k^+ \nabla u_k^+\cdot\nabla\eta dx,$$
	and the right-hand side converges to
	$$\int_{\mathbb{R}^2} -u_0^+\nabla u_0^+\cdot\nabla\eta dx = \int_{\mathbb{R}^2}|\nabla u_0^+|^2\eta dx.$$
	On the other hand, the monotonicity formula in Proposition \ref{mono2} gives that for all $0<\rho<\sigma<\infty$ and $r_k>0$,
	\begin{equation*}
	\begin{aligned}
	M_+(r_k\sigma)-M_+(r_k\rho)&=\int_{r_k\rho}^{r_k\sigma} 2 r^{-3}\int_{\partial B_r(x^0)}\left( \nabla u^+\cdot\nu-\frac32\frac{u^+}{r}  \right)^2d\mathcal{S} dr \\
	&=2\int_{\rho}^{\sigma} r^{-5} \int_{\partial B_r}\left( \nabla u_k^+\cdot x - \frac32u_k^+ \right)^2 d\mathcal{S} dr \\
	&=2\int_{B_\sigma(0)\backslash B_\rho(0)} |x|^{-5} \left( \nabla u_k^+\cdot x - \frac32 u_k^+ \right)^2 dx.
	\end{aligned}
	\end{equation*}
	The left-hand side of the above equation converges to $0$ as $r_k\rightarrow0+$, which yields the desired homogeneity of $u_0^+$.
\end{proof}

\begin{rem}
	We use the notation $u_0^+$ because it is in fact the positive part of some function $u_0$, which is defined to be the combination of the blow-up limits of $u^\pm$ at $x^0$ as follows. Denote $u_0^-$ to be the function that $u_k^-\rightarrow u_0^-:=U_0^-$ weakly in $W_{\rm{loc}}^{1,2}(\mathbb{R}^2)$ (thus strongly in $W_{\rm{loc}}^{1,2}(\mathbb{R}^2)$) as in Lemma \ref{lem1}, and take $u_0^+$ as in Lemma \ref{lem2} under subsequence, then the functions $u_0^+$ and $u_0^-$ have disjoint support. Set $u_0:=u_0^+ - u_0^-$ and then the notations are well-defined.
\end{rem}

\subsection{The weighted density}

In this section we derive structural properties of the "weighted density" $M_+(0+)$ for $u^+$, which is roughly related to the Lebesgue density $\frac{|\{u>0\}\cap B_r(x^0)|}{|B_r(x^0)|}$ as $r\rightarrow0+$. The blow-up limits at $x^0$, which leads to the asymptotic homogeneity of the solution $u$ as it approaches the stagnation point, classify the possible values of $M_+(0+)$ and obtain a geometric description of the solution at the stagnation point $x^0$.

It is remarkable that we have got the uniqueness of the blow-up limit of the negative phase $u_0^-=\lambda(x\cdot\nu_0)^-$ at $x^0$ in Section 2.2. Namely, there is a unique vector
$$\nu_0=(\nu_1^0,\nu_2^0)\in\partial B_1,$$
called the normal vector of $\partial\{u<0\}$ at $x^0$, such that $\{u_0>0\}\subset\{x\cdot\nu>0\}$, which determines the location of the open set $\{u_0>0\}$. To simplify the exposition, we denote
$$\theta_0:=\arctan\frac{\nu_1^0}{-\nu_2^0}$$
the asymptotic deflection angle of $\partial\{u<0\}$ at $x^0$, being perpendicular to $\nu_0$, which implies that $\{x=(x_1,x_2) \ | \ u_0^-(x)>0\}=\{ (r,\theta) \ | \ r>0 \ \text{and} \ \theta\in(\theta_0,\theta_0+\pi) \}$. (See Figure \ref{theta0}.)
\begin{figure}[!h] 
	\includegraphics[width=85mm]{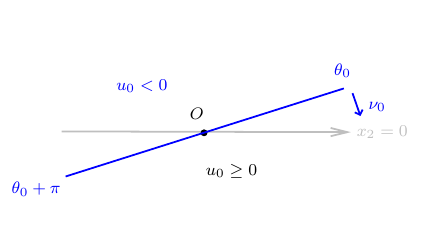}
	\caption{The asymptotic direction of $\partial\{u<0\}$.}
	\label{theta0}
\end{figure}

\begin{prop}(Weighted density for $u^+$) \label{density}
	Let $u$ be a variational solution of (\ref{eq1.1}) in $D$ with $\lambda>0$ satisfying Assumption \ref{assume}, $x^0\in S^u_+$, $\nu_0=(\nu_1^0,\nu_2^0)$ be the normal vector of $\partial\{u<0\}$ at $x^0$ and $\theta_0=\arctan\frac{\nu_1^0}{-\nu_2^0}\in[-\pi,\pi)$ is the asymptotic deflection angle of $\partial\{u<0\}$ at $x^0$. Let $u_0^+$ be the blow-up limit in Lemma \ref{lem2}. Then there exists $\Theta\in[-\pi, -\frac{2\pi}{3}]$ such that the weighted density $M_+(0+)$ has only three possible values
	$$M_+(0+)=\lim_{r_k\rightarrow0+}\int_{B_1} x_2^- \chi_{\{u_k^+>0\}} dx\in\left\{ 0, \ -\frac{\sqrt{3}}{2}\sin\left(\Theta+\frac{\pi}{3}\right), \ \frac{\cos\theta_0+1}{2} \right\}.$$
	More precisely, the weighted density can be classified as follows.
	
	(i) (Nontrivial blow-up limit) If
	\begin{equation*}
	u_k^+(x)=\frac{u^+(x^0+rx)}{r^{3/2}}\rightarrow u_0^+(x)=
	\begin{cases}
	C \rho^{3/2} \cos \left( \frac32\theta-\frac32\Theta-\frac{\pi}{2} \right), \quad \Theta<\theta<\Theta+\frac{2\pi}{3}, \\
	0, \qquad\qquad\qquad\qquad\qquad\quad \text{otherwise,}
	\end{cases}
	\end{equation*}
	for $C=C(\Theta)>0$ as $r\rightarrow0+$ strongly in $W_{\rm{loc}}^{1,2}(\mathbb{R}^2)$ and locally uniformly in $\mathbb{R}^2$, where $x=(\rho\cos\theta, \rho\sin\theta)$, then
	$$M_+(0+)= -\frac{\sqrt{3}}{2}\sin\left(\Theta+\frac{\pi}{3}\right).$$ Moreover, it can be divided into three subcases by the value of $\Theta$, which has only three possible values $\left\{ -\frac{5\pi}{6}, \ \theta_0-\pi, \ \theta_0-\frac{2\pi}{3} \right\}$.
	
	(i-1) (Detached case) In this subcase, $\Theta=-\frac{5\pi}{6}$ and the blow-up limit
	\begin{equation*}
	u_0^+(x)=u_0^+(\rho,\theta)=
	\begin{cases}
	\frac{\sqrt{2}}3 \rho^{3/2} \cos \left( \frac32\theta+\frac{3\pi}{4} \right), \quad -\frac{5\pi}{6}<\theta<-\frac{\pi}{6}, \\
	0, \qquad\qquad\qquad\qquad\quad\ \text{otherwise,}
	\end{cases}
	\end{equation*}
	and the weighted density
	$$M_+(0+)=\lim_{r\rightarrow0+}\int_{B_1}x_2^-\chi_{ \{ x:-\frac{5\pi}{6}<\theta<-\frac{\pi}{6} \} } dx =\frac{\sqrt{3}}2.$$
	
	(i-2) (Left-overlapping case) In this subcase, $\Theta=\theta_0-\pi$ and the blow-up limit
	\begin{equation*}
	u_0^+(x)=u_0^+(\rho,\theta)=
	\begin{cases}
	\frac23\sqrt{-\sin\left(\theta_0-\frac{\pi}{3}\right)} \rho^{3/2} \cos \left( \frac32\theta-\frac32\theta_0+\pi \right), \quad \theta_0-\pi<\theta<\theta_0-\frac{\pi}{3}, \\
	0, \qquad\qquad\qquad\qquad\qquad\qquad\qquad\qquad\quad\ \ \text{otherwise,}
	\end{cases}
	\end{equation*}
	and the weighted density
	$$M_+(0+)=\lim_{r\rightarrow0+}\int_{B_1}x_2^-\chi_{ \{ x:\theta_0-\pi<\theta<\theta_0-\pi/3 \} } dx =-\frac{\sqrt{3}}2\sin(\theta_0-\frac{2\pi}{3}).$$
	
	(i-3) (Right-overlapping case) In this subcase, $\Theta=\theta_0-\frac{2\pi}{3}$ and the blow-up limit
	\begin{equation*}
	u_0^+(x)=u_0^+(\rho,\theta)=
	\begin{cases}
	\frac23 \sqrt{-\sin\left( \theta_0-\frac{2\pi}{3} \right)} \rho^{3/2} \cos \left( \frac32\theta-\frac32\theta_0+\frac{\pi}{2} \right), \quad \theta_0-\frac{2\pi}{3}<\theta<\theta_0, \\
	0, \qquad\qquad\qquad\qquad\qquad\qquad\qquad\qquad\qquad \text{otherwise,}
	\end{cases}
	\end{equation*}
	and the weighted density
	$$M_+(0+)=\lim_{r\rightarrow0+}\int_{B_1}x_2^-\chi_{ \{ x:\theta_0-2\pi/3<\theta<\theta_0 \} } dx =-\frac{\sqrt{3}}2\sin(\theta_0-\frac{\pi}{3}).$$
	
	(ii) (Trivial blow-up limit) If
	\begin{equation*}
	u_k^+(x)=\frac{u^+(x^0+rx)}{r^{3/2}}\rightarrow u_0^+(x)\equiv0
	\end{equation*}
	as $r\rightarrow0+$ strongly in $W_{\rm{loc}}^{1,2}(\mathbb{R}^2)$ and locally uniformly in $\mathbb{R}^2$, then $M_+(0+)\in\left\{ 0, \ \frac{\cos\theta_0+1}{2} \right\}$ has only two possible values.
\end{prop}

\begin{rem}
	The statement (i) in Proposition \ref{density} says that $\Theta$, the starting angle of the cone $\{x=(x_1,x_2) \ | \ u_0^+(x)>0\}=\{(r,\theta) \ | \ r>0 \ \text{and} \ \theta\in(\Theta,\Theta+2\pi/3)\}$, only possesses three possible values $-\frac{5\pi}{6}$, $\theta_0-\pi$ and $\theta_0-\frac{2\pi}{3}$, corresponding to the cases in Figure \ref{theta1} below, and the range of $\Theta\in[-\pi,-\frac{2\pi}{3}]$ implies the range of $\nu_0$ in turn. Namely, $\nu_0$ is uniquely determined such that $\{x=(x_1,x_2) \ | \ u_0^+(x)>0\}\cap\{ (r,\theta) \ | \ r>0 \ \text{and} \ \theta\in(\theta_0,\theta_0+\pi) \}=\varnothing$ for $\theta_0 = \arctan\frac{\nu_1^0}{\nu_2^0}$.
	\begin{figure}[!h] 
		\includegraphics[width=155mm]{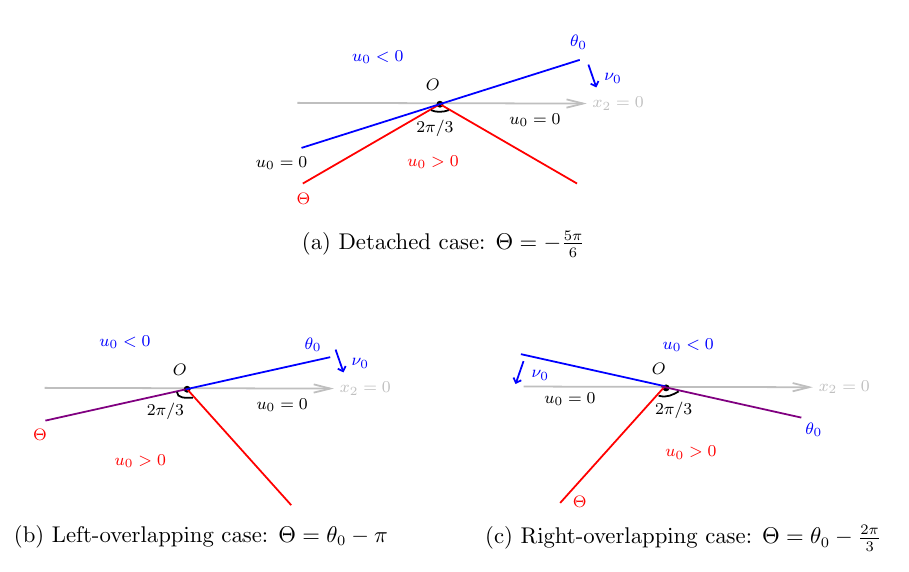}
		\caption{Three corresponding cases for $\Theta$ in case (i).}
		\label{theta1}
	\end{figure}
\end{rem}

\begin{rem}
	The asymptotic singular profile of the free boundary in one-phase water wave established in \cite{VW11} presents a symmetric feature in the $x_2$-direction at the stagnation point, which is coincide with the direction of gravity. However, in our two-phase situation, the two phases may have interactions on each other, and the singular profile of $\partial\{u>0\}$ is not necessarily symmetric. In the detached case when $\Theta=-\frac{5\pi}{6}$ and the deflection angle $\theta_0\in(-\frac{\pi}{6},\frac{\pi}{6})$, the component cone $\{u_0>0\}$ is symmetric about the $x_2$-axis. In the two overlapping cases when $\Theta=\theta_0-\pi$ or $\Theta=\theta_0-\frac{2\pi}{3}$ and the deflection angle $\theta_0\in[-\frac{\pi}{3},\frac{\pi}{3}]$, the component cone $\{u_0>0\}$ is not necessarily symmetric.
\end{rem}

\begin{rem}
	The following Table \ref{table2} compares the "weighted density" at the stagnation point $x^0$ in one-phase problem and two-phase problem. We simply write "Non-deg. stag. pt." to represent "Non-degenerate stagnation point", and "Deg. stag. pt." to represent "Degenerate stagnation point".
	\begin{table}[h]
		\caption{"Weighted densities" in one-phase and two-phase problems}\label{table2}
		\centering
		\begin{tabular}{|c|c|c|c|c|}
			\hline
			\multicolumn{2}{|c|}{ }& Blow-up limit $u_0(x)$ & \multicolumn{2}{c|}{Weighted density $M(0+)$} \\
			\hline
			\multirow{3}{*}{\makecell{One-phase \\ $|\nabla u(x^0)|=0$}} & \makecell{Non-deg. \\ stag. pt.} & Trigonometric function & \multicolumn{2}{c|}{$\sqrt{3}/2$} \\
			\cline{2-5}
			& \multirow{2}{*}{ \makecell{ Deg. \\ stag. pt.} } &\multirow{2}{*}{$u_0\equiv0$} & Trivial density & $0$ \\
			\cline{4-5}
			& & & Nontrivial density & $1$ \\
			\hline
			\multicolumn{2}{|c|}{}& Blow-up limit $u_0^+(x)$ & \multicolumn{2}{c|}{Weighted density $M_+(0+)$} \\
			\hline
			\multirow{3}{*}{\makecell{Two-phase \\ $|\nabla u^+(x^0)|=0$}} & \makecell{Non-deg. \\ stag. pt.} & Trigonometric function & \multicolumn{2}{c|}{\makecell{ {$-\frac{\sqrt{3}}{2}\sin\left(\Theta+\frac{\pi}{3}\right)$} \\ ($\Theta$ depends on $\theta_0$, $3$ subcases)} } \\
			\cline{2-5}
			& \multirow{2}{*}{ \makecell{ Deg. \\ stag. pt.} } & \multirow{2}{*}{$u_0^+\equiv0$} & Trivial density & $0$ \\
			\cline{4-5}
			& & & Nontrivial density & $(\cos\theta_0+1)/2$ \\
			\hline
		\end{tabular}
	\end{table}
\end{rem}

\begin{proof}[Proof of Proposition \ref{density}]
	Consider a blow-up sequence
	$$u_k(x)=u_k^+(x)-u_k^-(x)=\frac{u^+(x^0+r_kx)}{r_k^{3/2}}-\frac{u^-(x^0+r_kx)}{r_k}$$
	where $r_k\rightarrow0+$, with the blow-up limit $u_0=u_0^+ - u_0^-$. Using Lemma \ref{lem2} and the homogeneity of $u_0^+$, we obtain that
	\begin{equation*}
	\begin{aligned}
	\lim_{r_k\rightarrow0+} M_+(r_k) &= \int_{B_1}|\nabla u_0^+|^2 dx - \frac32\int_{\partial B_1} (u_0^+)^2 d\mathcal{S} + \lim_{r_k\rightarrow0+} r_k^{-3}\int_{B_r(x^0)} x_2^- \chi_{\{u>0\}} dx \\
	&=\lim_{r_k\rightarrow0+}\int_{B_1} x_2^- \chi_{\{u_k^+>0\}} dx.
	\end{aligned}
	\end{equation*}
	Now fix any small $\delta>0$. Take $\bm\phi=(\phi_1,\phi_2)\in W_0^{1,2}(\mathbb{R}^2\cap\{x\cdot\nu_0\geq\delta\};\mathbb{R}^2)$ and $\bm\phi_k(y):=\bm\phi(\frac{y-x^0}{r_k})=\bm\phi(x)$ for $y=x^0+r_kx$. Then the flatness of the free boundary of $u_0^-$ implies that $\{x\cdot\nu_0\geq\delta\}\subset\{x \ |\ u_k^-(x)=0\}$ for sufficiently small $r_k$ depending on $\delta$, hence $\bm\phi_k(y)\in W_0^{1,2}(\mathbb{R}^2\cap\{u(y)\geq0\};\mathbb{R}^2)$.
	
	We use $\bm\phi_k=(\phi_{k,1},\phi_{k,2})$ as the test function in the definition of the variational solution $u$, and obtain
	\begin{equation*}
	\begin{aligned}
	0&=\int_{B_r(x^0)} \left( |\nabla u^+|^2 div\bm \phi_k - 2\nabla u^+ D\bm\phi_k\nabla u^+ - (\phi_{k,2}+y_2 div\bm \phi_k)\chi_{\{u>0\}} \right) dy \\
	&= r_k^2 \int_{B_{r/r_k}} \left( |\nabla u_k^+|^2 div\bm\phi - 2\nabla u_k^+ D\bm\phi \nabla u_k^+ - (\phi_2+x_2 div\bm\phi)\chi_{\{u_k^+>0\}} \right) dx.
	\end{aligned}
	\end{equation*}
	Passing the limit $r_k\rightarrow0+$, the strong convergence of $u_k^+$ to $u_0^+$ in $W_{\rm{loc}}^{1,2}(\mathbb{R}^2)$ and the compact embedding from $BV$ to $L^1$ imply that
	\begin{equation} \label{eq1}
	0=\int_{\mathbb{R}^2} \left( |\nabla u_0^+|^2 div\bm\phi -2\nabla u_0^+ D\bm\phi \nabla u_0^+ - (\phi_2+x_2 div\bm\phi)\chi_0^+ \right) dx,
	\end{equation}
	where $\chi_0^+$ is the strong $L_{\rm{loc}}^1(\mathbb{R}^2\cap\{x_2<0\})$ limit of $\chi_{\{u_k^+>0\}}$ along a subsequence. Notice that the values of the function $\chi_0^+$ are almost everywhere in $\{0,1\}$, and the locally uniform convergence of $u_k^+$ to $u_0^+$ implies that $\chi_0^+=1$ in $\{u_0^+>0\}$.
	
	Moreover, the $\frac32$-homogeneity of $u_0^+$ and its harmonicity in $\{u_0>0\}$ show that each connected component of $\{u_0>0\}$ is a cone with vertex at the origin and of opening angle $\frac{2\pi}{3}$. Since $u\leq0$ in $\{x_2\geq0\}$, this shows that $\{u_0>0\}$ has at most one connected component. Note also that (\ref{eq1}) implies that $\chi_0^+$ is a constant in each open connected set $G\subset\{u_0\leq0\}^{\circ}\cap\{x\cdot\nu_0\geq\delta\}$, denoted as $I_0^+$, which does not intersect $\{x_2=0\}$. Here $\{u_0\leq0\}^{\circ}$ means the interior of the set $\{u_0\leq0\}$.
	
	Consider the first case when $\{u_0>0\}$ is non-empty, and is therefore a cone as described above. Let $z\in\partial\{u_0>0\}\backslash\{0\}$ be an arbitrary point. Note that the normal to $\partial\{u_0>0\}$ has the constant value vector $\bm n$ in $B_{\tau}(z)\cap\{u_0>0\}$ for some $\tau>0$. See Figure \ref{2-1}.
	
	\begin{figure}[!h] 
		\includegraphics[width=70mm]{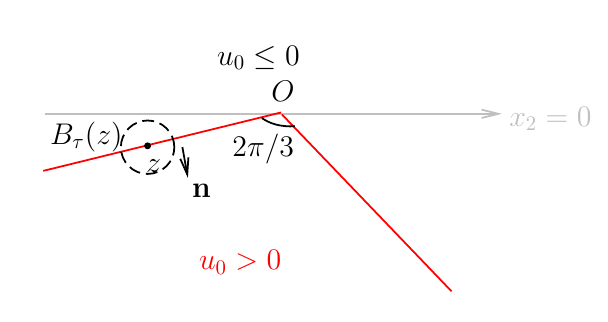}
		\caption{The ball $B_\tau(z)$ and the vector $\bm n$}
		\label{2-1}
	\end{figure}
	
	Plugging in $\bm\phi(x):=\eta(x)\bm n$ into (\ref{eq1}), where $\eta\in W_0^{1,2}(B_{\tau}(z)\cap\{x\cdot\nu_0\geq\delta\})$ is arbitrary. Integrating by parts, it follows that
	\begin{equation} \label{eq2}
	\begin{aligned}
	0&=\int_{B_{\tau}(z)\cap\{u_0>0\}\cap\{x\cdot\nu_0\geq\delta\}} \left( \nabla\cdot\left( |\nabla u_0^+|^2\bm\phi - 2(\bm\phi\cdot\nabla u_0^+)\nabla u_0^+ \right) - div(x_2\bm\phi) \right) dx \\
	&\quad - \int_{B_{\tau}(z)\cap\{u_0\leq0\}\cap\{x\cdot\nu_0\geq\delta\}} div(x_2\bm\phi) I_0^+ dx \\
	&=\int_{B_{\tau}(z)\cap\partial\{u_0>0\}\cap\{x\cdot\nu_0\geq\delta\}} \left( (|\nabla u_0^+|^2\bm\phi - 2(\bm\phi\cdot\nabla u_0^+)\nabla u_0^+)\bm n - (x_2\bm\phi)\bm n  \right) d\mathcal{S} \\
	&\quad + \int_{B_{\tau}(z)\cap\partial\{u_0>0\}\cap\{x\cdot\nu_0\geq\delta\}} x_2\bm\phi I_0^+\bm n d\mathcal{S} \\
	&=-\int_{B_{\tau}(z)\cap\partial\{u_0>0\}\cap\{x\cdot\nu_0\geq\delta\}} \left( |\nabla u_0^+|^2 +x_2(1-I_0^+) \right)\eta d\mathcal{S}.
	\end{aligned}
	\end{equation}
	Here $I_0^+$ denotes the constant value of $\chi_0^+$ in the respective connected component of $\{u_0\leq0\}^{\circ}\cap\{x_2\neq0\}\cap\{x\cdot\nu_0\geq\delta\}$. Note that thanks to Hopf's lemma, one has $\nabla u_0^+\cdot\bm n\neq0$ on $B_{\tau}(z)\cap\partial\{u_0>0\}$. It follows therefore that $I_0^+\neq1$, and necessarily $I_0^+=0$. We deduce from (\ref{eq2}) that
	$$|\nabla u_0^+|^2=-x_2 \quad \text{on} \quad B_{\tau}(z)\cap\partial\{u_0>0\}\cap\{x\cdot\nu_0\geq\delta\}.$$
	Write $u_0^+(\rho,\theta)=C\rho^{3/2}\cos(\frac32\theta+\zeta_0)$ and assume that the cone $\{u_0>0\}$ lies between $\theta=\Theta$ and $\theta=\Theta+\frac{2\pi}{3}$, we obtain
	\begin{equation} \label{eq5}
	\frac94 C^2 = -\sin\theta \quad \text{on} \quad B_{\tau}(z)\cap\partial\{u_0>0\}\cap\{x\cdot\nu_0\geq\delta\}.
	\end{equation}
	
	Note that $\partial\{u_0>0\}$ may intersect with $\{x\cdot\nu_0=0\}$. There are two cases distinguished by the position of $\partial\{u_0>0\}$, the detached case and the overlapping case, since $u_0^+$ and $u_0^-$ have disjoint support and $u_0^-(x)=\lambda(x\cdot\nu_0)^-$. In the detached case, the condition (\ref{eq5}) is satisfied on both two sides $\partial\{u_0>0\}$ of the cone $\{u_0>0\}$, while in the overlapping case, the condition (\ref{eq5}) is satisfied only on one side of the cone $\{u_0>0\}$ (on the left side in the left-overlapping case, and on the right side in the right-overlapping case).
	
	\emph{Subcase 1. (Detached case)} $\partial\{u_0^+>0\}\cap\partial\{u_0^->0\}=\{0\}$, namely, $x\cdot\nu_0\neq0$ for any $x\in\partial\{u_0^+>0\}\backslash\{0\}$.
	
	Due to the arbitrariness of the point $z$, we have
	\begin{equation*}
	\begin{cases}
	\frac94 C^2 = -\sin\theta \quad \text{on} \quad \{\theta=\Theta\}\cap\{x\cdot\nu_0\geq\delta\}, \\
	\frac94 C^2 = -\sin\theta \quad \text{on} \quad \{\theta=\Theta+\frac{2\pi}{3}\}\cap\{x\cdot\nu_0\geq\delta\}.
	\end{cases}
	\end{equation*}
	Hence, $\sin\Theta=\sin(\Theta+\frac{2\pi}{3})$, which determines the unique value $\Theta=-\frac{5\pi}{6}$. Then it is straight forward to compute that $C=\frac{\sqrt{2}}{3}$. Note that in this case, the disjoint open sets $\{u_0^\pm>0\}$ together with the location of $\{u_0^+>0\}$ imply that $\theta_0\in(-\frac{\pi}{6},\frac{\pi}{6})$. See Figure \ref{corner1-1}.
	
	\begin{figure}[!h] 
		\includegraphics[width=90mm]{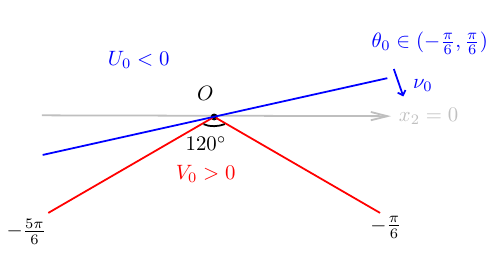}
		\caption{The detached case when $\lambda>0$.}
		\label{corner1-1}
	\end{figure}
	
	\emph{Subcase 2. (Overlapping case)} $\left( \partial\{u_0^+>0\}\cap\partial\{u_0^->0\} \right) \backslash \{0\} \neq\varnothing$, namely, $x\cdot\nu_0=0$ if $\theta=\Theta$ or $\theta=\Theta+\frac{2\pi}{3}$.
	
	In this case denote $\theta_0=\arctan\frac{\nu_1^0}{-\nu_2^0}\in[0,2\pi)$ for $\nu_0=(\nu_1^0,\nu_2^0)$. Then $\Theta=\theta_0-\pi$ in the left-overlapping case or $\Theta=\theta_0-\frac{2\pi}{3}$ in the right-overlapping case. Moreover, there holds either
	\begin{equation*}
	\frac94 C^2 = -\sin\theta \quad \text{on} \quad \{\theta=\Theta+2\pi/3\}\cap\{x\cdot\nu_0\geq\delta\}.
	\end{equation*}
	in the left-overlapping case or
	\begin{equation*}
	\frac94 C^2 = -\sin\theta \quad \text{on} \quad \{\theta=\Theta\}\cap\{x\cdot\nu_0\geq\delta\}
	\end{equation*}
	in the right overlapping case. Hence,
	$$u_0^+=\frac23 C \rho^{3/2} \cos \left( \frac32\theta-\frac32\Theta-\frac{\pi}{2} \right) \quad \text{in} \quad \{u_0>0\},$$
	where $C=\sqrt{-sin\left(\Theta+\frac{2\pi}3\right)}$ or $C=\sqrt{-sin\Theta}$, namely, $C=\sqrt{-\sin\left(\theta_0-\frac{\pi}{3}\right)}$ or $C=\sqrt{-\sin\left( \theta_0-\frac{2\pi}{3} \right)}$. Note that the angle $2\pi/3$ of the cone $\{u_0^+>0\}$ in $\{x_2<0\}$ implies that $\Theta\in[-\pi,-\frac{2\pi}{3}]$, which in turn gives that $\theta_0\in[-\frac{\pi}{3},\frac{\pi}{3}]$ since $\{u_0^+>0\}$ and $\{u_0^->0\}$ share common boundary on a half-line locally.
	
	\begin{figure}[!h]
		\includegraphics[width=160mm]{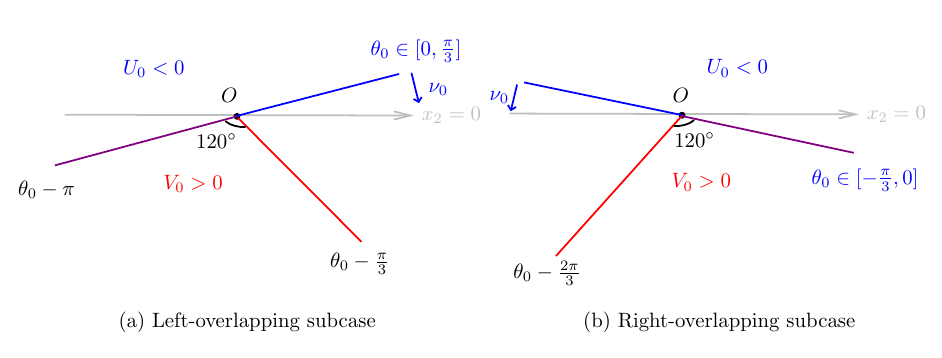}
		\caption{The overlapping case when $\lambda>0$.}
		\label{corner1-2}
	\end{figure}
	
	Notice that we only have $\chi_0^+=0$ in $\{u_0\leq0\}^{\circ}\cap\{x_2\neq0\}\cap\{x\cdot\nu_0\geq\delta\}$. Now take $\delta\rightarrow0+$, the strong convergence of $\chi_{\{u_k^+>0\}}$ to $\chi_0^+$ in $L_{\rm{loc}}^1(\mathbb{R}^2\cap\{x_2<0\})$ gives that $\int_{\{u_0\leq0\}\cap\{x\cdot\nu_0>0\}}\chi_0^+ dx=0$. Hence we can deduce that $M_+(0+)=\int_{B_1}x_2^- \chi_{\{x:\Theta<\theta<\Theta+\frac{2\pi}{3}\}} dx$, which can be computed explicitly depending only on $\theta_0$. 
	See Figure \ref{corner1-2}.
	
	Consider now the case $u_0^+\equiv0$. It follows from (\ref{eq1}) that
	$$0=-\int_{\mathbb{R}^2}\left( \phi_2 + x_2div\bm\phi \right)\chi_0^+ dx = \int_{\mathbb{R}^2} x_2\bm\phi\cdot\nabla \chi_0^+ dx$$
	for $\bm\phi(x)\in W_0^{1,2}(\mathbb{R}^2\cap\{x\cdot\nu_0\geq\delta\};\mathbb{R}^2)$, which yields that $\chi_0^+$ is a constant in $\{x_2<0\}\cap\{x\cdot\nu_0\geq\delta\}$. Again take $\delta\rightarrow0+$ and the strong convergence of $\chi_{\{u_k^+>0\}}$ to $\chi_0^+$ in $L_{\rm{loc}}^1$ gives that $\int_{\{x_2<0\}\cap\{x\cdot\nu_0>0\}}\chi_0^+ dx$ is a constant. Its value may be either $0$ in which case $M_+(0+)=0$, or $1$ in which case $M_+(0+)=\int_{B_1}x_2^- \chi_{\{x\cdot\nu_0>0\}} dx=\frac{\cos\theta_0+1}{2}$.
	
\end{proof}

\begin{rem}
	In this situation, the range of the unique normal $\nu_0$ and the deflection angle of the negative-phase free boundary $\theta_0$ are not determined in Lemma \ref{lem1} until the case $\{u_0>0\}\neq\varnothing$, since $u_0^+$ and $u_0^-$ have disjoint support. Moreover, the subcase $\partial\{u_0^+>0\}\cap\partial\{u_0^->0\}=\{0\}$ implies that $\theta_0\in(-\frac{\pi}{6},\frac{\pi}{6})$, and the subcase $\left( \partial\{u_0^+>0\}\cap\partial\{u_0^->0\} \right)\backslash \{0\} \neq\varnothing$ implies that $\theta_0\in[-\frac{\pi}{3},\frac{\pi}{3}]$. For the case $\{u_0>0\}=\varnothing$, the range of $\theta_0$ cannot be determined yet, which will be discussed later.
\end{rem}

\begin{rem}
	When $\{u_0>0\}\neq\varnothing$, the detached case and the overlapping case are divided by the position of $\partial\{u_0>0\}$ that whether it coincides with $\{x
	\cdot\nu_0=0\}$ or not. Namely, if $\partial\{u_0>0\}$ and $\partial\{u_0<0\}$ contact only at the point $x^0$, then it falls into the detached case. If $\partial\{u_0>0\}$ and $\partial\{u_0<0\}$ are overlapped in a segment, then it belongs to the overlapping case. Nevertheless, it is noticeable that in Theorem A-1 and Theorem A-2, these two cases are distinguished by the position of $\partial\{u>0\}$ and $\partial\{u<0\}$. In fact, we have already known the local Hausdorff convergence from $\partial\{u<0\}$ to $\partial\{u_0<0\}$ by \cite{V23}, Section 6, while $u_k^+$ possesses some non-degeneracy since $\{u_0>0\}\neq\varnothing$, namely, there is a strictly increasing function
	$$\omega: [0,+\infty)\rightarrow[0,+\infty),$$
	such that $\omega(0)=0$ and
	$$\Vert u_k^+ \Vert_{L^\infty(B_s(Y^0))}\geq\omega(s) \quad 
	\text{for every} \quad Y^0\in\overline{\{u_k^+>0\}}\cap B_{s_0}\quad \text{and} \quad s\in(0,s_0/2) $$
	for some constant $s_0>0$. The argument in Section 6.2, \cite{V23} gives the Hausdorff convergence of the free boundary $\partial\{u>0\}$ to $\partial\{u_0>0\}$. Hence the two cases can be classified into the detached case when $\left( \partial\{u>0\}\cap\partial\{u<0\}\cap B_r(x^0) \right)\backslash\{x^0\}=\varnothing$ for any small $r$, and the overlapping case when $\left( \partial\{u>0\}\cap\partial\{u<0\}\cap B_r(x^0) \right)\backslash\{x^0\}\neq\varnothing$ for any small $r$.
\end{rem}

\begin{rem} \label{bv}
	For any weak solutions of the problem (\ref{eq1.1}) in $D$, if we assume the Bernstein estimate that $|\nabla u^+ (x)|^2 \leq C x_2^-$ locally in $D\cap\{x_2<0\}$, then $\chi_{\{u>0\}}$ is locally a BV function in $D\cap\{x_2<0\}$. Moreover, the total variation measure $|\nabla \chi_{\{u>0\}}|$ satisfies
	$$\int_{B_r(x^0)} \sqrt{-x_2} d|\nabla \chi_{\{u>0\}}|\leq C_0 r^{3/2}$$
	for any $x^0\in\partial\{u>0\}\cap D$ and $r$ small enough. Indeed, the Bernstein estimate gives
	\begin{equation*}
	C_0 r^{3/2} \geq \int_{\partial B_r(x^0)} \nabla u^+ \cdot\nu d\mathcal{S} = \Delta u^+(B_r(x^0)) \geq \int_{B_r(x^0)\cap\partial_{red}\{u>0\}} \sqrt{-x_2} d\mathcal{S},
	\end{equation*}
	as required, where $\Delta u^+(B_r(x^0))$ denotes the Radon measure $\Delta u^+$ on $B_r(x^0)$
\end{rem}

With Remark \ref{bv} at hand, we can get the following proposition, which gives a description of the asymptotic singular profile of the free boundary $\partial\{u>0\}$ as it approaches $x^0$ in each case of density.

\begin{prop}\label{curve}
	Let $u$ be a weak solution of (\ref{eq1.1}) in $D$ with $\lambda>0$ satisfying Assumption \ref{assume}, $x^0\in S^u_+$, $\nu_0=(\nu_1^0,\nu_2^0)$ be the normal vector of $\partial\{u<0\}$ at $x^0$, $\theta_0$ be the asymptotic deflection angle of $\partial\{u<0\}$ at $x^0$. Suppose in addition that $\partial\{u>0\}$ is in a neighborhood of $x^0$ an injective curve $\sigma$: $(-1,1)\rightarrow\mathbb{R}^2$ such that $\sigma(t)=(\sigma_1(t),\sigma_2(t))$ and $\sigma(0)=x^0$. Then the following hold:
	
	(i-a) (Detached case) If $\Theta=-\frac{5\pi}{6}$ and $M_+(0+)=\frac{\sqrt{3}}2$ , then (see Figure \ref{corner2-1}) $\sigma_1(t)\neq x_1^0$ in $(-t_0,t_0)\backslash\{0\}$ for some $0<t_0<1$ and, depending on the parametrization, either
	$$\lim_{t\rightarrow0+}\frac{\sigma_2(t)}{\sigma_1(t)-x_1^0}=\frac{1}{\sqrt{3}} \quad \text{and} \quad \lim_{t\rightarrow0-}\frac{\sigma_2(t)}{\sigma_1(t)-x_1^0}=-\frac{1}{\sqrt{3}},$$
	or
	$$\lim_{t\rightarrow0+}\frac{\sigma_2(t)}{\sigma_1(t)-x_1^0}=-\frac{1}{\sqrt{3}} \quad \text{and} \quad \lim_{t\rightarrow0-}\frac{\sigma_2(t)}{\sigma_1(t)-x_1^0}=\frac{1}{\sqrt{3}}.$$
	
	\begin{figure}[!h] 
		\includegraphics[width=90mm]{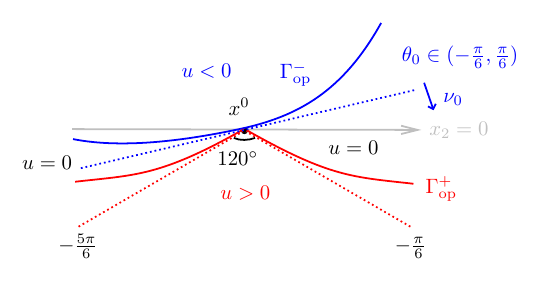}
		\caption{Stokes corner in detached case when $\lambda>0$.}
		\label{corner2-1}
	\end{figure}
	
	(i-b) (Left-overlapping case) If $\Theta=\theta_0-\pi$ and $M_+(0+)=-\frac{\sqrt{3}}2\sin(\theta_0-\frac{2\pi}{3})$, then (see Figure \ref{corner2-2} (a)) $\sigma_1(t)\neq x_1^0$ in $(-t_0,t_0)\backslash\{0\}$ for some $0<t_0<1$ and, depending on the parametrization, either
	$$\lim_{t\rightarrow0+}\frac{\sigma_2(t)}{\sigma_1(t)-x_1^0}=\tan\theta_0 \quad \text{and} \quad \lim_{t\rightarrow0-}\frac{\sigma_2(t)}{\sigma_1(t)-x_1^0}=\tan\left(\theta_0-\frac{\pi}{3}\right),$$
	or
	$$\lim_{t\rightarrow0+}\frac{\sigma_2(t)}{\sigma_1(t)-x_1^0}=\tan\left(\theta_0-\frac{\pi}{3}\right) \quad \text{and} \quad \lim_{t\rightarrow0-}\frac{\sigma_2(t)}{\sigma_1(t)-x_1^0}=\tan\theta_0.$$
	
	(i-c) (Right-overlapping case) If $\Theta=\theta_0-\frac{2\pi}{3}$ and $M_+(0+)=-\frac{\sqrt{3}}2\sin(\theta_0-\frac{\pi}{3})$, then (see Figure \ref{corner2-2} (b)) $\sigma_1(t)\neq x_1^0$ in $(-t_0,t_0)\backslash\{0\}$ for some $0<t_0<1$ and, depending on the parametrization, either
	$$\lim_{t\rightarrow0+}\frac{\sigma_2(t)}{\sigma_1(t)-x_1^0}=\tan(\theta_0-\frac{2\pi}{3}) \quad \text{and} \quad \lim_{t\rightarrow0-}\frac{\sigma_2(t)}{\sigma_1(t)-x_1^0}=\tan\theta_0,$$
	or
	$$\lim_{t\rightarrow0+}\frac{\sigma_2(t)}{\sigma_1(t)-x_1^0}=\tan\theta_0 \quad \text{and} \quad \lim_{t\rightarrow0-}\frac{\sigma_2(t)}{\sigma_1(t)-x_1^0}=\tan(\theta_0-\frac{2\pi}{3}).$$
	
	\begin{figure}[!h] 
		\includegraphics[width=160mm]{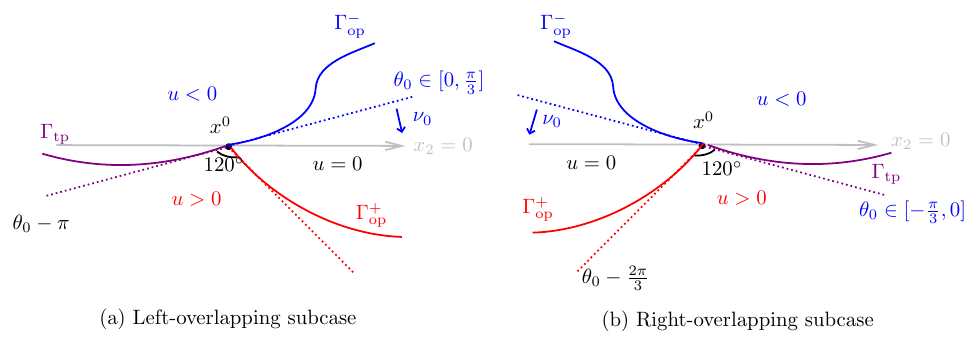}
		\caption{Stokes-type corner in overlapping case when $\lambda>0$.}
		\label{corner2-2}
	\end{figure}
	
	(ii) If $M_+(0+)=\frac{\cos\theta_0+1}{2}$, then (see Figure \ref{horizontal}) $\nu_0=(0,-1)$, $\theta_0=0$ and in particular, $M_+(0+)=1$. Moreover, $\sigma_1(t)\neq x_1^0$ in $(-t_0,t_0)\backslash\{0\}$ for some $0<t_0<1$, $\sigma_1-x_1^0$ changes sign at $t=0$ and
	$$\lim_{t\rightarrow0}\frac{\sigma_2(t)}{\sigma_1(t)-x_1^0}=0.$$
	
	\begin{figure}[!h] 
		\includegraphics[width=95mm]{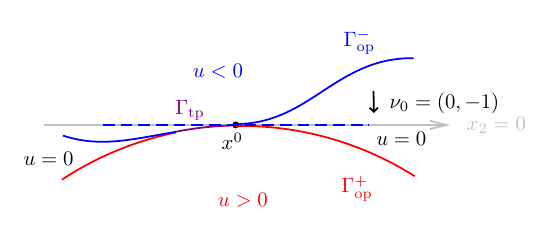}
		\caption{Horizontally flat when $\lambda>0$.}
		\label{horizontal}
	\end{figure}
	
	(iii) If $M_+(0+)=0$, then (see Figure \ref{cusp}) $\sigma_1(t)\neq x_1^0$ in $(-t_0,t_0)\backslash\{0\}$ for some $0<t_0<1$, $\sigma_1-x_1^0$ does not change sign at $t=0$, and
	$$\lim_{t\rightarrow0}\frac{\sigma_2(t)}{\sigma_1(t)-x_1^0}=0.$$
	
	\begin{figure}[!h] 
		\includegraphics[width=160mm]{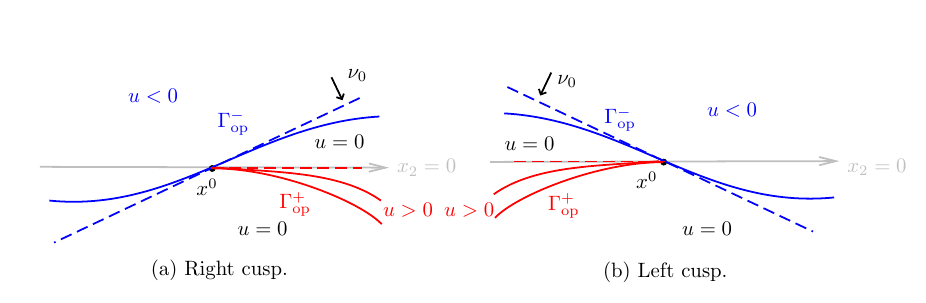}
		\caption{Cusp when $\lambda>0$.}
		\label{cusp}
	\end{figure}
\end{prop}

\begin{proof}
	We may assume that $x_1^0=0$. Define $\mathcal{L}_\pm=\{\theta^*\in[0,\pi]:\exists t_m\rightarrow0^\pm \ \ \text{s.t.} \ \  \arg\sigma(t_m)\rightarrow\theta^* \ \ \text{as}\ \ m\rightarrow\infty\}$, where $\arg y$ denotes the complex argument for $y\in\mathbb{R}^2$.
	
	\emph{Step 1.} We first prove that both $\mathcal{L}_\pm$ are subsets of $\left\{-\pi,\Theta,\Theta+\frac{2\pi}{3},0 \right\}$, where $\Theta\in\left\{-\frac{5\pi}{6}, \theta_0-\pi, \theta_0-\frac{2\pi}{3} \right\}$.
	
	Indeed, suppose towards a contradiction that there exists a sequence $0\neq t_k\rightarrow0+$ as $k\rightarrow\infty$ such that $\arg \sigma(t_k)\rightarrow\theta^*\in(\mathcal{L}_+\cup\mathcal{L}_-)\backslash\{-\pi,\Theta,\Theta+\frac{2\pi}{3},0\}$. Let $r_k:=|\sigma(t_k)|$ and let
	$$u_k^+(x)=\frac{u^+(x^0+r_kx)}{r_k^{3/2}}.$$
	Denote for simplicity that $\theta_1=\Theta+2\pi/3$. For each $\rho>0$ such that $\tilde{B}:=B_{\rho}(\cos\theta^*,\sin\theta^*)$ satisfies
	$$\varnothing=\tilde{B}\cap\left( \{(x,0):x\in\mathbb{R}\} \cup \{ (x,x\tan\theta_i), \ i=1,2, \ x\in\mathbb{R} \} \right),$$
	we infer from the formula for the unique blow-up $u_0$ in each case (see Proposition \ref{density}) that the nonnegative Radon measure
	$$\Delta u_k^+(\tilde{B})\rightarrow\Delta u_0^+(\tilde{B})=0 \quad \text{as} \quad r_k\rightarrow0+.$$
	On the other hand,
	$$\Delta u_k^+ = |\nabla u_k^+| \mathcal{S}\lfloor_{\partial\{u_k>0\}} \geq \sqrt{-x_2}\mathcal{S}\lfloor_{\partial\{u_k>0\}},$$
	where $\mathcal{S}\lfloor_{\partial\{u_k>0\}}$ means the 1-dimensional Hausdorff measure $\mathcal{S}$ restricted to the set $\partial\{u_k>0\}$, and implies that 
	$$0 \leftarrow \Delta u_k^+(\tilde{B})\geq c(\theta^*,\rho)$$
	since $\tilde{B}\cap\partial\{u_k>0\}$ contains a curve of length at least $2\rho-o(1)$, where $c(\theta^*,\rho)>0$, a contradiction. Thus the property claimed in Step 1 holds.
	
	\emph{Step 2.} It follows that $\sigma_1(t)\neq0$ for all sufficiently small $t\neq0$. Now a continuity argument yields that both $\mathcal{L}_+$ and $\mathcal{L}_-$ are connected sets. Consequently
	$$l_+:=\lim_{t\rightarrow0+} \arg\sigma(t)$$
	exists and has to be contained in the set $\{-\pi,\Theta,\Theta+\frac{2\pi}{3},0\}$, and
	$$l_-:=\lim_{t\rightarrow0-} \arg\sigma(t)$$
	exists and has to be contained in the set $\{\pi,\Theta,\Theta+\frac{2\pi}{3},2\pi\}$.
	
	\emph{Step 3.} In the case $M_+(0+)=\int_{B_1}x_2^- \chi_{\{x:\Theta<\theta<\Theta+\frac{2\pi}{3}\}} dx$, denote for simplicity that $\theta_1=\Theta+2\pi/3$ and we know from the formula for $u_0^+$ that
	$$\Delta u_0^+\left( B_{\sin\theta_i/2}(\cos\theta_i,\sin\theta_i) \right)>0 \quad \text{for} \quad i=1,2.$$
	It follows that the set $\{l_+, l_-\}$ contains both $\Theta$ and $\theta_2$. But then the sets $\{l_+,l_-\}$ and $\{\Theta, \Theta+2\pi/3\}$ must be equal, and the fact that $u\leq0$ in $B_R(x^0)\cap\{x_2=0\}$ for some small $R>0$ implies the cases (i-a)(i-b)(i-c) of the theorem.
	
	\emph{Step 4.} In the case $M_+(0+)\in\left\{ 0, \ \frac{\cos\theta_0+1}{2} \right\}$, we have that
	$$\Delta u_0^+\left( B_{\sin\theta/2}(\cos\theta,\sin\theta) \right)=0 \quad \text{for} \quad \theta=\Theta,\theta_1,$$
	which implies that $l_\pm\notin\{\Theta, \Theta+2\pi/3\}$. Thus $l_\pm\in\{-\pi,0\}$. Using the fact that $u\leq0$ on $B_R(x^0)\cap\{x_2=0\}$, we obtain in the case $l_+\neq l_-$ that $M_+(0+)=\int_{B_1} x_2^- dx =1$ for $\nu_0=(0,-1)$, and in the case $l_+=l_-$ that $M_+(0+)=0$. Together, the last two properties give the conclusion for case (ii) and case (iii) of the proposition.
\end{proof}

\subsection{Degenerate stagnation points}

In this section we define the degenerate stagnation points in the sense that the trivial blow-up limit $u^0\equiv0$. Under some additional assumptions, we will prove that the stagnation point $x^0$ can not be degenerate.

\begin{define}
	Let $u$ be a variational solution of (\ref{eq1.1}) in $D$ with $\lambda>0$ and $x^0\in S^u_+$. The set $D^u$ of degenerate stagnation points is defined by
	$$D^u:= \left\{ x^0\in S^u_+ \ \Bigg| \ \frac{u^+(x^0+rx)}{r^{3/2}}\rightarrow0 \quad \text{strongly in $W_{\rm{loc}}^{1,2}(\mathbb{R}^2)$ as $r\rightarrow0+$} \right\}.$$
	Otherwise we say $x^0\in S^u_+$ is non-degenerate.
\end{define}

Notice that Proposition \ref{density} gives alternative characterizations of non-degeneracy and degeneracy in terms of different blow-up limits and weighted densities.

First we consider the degenerate stagnation point whose weighted density $M_+(0+)=0$. The following proposition states that under the assumption of strong Bernstein estimate, we can exclude the points with weighted density $0$.

\begin{prop} \label{prop1}
	Let $u$ be a weak solution of (\ref{eq1.1}) in $D$ with $\lambda>0$ satisfying Assumption \ref{assume} and $x^0\in S^u_+$. Moreover, suppose the strong Bernstein estimate holds that
	$$|\nabla u^+|^2\leq x_2^- \quad \text{in $D$}.$$
	Then $M_+(0+)=0$ implies that $u\equiv0$ in some open ball containing $x^0$. In other words, $M_+(0+)\neq0$.
\end{prop}

\begin{proof}
	Suppose towards a contradiction that $x^0\in S^u_+$ and a blow-up sequence
	$$u_k^+(x):=\frac{u^+(x^0+r_kx)}{r_k^{3/2}}$$
	converging weakly in $W_{\rm{loc}}^{1,2}(\mathbb{R}^2)$ to a blow-up limit $u_0^+$ as in Lemma \ref{lem2}. Then $u_0^+\equiv0$ in $\mathbb{R}^2$. Consequently, recall that $\Delta u^+$ is a non-negative Radon measure in $B_R(x^0)\Subset D$,
	\begin{equation} \label{5.1}
	0 \leftarrow \Delta u_k^+(B_2)\geq\int_{B_2\cap\partial_{red}\{u_k^+>0\}} \sqrt{-x_2} d\mathcal{S} \quad \text{as} \quad r_k\rightarrow0+.
	\end{equation}
	On the other hand, construct a box $A=(-1,1)\times(-1,0)$, then there is at least one connected component $V_k$ of $\{u_k>0\}$ touching the origin and containing, by the maximum principle, a point $x^k=(x_1^k,x_2^k)\in\partial A$. If
	$$\max \{-x_2: x=(x_1,x_2)\in V_k\cap\partial A\} \ \text{does not converge to} \ 0 \ \text{as} \ r_k\rightarrow0+,$$
	we immediately obtain a contradiction to (\ref{5.1}). If
	$$\max \{-x_2: x=(x_1,x_2)\in V_k\cap\partial A\}\rightarrow0+,$$
	we use the free boundary condition as well as $|\nabla u^+|^2\leq x_2^-$ to obtain
	$$0=\Delta u_k^+(V_k\cap A)\leq\int_{V_k\cap\partial A}\sqrt{-x_2} d\mathcal{S} + \int_{A\cap\partial_{red}V_k}-\sqrt{-x_2} d\mathcal{S},$$
	namely
	$$\int_{A\cap\partial_{red}V_k}\sqrt{-x_2} d\mathcal{S} \leq \int_{V_k\cap\partial A}\sqrt{-x_2} d\mathcal{S},$$
	which is impossible.
\end{proof}

Now we focus on the degenerate stagnation point with nontrivial weighted density but trivial blow-up limit.

\begin{define}
	Let $u$ be a variational solution of (\ref{eq1.1}) in $D$ with $\lambda>0$ and $x^0\in S^u_+$. Let $\nu_0$ be a unit vector determined in Lemma \ref{lem1}. We define
	$$\Sigma^u:=\left\{x^0\in S^u_+: \ M_+(0+)=\frac{\cos\theta_0+1}{2} \right\}.$$
\end{define}
Notice that if $x^0\in\Sigma^u$, then the blow-up limit of $u_0^+\equiv0$.

We first present a result of Oddson \cite{O68}, which is a significant tool in excluding the horizontal points.

\begin{lem} \label{Oddson} (Main Theorem in \cite{O68})
	Let $r_0>1$, $\mu>1$ and
	$$G:=\left\{ (\rho\cos\theta,\rho\sin\theta): 0<\rho<r_0, |\theta|<\pi/(2\mu) \right\}.$$
	Let $w\in C^2(G)\cap C(\bar{G})$ be a superharmonic function in $G$, such that $w(0,0)=0$ and $w>0$ in $\bar{G}\backslash\{0\}$. Then there exists $\kappa>0$ such that
	$$w(\rho\cos\theta,\rho\sin\theta)\geq\kappa\rho^{\mu}\cos\mu\theta \quad \text{in} \quad G,$$
	and in particular,
	$$w(\rho\cos\theta,\rho\sin\theta)\Big|_{\theta=0} = w(\rho,0)\geq\kappa\rho^{\mu} \quad \text{for all} \quad \rho\in(0,r_0).$$
\end{lem}

We proceed to show that $\Sigma^u=\varnothing$, and then exclude the possibility of degenerate stagnation points.

\begin{prop}
	Let $u$ be a weak solution of (\ref{eq1.1}) in $D$ with $\lambda>0$ satisfying Assumption \ref{assume}, $\nu_0$ be the normal vector of $\partial\{u<0\}$ as in Lemma \ref{lem1}, and let $x^0\in S^u_+$. Suppose in addition that $\partial\{u>0\}$ is an injective curve in a neighborhood of $x^0$. Then $M_+(0+)\neq \frac{\cos\theta_0+1}{2}$, namely, $\Sigma^u=\varnothing$.
\end{prop}

\begin{proof}
	Suppose $M_+(0+)= \int_{B_1} x_2^-\chi_{\{x\cdot\nu_0>0\}}dx =\frac{\cos\theta_0+1}{2}$. It first follows from Proposition \ref{curve} that $\nu_0=(0,-1)$ and $M_+(0+)= \int_{B_1} x_2^- dx$. Otherwise $u_0>0$ in $\{x\cdot\nu_0<0\}\cap\{x_2<0\}$ by Proposition \ref{curve}, where $u_0<0$ holds too, a contradiction. Hence it yields that there are $r_0\in(0,R)$ and $\alpha\in(0,\pi/6)$ such that $u^+$ is harmonic in $\{u>0\}\cap B_{r_0}$ and $\bar{G}\backslash\{0\}\subset\{u>0\}\cap B_{r_0}$, where $G:=\left\{ (\rho\cos\theta,\rho\sin\theta): 0<\rho<r_0, \alpha<\theta<\pi-\alpha \right\}$. After a suitable rotation, we may apply Lemma \ref{Oddson} to obtain the existence of $\kappa>0$ such that
	$$u(0,x_2)\geq\kappa (-x_2)^{\mu} \quad \text{for all} \quad x_2\in(-r_0,0),$$
	where $\mu:=\pi/(\pi-2\alpha)$, so that $\mu<3/2$. But this contradicts the estimate
	$$u(0,x_2)\leq C(-x_2)^{3/2},$$
	which is a consequence of the Bernstein estimate in Assumption \ref{assume}.
\end{proof}

\subsection{The frequency formula}

In this section we replace the assumption of continuous injective curve by a weaker condition that $\{u=0\}$ has locally finite many connected components. We proceed by means of the frequency formula, allowing a blow-up limit analysis at degenerate stagnation points, where the scaling of the solution is different from the invariant scaling of the equation, and leads in combination with the result of concentration compactness in \cite{EM94}. The root of this formula is the classical frequency formula for Q-valued harmonic functions of Almgren in \cite{A00}, and it is successfully applied to investigate the singular profiles for one-phase water waves in \cite{DHP23}\cite{DY23}\cite{DY24}\cite{DY}\cite{VW11}\cite{VW12} and \cite{VW14}.

Before we introduce the frequency formula, there is an important observation that if $\nu_0\neq(0,-1)$, then the set $\Sigma^u=\varnothing$. In other words, we can verify the angle $\theta_0=0$ for $x^0\in\Sigma^u$.

\begin{lem}
	Let $u$ be a variational solution of (\ref{eq1.1}) in $D$ with $\lambda>0$ satisfying Assumption \ref{assume}, $x^0\in S^u_+$ and $\nu_0$ be the normal vector of $\partial\{u<0\}$ as in Lemma \ref{lem1}, and $M_+(0+)=\int_{B_1} x_2^-\chi_{\{x\cdot\nu_0>0\}}dx=\frac{\cos\theta_0+1}{2}$. Then $\nu_0=(0,-1)$ and thus $M_+(0+)=\int_{B_1} x_2^- dx =1$.
\end{lem}

\begin{proof}
	Let $Y_0=(y_1^0,-1)\in\mathbb{R}^2$ be a point such that $Y_0\cdot\nu_0=0$. Suppose by contradiction that $\nu_0\neq(0,-1)$. Consider the ball $B_{1/2}(Y_0)\Subset\{x_2<0\}$. Then the blow-up sequence
	$$u_k(x)=\frac{u^+(x^0+r_kx)}{r_k^{3/2}}-\frac{u^-(x^0+r_kx)}{r_k}$$
	satisfies $u_k(x)<0$ in $\{x \ | \ x\cdot\nu_0<-\delta\}\cap B_{1/2}(Y_0)$ for $r_k$ small enough. Meanwhile, the $L_{\rm{loc}}^1(\mathbb{R}^2)$ limit $\chi_0^+$ of $\chi_{\{u>0\}}$ satisfies $\chi_0^+\equiv0$ in $\{x\cdot\nu_0>0\}$. Hence the continuity of $u_k$ implies that $u_k$ changes sign in $B_{1/2}(Y_0)$, and 
	$$\partial\{u_k>0\}\cap B_{1/2}(Y_0)\neq\varnothing,$$
	and we can deduce that the Hausdorff measure
	$$\mathcal{H}^1 \left( \{x_2\leq-1/2\}\cap\partial_{red}\{x: \ u(x^0+r_kx)>0\} \right)\geq c_1 >0$$
	for some constant $c_1$ independent of $k$. Consequently,
	\begin{equation*}
	\begin{aligned}
	\Delta u_k^+(B_{1/2}(Y_0))&=\int_{B_{1/2}(Y_0)\cap\partial\{x: \ u(x^0+r_kx)>0\}} |\nabla u_k^+| d\mathcal{S} \\
	&\geq \int_{B_{1/2}(Y_0)\cap\partial_{red}\{x: \ u(x^0+r_kx)>0\}} \sqrt{-x_2} d\mathcal{S} \\
	&>0.
	\end{aligned}
	\end{equation*}
	However, the strong convergence of $u_k^+$ to $u_0^+$ in $W_{\rm{loc}}^{1,2}(\mathbb{R}^2)$ together with Proposition \ref{density} gives that $\Delta u_k^+(B_{1/2}(Y_0))\rightarrow0+$, a contradiction. Hence $\nu_0=(0,-1)$ and $M_+(0+)=\int_{B_1} x_2^- dx$.
\end{proof}

The subsequent process follows the idea from V$\check{a}$rv$\check{a}$ruc$\check{a}$ and Weiss in \cite{VW11}\cite{VW12} and \cite{VW14} for one-phase gravity water wave and introduces a frequency formula to exclude the horizontally flat singularities, whose weighted density is equal to $1$. 

Define
\begin{equation}
D_+(r)=D_{+,x^0,u}(r)=\frac{r\int_{B_r(x^0)} |\nabla u^+|^2 dx}{\int_{\partial B_r(x^0)} (u^+)^2 d\mathcal{S}}
\end{equation}
and
\begin{equation}
V_+(r)=V_{+,x^0,u}(r)=\frac{r\int_{B_r(x^0)} x_2^-\left( 1-\chi_{\{u>0\}} \right) dx}{\int_{\partial B_r(x^0)} (u^+)^2 d\mathcal{S}}.
\end{equation}
Define the "frequency" function
\begin{equation} \label{H}
\begin{aligned}
H_+(r)&=H_{+,x^0,u}(r)=D_+(r)-V_+(r) \\
&=\frac{r\int_{B_r(x^0)} \left( |\nabla u^+|^2 - x_2^-\left( 1-\chi_{\{u>0\}} \right) \right) dx }{\int_{\partial B_r(x^0)} (u^+)^2 d\mathcal{S}}.
\end{aligned}
\end{equation}

The "frequency" function $H_+(r)$ has an implication of the decay rate of $|\nabla u^+|$ at $x^0$, which allows the compactness of a new blow-up sequence
\begin{equation} \label{vk}
v_k^+(x):=\frac{u^+(x^0+r_kx)}{\sqrt{ r_k^{-1}\int_{\partial B_{r_k}(x^0)} (u^+)^2 d\mathcal{S} }}
\end{equation}
as $r_k\rightarrow0+$ with nontrivial blow-up limit, and demonstrates that the horizontally flat singularity is impossible.

Notice that $H_+(r)$ only involves the positive phase $u^+$, and the "weighted density" $M_+(0+)$ is proved to be the constant $1$, which is same with the result for one-phase water wave when $x^0\in\Sigma^u$. Hence the subsequent analysis can deduced as in the one-phase work \cite{VW11}, and we sketch the demonstration here for the completeness of the proof.

The derivative of the "frequency" function $H_+(r)$ writes
\begin{equation}
\begin{aligned}
H_+'(r) &= \frac{2}{r}\left( \int_{\partial B_r(x^0)} (u^+)^2 d\mathcal{S} \right)^{-1}\int_{\partial B_r(x^0)} \left[ r(\nabla u^+\cdot \nu) - D_+(r)u^+ \right]^2 d\mathcal{S} \\
& \quad + \frac2r V_+^2(r) +\frac2r V_+(r)\left( H_+(r)-\frac32 \right),
\end{aligned}
\end{equation}
or
\begin{equation}
\begin{aligned}
H_+'(r) &= \frac{2}{r}\left( \int_{\partial B_r(x^0)} (u^+)^2 d\mathcal{S} \right)^{-1}\int_{\partial B_r(x^0)} \left[ r(\nabla u^+\cdot \nu) - H_+(r)u^+ \right]^2 d\mathcal{S} \\
& \quad + \frac2r V_+(r)\left( H_+(r)-\frac32 \right).
\end{aligned}
\end{equation}
Hence the function $H_+(r)$ has a right limit $H_+(0+)$, since we have $V_+(r)>0$ and $H_+(r)\geq\frac32$. Direct calculation shows that
$$H_+(0+)\geq\frac32,$$
and the sequence $v_k^+$ defined in (\ref{vk}) is bounded in $W^{1,2}(B_1)$ and satisfies
\begin{equation} \label{6-1}
\int_{B_\sigma\backslash B_\rho} |x|^{-5}[\nabla v_k^+(x)\cdot x - H_+(0+)v_k^+(x)]^2 dx \rightarrow 0 \quad \text{as} \quad r_k\rightarrow0+
\end{equation}
for every $0<\rho<\sigma<1$. Denote the blow-up limit
$$v_0(x):=\lim_{r_k\rightarrow0+} v_k(x) \quad \text{weakly in $W^{1,2}(B_1)$},$$
it is straightforward to verify that $v_0\geq0$ is a homogeneous function of degree $H_+(0+)$. The next lemma shows the concentration compactness that $v_k^+$ converges strongly to $v_0$ in $W_{\rm{loc}}^{1,2}(B_1\backslash\{0\})$, and gives the nontrivial form of the blow-up limit $v_0$.

\begin{lem} \label{cc}
	Let $u$ be a variational solution of (\ref{eq1.1}) in $D$ with $\lambda>0$ satisfying Assumption \ref{assume}, and $x^0\in\Sigma^u$. Let $r_k\rightarrow0+$ be such that the sequence $v_k^+$ given by (\ref{vk}) converges weakly to $v_0$ in $W^{1,2}(B_1)$. Then the following holds:
	
	(i) $v_k^+$ converges to $v_0$ strongly in $W_{\rm{loc}}^{1,2}(B_1\backslash\{0\})$, $v_0$ is continuous on $B_1$ and $\Delta v_0$ is a non-negative Radon measure satisfying $v_0\Delta v_0=0$ in the sense of Radon measures in $B_1$.
	
	(ii) There exists an integer $N(x^0)\geq2$ such that
	$$H_{+,x^0,u}(0+)=N(x^0)$$
	and
	$$\frac{u^+(x^0+rx)}{\sqrt{ r^{-1}\int_{\partial B_r(x^0)} (u^+)^2 d\mathcal{S} }} \rightarrow \frac{\rho^{N(x^0)}|\sin(N(x^0)\min(\max(\theta,0),\pi))|}{\sqrt{\int_0^\pi \sin^2(N(x^0)\theta)d\theta }} \quad \text{as} \quad r\rightarrow0+$$
	strongly in $W_{\rm{loc}}^{1,2}(B_1\backslash\{0\})$ and weakly in $W^{1,2}(B_1)$, where $x=(\rho\cos\theta,\rho\sin\theta)$.
\end{lem}

In the two-dimensional case we prove concentration compactness which allows us to preserve variational solutions in the blow-up limit at degenerate stagnation points and excludes concentration. As in \cite{VW11} and \cite{VW12}, we combine the concentration compactness result of Evans and M\"uller in \cite{EM94} with information gained by our frequency formula. In addition, we obtain strong convergence of our blow-up sequence which is necessary in order to prove our main theorems and get the explicit form of the limit of $v_k$. We refer the detailed proof to the Appendix C, since it is quite similar with the aforementioned papers \cite{VW11} and \cite{VW12} for one-phase gravity water wave.

Now we come to our conclusion.

\begin{thm}
	Let $u$ be a weak solution of (\ref{eq1.1}) in $D$ with $\lambda>0$ and $x^0\in S^u_+$ satisfying Assumption \ref{assume}, $\nu_0$ and $\theta_0$ be as in Lemma \ref{lem1}, and suppose that $|\nabla u^+|^2\leq x_2^-$ in $D$. Suppose moreover that $\{u=0\}$ has locally only finite many connected components. Then we have $M_{+,x^0,u}(0+) = \int_{B_1} x_2^- \chi_{\{x:\Theta<\theta<\Theta+\frac{2\pi}{3}\}} dx$ for $\Theta\in[-\pi, -\frac{2\pi}{3}]$ lying in $\{-\frac{5\pi}{6}, \theta_0-\pi, \theta_0-\frac{2\pi}{3}\}$, and
	\begin{equation*}
	\frac{u^+(x^0+rx)}{r^{3/2}}\rightarrow
	\begin{cases}
	\frac23 C \rho^{3/2} \cos \left( \frac32\theta-\frac32\Theta-\frac{\pi}{2} \right), \quad\, \Theta<\theta<\Theta+\frac23\pi, \\
	0, \qquad\qquad\qquad\qquad\qquad\qquad \text{otherwise}
	\end{cases}
	\end{equation*}
	strongly in $W_{\rm{loc}}^{1,2}(\mathbb{R}^2)$ and locally uniformly in $\mathbb{R}^2$, where
	\begin{equation*}
	C=
	\begin{cases}
	\sqrt{2}/2 \qquad\qquad\qquad \text{when} \ \Theta=-\frac56\pi \ \text{(detached case)}, \\
	\sqrt{-\sin\left(\theta_0-\frac{\pi}{3}\right)} \quad\ \ \text{when} \ \Theta=\theta_0-\pi \ \text{(left-overlapping case)}, \\
	\sqrt{-\sin\left(\theta_0-\frac{2\pi}{3}\right)} \quad \text{when} \ \Theta=\theta_0-\frac{2\pi}{3} \ \text{(right-overlapping case)}.
	\end{cases}
	\end{equation*}
	Moreover, the free boundary $\partial\{u>0\}$ is in a neighborhood of $x^0$ the union of two $C^1$-graphs.
\end{thm}

\begin{proof}
	We first show that $\Sigma^u$ is empty. Suppose towards a contradiction that there exists $x^0\in\Sigma^u$. Recalling Lemma \ref{cc} we infer that there exists an integer $N(x^0)\geq2$ such that
	$$\frac{u^+(x^0+rx)}{\sqrt{ r^{-1}\int_{\partial B_r(x^0)} (u^+)^2 d\mathcal{S} }} \rightarrow \frac{\rho^{N(x^0)}|\sin(N(x^0)\min(\max(\theta,0),\pi))|}{\sqrt{\int_0^\pi \sin^2(N(x^0)\theta)d\theta }} \quad \text{as} \quad r\rightarrow0+$$
	strongly in $W_{\rm{loc}}^{1,2}(B_1\backslash\{0\})$ and weakly in $W^{1,2}(B_1)$, where $x=(\rho\cos\theta,\rho\sin\theta)$. But then the assumption about connected components implies that $\partial_{red}\{x: u(x^0+rx)>0\}$ contains the image of a continuous curve converging, as $r\rightarrow0+$, locally in $\{x_2<0\}$ to a half-line $\{\alpha z: \alpha>0\}$ where $x_2<0$. It follows that
	$$\mathcal{H}^1(\{x_2<0\}\cap\partial_{red}\{x:u(x^0+rx)>0\})\geq c_1>0,$$
	which leads to a contradiction with
	$$0 \leftarrow \Delta \frac{u^+(x^0+rx)}{r^{3/2}}(B_1) = \int_{B_1\cap\partial_{red}\{x:u(x^0+rx)>0\}} \sqrt{-x_2} d\mathcal{S}.$$
	Hence $\Sigma^u$ is empty.
	
	We assume for simplicity that $x^0=0\in S^u_+$, we have that
	\begin{equation*}
	\frac{u^+(x^0+rx)}{r^{3/2}}\rightarrow
	\begin{cases}
	\frac23 C \rho^{3/2} \cos \left( \frac32\theta-\frac32\Theta-\frac{\pi}{2} \right), \quad \Theta<\theta<\Theta+\frac23\pi, \\
	0, \qquad\qquad\qquad\qquad\qquad\qquad \text{otherwise}
	\end{cases}
	\end{equation*}
	strongly in $W_{\rm{loc}}^{1,2}(\mathbb{R}^2)$ and locally uniformly in $\mathbb{R}^2$, where $\Theta\in[-\pi, -\frac{2\pi}{3}]$ has three possible values $\{-\frac{5\pi}{6}, \theta_0-\pi, \theta_0-\frac{2\pi}{3}\}$, and $C$ is uniquely determined. We will show that in a neighborhood of $0$ the free boundary $\partial\{u>0\}$ is the union of two $C^1$-graphs. Note that it suffices to prove for the part $\partial\{u>0\}\backslash\partial\{u<0\}$, since the part $\partial\{u>0\}\cap\partial\{u<0\}\subset\partial\{u<0\}$ is already $C^{1,\alpha}$ by Proposition \ref{reg}, and the $C^{1,\alpha}$ regularity for $\Gamma_{\rm tp}$ near branch point is referred to \cite{PSV21}. We will give only the proof for $x_1<0$, since the case for $x_1>0$ is quite similar.
	
	For
	$$v_\rho(x):=\frac{u^+(\rho x)}{\rho^{3/2}} - \frac{u^-(\rho x)}{\rho}$$
	we have that
	\begin{equation*}
	\begin{cases}
	\Delta v_\rho(x) = 0 \qquad\quad\, \text{for} \quad x\in\{v_\rho(x)>0\}, \\
	|\nabla v_\rho(x)|^2 = -x_2 \quad \text{for} \quad x\in\partial\{v_\rho>0\}\backslash\partial\{v_\rho<0\}.
	\end{cases}
	\end{equation*}
	Scaling once more for $\xi=(\xi_1,\xi_2)\in\partial B_1\cap(\partial\{v_\rho>0\}\backslash\partial\{v_\rho<0\})$ away from the origin $0$, which implies that for $\rho$ small enough, $\xi_2\leq-\frac1{10}$, we obtain for
	$$w_r(x):= \frac{v_\rho^+(\xi+rx)}{\sqrt{-\xi_2}r}$$
	that
	\begin{equation*}
	\begin{cases}
	\Delta w_r(x) = 0 \qquad\quad\quad\ \text{for} \quad x\in\{w_r(x)>0\}, \\
	|\nabla w_r(x)|^2 = 1 + \frac{rx_2}{\xi_2} \quad \text{for} \quad x\in\partial\{w_r>0\}\backslash\partial\{w_r<0\}.
	\end{cases}
	\end{equation*}
	We are going to use a flatness-implies-regularity result in \cite{S11}. For each $\epsilon\in(0,\epsilon_0)$ for some $\epsilon_0>0$,
	\begin{equation} \label{7-1}
	\max(x\cdot\bm e_0 - \epsilon, 0)\leq w \leq \max(x\cdot\bm e_0 + \epsilon, 0) \quad \text{in} \quad B_1 \quad \text{for some unit vector $\bm e_0$}
	\end{equation}
	implies that the outward unit normal $\bm e_r$ on the free boundary $\partial\{w_r>0\}\backslash\partial\{w_r<0\}$ satisfies
	$$|\bm e_r(0)-\bm e_0|\leq C\epsilon^2.$$
	Note that $\bm e_r(0)=v_\rho(\rho\xi)$. Since (\ref{7-1}) is satisfied for $\bm e_0=(\sin\theta_0,-\cos\theta_0)$, $r=r(\epsilon)$ and every sufficiently small $\rho>0$, we obtain that the outward unit normal $\bm e(x)$ on $\partial\{u>0\}\backslash\partial\{u<0\}$ converges to $\bm e_0$ as $x\rightarrow0$, $x_1<0$. It follows that the present curve component is the graph of a $C^1$-function up to $x_1=0$. The same argument holds for $x_1>0$, which completes the proof.
\end{proof}

\section{The complete-degenerate case}

In this section, we consider the remaining case $\lambda=0$, which means that if there is a stagnation point $x^0=(x_1^0,0)\in\Gamma_{\rm tp}$, then both $u^\pm$ have degenerate gradient at $x^0$, namely, $|\nabla u^\pm(x^0)|=0$. However, we will prove that the stagnation point cannot be on the two-phase free boundary $\Gamma_{\rm tp}$. Otherwise, the monotonicity formula says that at least one of the weighted densities of $u^\pm$ at $x^0$ vanish, which in turn implies that $u^+$ or $u^-$ is equal to $0$ near $x^0$, contradicting with the definition of $\Gamma_{\rm tp}$.

The corresponding governing equation in $D$ writes
\begin{equation}
\begin{cases}
\Delta u = 0 \qquad\qquad\qquad\ \ \; \text{in} \quad D\cap\{u\neq0\}, \\
|\nabla u^\pm|^2 = -x_2 \qquad\quad\ \ \ \text{on} \quad D\cap\Gamma_{\rm op}^\pm, \\
|\nabla u^+|^2 - |\nabla u^-|^2 = 0 \quad \text{on} \quad D\cap\Gamma_{\rm tp},
\end{cases}
\end{equation}
when $\lambda=0$. The last equality implies that $|\nabla u^+|$ and $|\nabla u^-|$ possess the same decay rate near the stagnation point $x^0\in\Gamma_{\rm tp}$. Similarly as in Section 2, we give the following assumption, which is somewhat equivalent to the Bernstein estimate to the two fluids.
\begin{assume} \label{assume2}
	(Bernstein estimate) There is a positive constant $C$ such that $|\nabla u|^2 \leq C x_2^-$ locally in $D$.
\end{assume}

\subsection{Monotonicity formula and densities}

In this section we present a monotonicity formula involving the two-phase terms $u^\pm$, hoping to make a blow-up analysis for both $u^\pm$ at $x^0$. The polynomial convergence rates of $u^\pm$ near the stagnation point $x^0\in\Gamma_{\rm tp}$ are both of at most $\frac32$-order, and we construct the monotonicity formula as follows.

We define
\begin{equation}
I(r)=I_{x^0,u}(r)=r^{-3}\int_{B_r(x^0)}\left( |\nabla u|^2 + (-x_2)\chi_{\{u\neq0\}} \right)dx,
\end{equation}
\begin{equation}
J(r)=J_{x^0,u}(r)=r^{-4}\int_{\partial B_r(x^0)}  u^2 d\mathcal{S},
\end{equation}
and the Weiss boundary adjusted energy
\begin{equation}
\begin{aligned}
M(r)&=M_{x^0,u}(r) \\
&=I(r)-\frac32 J(r) \\
&=r^{-3}\int_{B_r(x^0)}\left( |\nabla u|^2 + (-x_2)\chi_{\{u\neq0\}} \right)dx-\frac32 r^{-4}\int_{\partial B_r(x^0)}  u^2 d\mathcal{S}.
\end{aligned}
\end{equation}
For notational simplicity, we define further that
$$M_+(r)=r^{-3}\int_{B_r(x^0)}\left( |\nabla u^+|^2 + (-x_2)\chi_{\{u>0\}} \right)dx-\frac32 r^{-4}\int_{\partial B_r(x^0)}  (u^+)^2 d\mathcal{S}$$
and
$$M_-(r)=r^{-3}\int_{B_r(x^0)}\left( |\nabla u^-|^2 + (-x_2)\chi_{\{u<0\}} \right)dx-\frac32 r^{-4}\int_{\partial B_r(x^0)}  (u^-)^2 d\mathcal{S}.$$
Notice that
$$M(r)=M_+(r)+M_-(r).$$
We will first give the monotonicity of $M(r)$. It should be noted that unlike the one-phase situation, the function $u$ has no sign condition in $M(r)$. Nevertheless, the explicit form of the derivative $M'(r)$ is showed to be same as in the one-phase problem.

\begin{prop}
	Let $u$ be a variational solution of (\ref{eq1.1}) in $D$ with $\lambda=0$ and $x^0\in\Gamma_{\rm tp}$ be a stagnation point. Then for any $0<r<R/2$.
	Then, for a.e. $r\in(0,R/2)$,
	\begin{equation*}
	\frac{d}{dr} M(r)=2r^{-3}\int_{\partial B_r(x^0)} \left( \nabla u\cdot\nu-\frac32\frac{u}r \right)^2 d\mathcal{S}.
	\end{equation*}
\end{prop}

\begin{proof}
	We calculate the derivative of $M(r)$ directly and get that
	\begin{equation}
	\begin{aligned}
	\frac{d}{dr}M(r)
	&=r^{-3}\int_{\partial B_r(x^0)} \left( |\nabla u|^2 - x_2\chi_{\{u\neq0\}} \right)d\mathcal{S} + 3r^{-4}\int_{B_r(x^0)} x_2\chi_{\{u\neq0\}} dx \\
	&\quad -6r^{-4}\int_{\partial B_r(x^0)} u\nabla u\cdot\nu d\mathcal{S} + \frac92 r^{-5}\int_{\partial B_r(x^0)} u^2 d\mathcal{S}.
	\end{aligned}
	\end{equation}
	Notice that we use the relationship
	$$\int_{B_r(x^0)} |\nabla u|^2 dx = \int_{\partial B_r(x^0)} u\nabla u\cdot\nu d\mathcal{S},$$
	which is approximated by
	$$\int_{B_r(x^0)} \nabla u^\pm \cdot \nabla(\max\{u^\pm-\epsilon,0\}^{1+\epsilon}) dx = \int_{\partial B_r(x^0)} \max\{u^\pm - \epsilon,0\}^{1+\epsilon}\nabla u^\pm\cdot\nu d\mathcal{S}$$
	while $\epsilon\rightarrow0$.
	
	Now for small $\kappa$ and $\eta_\kappa(t):=\max\{0,\min\{1,\frac{r-t}{\kappa}\}\}$, we take after approximation $\bm\phi_\kappa(x)=\eta_\kappa(|x-x^0|)(x-x^0)\in W_0^{1,2}(B_r(x^0);\mathbb{R}^2)$ as a test function in the definition of the variational solution $u$. We obtain
	\begin{equation*}
	\begin{aligned}
	0 &= \int_{D} \left( |\nabla u|^2 - x_2\chi_{\{u\neq0\}} \right) \left( 2\eta_\kappa(|x-x^0|)+\eta_\kappa'(|x-x^0|)|x-x^0| \right)dx \\
	& \quad -2\int_{D} \left( |\nabla u|^2 \eta_\kappa(|x-x^0|) + \left( \nabla u \cdot \frac{x-x^0}{|x-x^0|} \right)^2 \eta_\kappa'(|x-x^0|)|x-x^0| \right)dx \\
	& \quad - \int_{D} x_2\eta_\kappa(|x-x^0|) \chi_{\{u\neq0\}} dx
	\end{aligned}
	\end{equation*}
	and hence we get
	\begin{equation*}
	0 = 2r\int_{\partial B_r(x^0)}(\nabla u\cdot\nu)^2 d\mathcal{S} -r\int_{\partial B_r(x^0)} \left( |\nabla u|^2 + (-x_2)\chi_{\{u\neq0\}} \right)d\mathcal{S} -\int_{B_r(x^0)} 3x_2\chi_{\{u\neq0\}} dx
	\end{equation*}
	by passing the limit $\kappa\rightarrow0$. Plugging this into $\frac{d}{dr}M(r)$, we obtain that for a.e. $r\in(0,R/2)$,
	\begin{equation*}
	\begin{aligned}
	\frac{d}{dr}M(r)&=\frac{2}{r^3}\int_{\partial B_r(x^0)} (\nabla u\cdot\nu)^2 d\mathcal{S} -\frac6{r^4}\int_{\partial B_r(x^0)} u\nabla u\cdot\nu d\mathcal{S} +\frac9{2r^5}\int_{\partial B_r(x^0)} u^2 d\mathcal{S} \\
	&= 2r^{-3}\int_{\partial B_r(x^0)} \left( \nabla u\cdot\nu -\frac32\frac{u}r \right)^2 d\mathcal{S}.
	\end{aligned}
	\end{equation*}
\end{proof}

Now we come to the scaling argument and define the blow-up sequence at $x^0$:
\begin{equation} \label{blowup2}
u_k(x)=\frac{u(x^0+r_kx)}{r_k^{3/2}}
\end{equation}
for $r_k\rightarrow0+$ and $x\in B_{R/r_k}(0)$ such that $x^0+r_kx\in B_R(x^0)\Subset D$. We will first give a key observation for $u_k(x)$ in Lemma \ref{lem3} that the blow-up sequences of the two phases $u_k^\pm$ converge to the positive part and the negative part of $u_0:=\lim_{r_k\rightarrow0+} u_k$ respectively, which is untouched in the one-phase works.

\begin{lem} \label{lem3}
	Let $u$ be a (local) variational solution of (\ref{eq1.1}) with $\lambda=0$ satisfying Assumption \ref{assume2} and $x^0\in\Gamma_{\rm tp}$ be a stagnation point. Let $u_k$ be a blow-up sequence of $u$ at $x^0$ that converges weakly in $W_{\rm{loc}}^{1,2}(\mathbb{R}^2)$ to a blow-up limit $u_0$. Then $u_k$ converges strongly to $u_0$ in $\mathbb{R}^2$, and $u_0^+$ is a homogeneous function of degree $\frac32$. Moreover, if we denote the limit functions as $V(x)$ and $W(x)$ in $W_{\rm{loc}}^{1,2}(\mathbb{R}^2)$ such that
	$$u_k^+:=\frac{u^+(x^0+r_kx)}{r_k^{3/2}}\rightarrow V(x)$$
	and
	$$u_k^-:=\frac{u^-(x^0+r_kx)}{r_k^{3/2}}\rightarrow W(x)$$
	as $r_k\rightarrow0+$, then
	$$V=u_0^+=\max\{u_0,0\} \quad \text{and} \quad W=u_0^-=\max\{-u_0,0\}.$$
\end{lem}

\begin{proof}
	The proof of the strong convergence of $u_k$ to a $\frac32$-homogeneous function $u_0$ in $W_{\rm{loc}}^{1,2}(\mathbb{R}^2)$ follows along the same argument as in Lemma \ref{lem2} and is omitted here, and we only prove the statements for $V$ and $W$. Notice that we have $u_0=V-W=u_0^+ - u_0^-$, and $V$ and $W$ have disjoint support. It is straightforward to deduce that for $V$, in $\{u_0^+>0\}$,
	$$u_0^+ = u_0 = \lim_{r_k\rightarrow0+}u_k = \lim_{r_k\rightarrow0+} u_k^+ = V$$
	and in $\{u_0^->0\}$,
	$$u_0^+ = 0 = V.$$
	It remains to show that $u_0^+ = V$ in $\{u_0\equiv0\}$. In fact, if there is a point $y\in\{u_0\equiv0\}$ such that $V(y)>0$, then $W(y)= V(y) - u_0(y)>0$, a contradiction to the fact that $V$ and $W$ have disjoint support. It concludes the proof.
\end{proof}

Proposition \ref{density2} gives the description of possible values of $M(0+)$, and the corresponding blow-up limits allow the singularity analysis at $x^0$ and implies the asymptotic behavior of the free boundaries near the stagnation point.

\begin{prop}(Densities for $u^\pm$) \label{density2}
	Let $u$ be a variational solution of (\ref{eq1.1}) in $D$ with $\lambda=0$ satisfying Assumption \ref{assume2} and $x^0\in\Gamma_{\rm tp}$ be a stagnation point. Then the weighted density $M(0+)$ has three possible values
	$$M(0+)=\lim_{r\rightarrow0+}\int_{B_1} x_2^- \chi_{\{u_r\neq0\}}dx\in\left\{ 0, \ \frac{\sqrt{3}}{2}, \ 1 \right\}.$$
	More precisely, the weighted density can be classified as follows.
	
	(i) (Non-trivial blow-up limit) If
	\begin{equation*}
	u_k(x)=\frac{u(x^0+rx)}{r^{3/2}}\rightarrow u_0(x)=
	\begin{cases}
	\frac{\sqrt{2}}{3}\rho^{3/2}\cos(\frac32\theta+\frac{3\pi}{4}), \quad -\frac{5\pi}{6}<\theta<-\frac{\pi}{6}, \\
	0, \qquad\qquad\qquad\qquad\quad \text{otherwise}
	\end{cases}
	\end{equation*}
	or
	\begin{equation*}
	u_k(x)=\frac{u(x^0+rx)}{r^{3/2}}\rightarrow u_0(x)
	\begin{cases}
	-\frac{\sqrt{2}}{3}\rho^{3/2}\cos(\frac32\theta+\frac{3\pi}{4}), \quad -\frac{5\pi}{6}<\theta<-\frac{\pi}{6}, \\
	0, \qquad\qquad\qquad\qquad\quad\ \; \text{otherwise}
	\end{cases}
	\end{equation*}
	as $r\rightarrow0+$ strongly in $W_{\rm{loc}}^{1,2}(\mathbb{R}^2)$ and locally uniformly in $\mathbb{R}^2$, where $x=(\rho\cos\theta, \rho\sin\theta)$, then 
	$$M(0+)=\frac{\sqrt{3}}{2}.$$
	In this case, either
	$$M_+(0+)=\int_{B_1}x_2^- \chi_{\{x:-\frac{5\pi}{6}<\theta<-\frac{\pi}{6}\}} dx=\frac{\sqrt{3}}{2}, \quad M_-(0+)=0,$$
	or
	$$M_+(0+)=0, \quad M_-(0+)=\int_{B_1}x_2^- \chi_{\{x:-\frac{5\pi}{6}<\theta<-\frac{\pi}{6}\}} dx=\frac{\sqrt{3}}{2}.$$
	
	(ii) (Trivial blow-up limit) If
	$$u_k(x)=\frac{u(x^0+rx)}{r^{3/2}}\rightarrow u_0(x)\equiv0$$
	as $r\rightarrow0+$ strongly in $W_{\rm{loc}}^{1,2}(\mathbb{R}^2)$ and locally uniformly in $\mathbb{R}^2$,then
	$$M(0+)\in\{0,1\}$$
	has two possible values. In this case, we have
	$$M_+(0+)=0, \quad M_-(0+)=0,$$
	or
	$$M_+(0+)=\int_{B_1}x_2^- dx=1, \quad M_-(0+)=0,$$
	or
	$$M_+(0+)=0, \quad M_-(0+)=\int_{B_1}x_2^- dx=1.$$
\end{prop}

\begin{proof}
	Consider a blow-up sequence $u_k$ as in (\ref{blowup2}), where $r_k\rightarrow0+$, with the blow-up limit $u_0=u_0^+ - u_0^-$. The scaling argument and the $\frac32$-homogeneity of $u_0$ gives straightforward that
	$$M(0+)=\lim_{r\rightarrow0+}\int_{B_1} x_2^- \chi_{\{u_r\neq0\}}dx.$$
	Take $\bm\phi=(\phi_1,\phi_2)\in W_0^{1,2}(\mathbb{R}^2;\mathbb{R}^2)$ and $\bm\phi_k(y):=\bm\phi(\frac{y-x^0}{r_k})=\bm\phi(x)$ for $y=x^0+r_kx$. Notice that $\bm\phi_k(y)\in W_0^{1,2}(\mathbb{R}^2;\mathbb{R}^2)$. We use $\bm\phi_k=(\phi_{k,1},\phi_{k,2})$ as the test function in the definition of the variational solution $u$, and obtain
	\begin{equation*}
	\begin{aligned}
	0&=\int_{B_r(x^0)} \left( |\nabla u|^2 div \bm\phi_k - 2\nabla u D\bm\phi_k\nabla u - (\phi_{k,2}+y_2 div \bm\phi_k)\chi_{\{u\neq0\}} \right) dy \\
	&= r_k^2 \int_{B_{r/r_k}} \left( |\nabla u_k|^2 div\bm\phi - 2\nabla u_k D\bm\phi \nabla u_k - (\phi_2+x_2 div\bm\phi)\chi_{\{u_k\neq0\}} \right) dx.
	\end{aligned}
	\end{equation*}
	Passing the limit $r_k\rightarrow0+$, the strong convergence of $u_k$ to $u_0$ in $W_{\rm{loc}}^{1,2}(\mathbb{R}^2)$ and the compact embedding from $BV$ to $L^1$ gives
	\begin{equation} \label{eq3}
	0=\int_{\mathbb{R}^2} \left( |\nabla u_0|^2 div\bm\phi -2\nabla u_0 D\bm\phi \nabla u_0 -(\phi_2+x_2 div\bm\phi)\chi_0 \right) dx,
	\end{equation}
	where $\chi_0$ is the strong $L_{\rm{loc}}^1$ limit of $\chi_{\{u_k\neq0\}}$ along a subsequence. The values of the function $\chi_0$ are almost everywhere in $\{0,1\}$, and the locally uniform convergence of $u_k^+$ to $u_0^+$ implies that $\chi_0=1$ in $\{u_0\neq0\}$. Moreover, if we denote $\chi_0^\pm$ the strong $L_{\rm{loc}}^1$ limit of $\chi_{\{u_k^\pm>0\}}$, then $\int_B \chi_0 dx = \int_B \left( \chi_0^+ + \chi_0^- \right)dx$ and $\int_B \chi_0^+ \chi_0^- dx = 0$ for any domain $B\Subset\mathbb{R}^2$.
	
	Utilizing the $\frac32$-homogeneity of $u_0$, we write the harmonic equation of $u_0(x_1,x_2)=u_0(\rho,\theta)=\rho^{3/2}f(\theta)$ in polar coordinate for $(x_1,x_2)=(\rho\cos\theta,\rho\sin\theta)$,
	$$\frac34 \rho^{-1/2}f(\theta) + \frac1r\frac32 \rho^{1/2}f(\theta) + \rho^{-1/2}f''(\theta)=0 \quad \text{in} \quad \{u_0\neq0\}.$$
	That is,
	$$f''(\theta)+\frac94f(\theta)=0 \quad \text{in} \quad \{u_0\neq0\}.$$
	Consequently, each connected component of $\{u_0>0\}$ and $\{u_0<0\}$ is a cone with vertex at the origin and of opening angle $\frac{2\pi}{3}$. Since $u\equiv0$ in $\{x_2\geq0\}$, at least one of $u_0^\pm$ must vanish. Note also that $\chi_0$ is a constant, denoted as $I_0$, in each open connected set $G\subset\{u_0=0\}$ that does not intersect the axis $\{x_2=0\}$.
	
	Consider the first case when $\{u_0^+>0\}\neq\varnothing$ and $\{u_0^->0\}=\varnothing$. In this case $\{u_0^+>0\}$ is a cone as described above. The case when $\{u_0^->0\}\neq\varnothing$ and $\{u_0^+>0\}=\varnothing$ is symmetric.
	
	Let $z\in\partial\{u_0^+>0\}\backslash\{0\}$ be an arbitrary point and the normal to $\partial\{u_0^+>0\}$ is the constant value $\bm n$ in $B_{\tau}(z)\cap\{u_0^+>0\}$ for some $\tau>0$. Plugging in $\bm\phi(x):=\eta(x)\bm n$ into (\ref{eq3}), where $\eta\in W_0^{1,2}(B_{\tau}(z))$ is arbitrary. Integrating by parts, it follows that
	\begin{equation} \label{eq4}
	0=-\int_{B_{\tau}(z)\cap\partial\{u_0^+>0\}} \left( |\nabla u_0^+|^2 + x_2(1-I_0) \right)\eta d\mathcal{S}.
	\end{equation}
	Here $I_0$ denotes the constant value of $\chi_0$ in the respective connected component of $\{u_0=0\}^{\circ}\cap\{x_2\neq0\}$. Note that by Hopf's lemma, $\nabla u_0^+\cdot\bm n\neq0$ on $B_{\tau}(z)\cap\{u_0^+>0\}$. It follows therefore that $I_0\neq1$, and necessarily $I_0=0$. We deduce from (\ref{eq4}) that
	$$|\nabla u_0^+|^2=-x_2 \quad \text{on} \quad B_{\tau}(z)\cap\partial\{u_0^+>0\}.$$
	Write $u_0^+(\rho,\theta)=C\rho^{3/2}\cos(\frac32\theta+\zeta_0)$ and assume that the cone $\{u_0^+>0\}$ lies between $\Theta$ and $\Theta+\frac{2\pi}{3}$, we obtain
	\begin{equation*}
	\frac94 C^2 = -\sin\theta \quad \text{on} \quad B_{\tau}(z)\cap\partial\{u_0^+>0\}.
	\end{equation*}
	
	Hence $\sin\Theta=\sin(\Theta+\frac{2\pi}{3})$, which means that $\Theta=-\frac{5\pi}{6}$. Then it is straight forward to compute that $C=\frac{\sqrt{2}}{3}$. See Figure \ref{corner1-3}.
	
	\begin{figure}[!h] 
		\includegraphics[width=85mm]{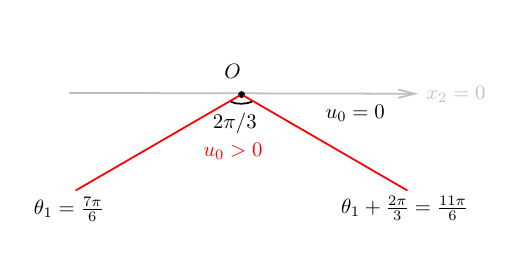}
		\caption{The asymptotic behavior when $\lambda=0$.}
		\label{corner1-3}
	\end{figure}
	
	Consider now the case $u_0^\pm$ both vanish, which means $u_0\equiv0$. It follows from (\ref{eq3}) that
	$$0=-\int_{\mathbb{R}^2}\left( \phi_2 + x_2div\bm\phi \right)\chi_0 dx = \int_{\mathbb{R}^2} x_2\bm\phi\cdot\nabla \chi_0 dx$$
	for $\bm\phi(x)\in W_0^{1,2}(\mathbb{R}^2;\mathbb{R}^2)$, which yields that $\chi_0$ is constant in $\{x_2<0\}$. Its value may be either $0$ in which case $M(0+)=0$, or $1$ in which case $M(0+)=\int_{B_1}x_2^- dx=1$. The disjoint support of non-negative functions $\chi_0^\pm$ gives the remaining part in the theorem for the values of $M_\pm(0+)$.
\end{proof}

It is remarkable that Proposition \ref{density} says if $x^0\in\Gamma_{\rm tp}$ is a two-phase point, then at least one of $M_\pm(0+)$ must vanish. We will use this fact in the following section to prove that $x^0$ cannot be a two-phase free boundary point.

\subsection{Degenerate stagnation points}

In this section we define the degenerate stagnation points for $u^\pm$ as before. Under the strong Bernstein assumption we will prove that the two-phase stagnation point $x^0$ cannot be a degenerate stagnation point for both $u^\pm$, namely, neither of the weighted densities $M_\pm(0+)$ can vanish, which is a contradiction with Proposition \ref{density2}.

We first recall the definition of degenerate stagnation points.

\begin{define}
	Let $u$ be a variational solution of (\ref{eq1.1}) in $D$ with $\lambda=0$ and $x^0\in\Gamma_{\rm tp}$ be a stagnation point. The set $D_\pm^u$ of degenerate stagnation points is defined by
	$$D_\pm^u:= \left\{ x^0\in \Gamma_{\rm tp} \ \Bigg| \ |\nabla u^\pm(x^0)|=0 \ \text{and} \ \frac{u^\pm(x^0+rx)}{r^{3/2}}\rightarrow0 \ \text{strongly in $W_{\rm{loc}}^{1,2}(\mathbb{R}^2)$ as $r\rightarrow0+$} \right\}.$$
	Otherwise we say $x^0$ is non-degenerate for $u^\pm$.
\end{define}

Proposition \ref{prop2} excludes the case that $M_\pm(0+)=0$.

\begin{prop} \label{prop2}
	Let $u$ be a weak solution of (\ref{eq1.1}) in $D$ with $\lambda=0$ satisfying Assumption \ref{assume2} and $x^0\in\Gamma_{\rm tp}$ be a stagnation point. Moreover, suppose the strong Bernstein estimate holds that
	$$|\nabla u|^2\leq x_2^- \quad \text{in} \quad D.$$
	Then $M_\pm(0+)=0$ implies that $u^\pm\equiv0$ in some open ball containing $x^0$. In other words, $M_\pm(0+)\neq0$.
\end{prop}

\begin{proof}
	We state for the case $u_0^+\equiv0$. As in Proposition \ref{prop1}, argue by contradiction that the blow-up sequence
	$$u_k^+(x):=\frac{u^+(x^0+r_kx)}{r_k^{3/2}}$$
	at the stagnation point $x^0\in\Gamma_{\rm tp}$ converges weakly in $W_{\rm{loc}}^{1,2}(\mathbb{R}^2)$ to a blow-up limit $u_0^+\equiv0$. Hence we have
	\begin{equation*}
	0 \leftarrow \Delta u_k^+(B_2)\geq\int_{B_2\cap\partial_{red}\{u_k^+>0\}} \sqrt{-x_2} d\mathcal{S} \quad \text{as} \quad r_k\rightarrow0+.
	\end{equation*}
	On the other hand, there is at least one connected component $V_k$ of $\{u_k>0\}$ touching the origin and containing a point $x^k\in\partial A$ for $A=(-1,1)\times(-1,0)$. Consequently,
	$$\max \{-x_2: x\in V_k\cap\partial A\}\rightarrow0,$$
	and the free boundary condition as well as $|\nabla u|^2\leq x_2^-$ gives
	$$0=\Delta u_k(V_k\cap A)\leq\int_{V_k\cap\partial A}\sqrt{-x_2} d\mathcal{S} + \int_{A\cap\partial_{red}V_k}-\sqrt{-x_2} d\mathcal{S},$$
	a contradiction.
\end{proof}

\begin{rem}
	Proposition \ref{prop2} gives immediately that none of the cases in Proposition \ref{density2} is true. In fact, the proof of Proposition \ref{prop2} implies that if $u_0^+=0$, then $u^+\equiv0$ in a small neighborhood of $x^0$, hence there is no free boundary $\partial\{u>0\}$ in a small neighborhood of $x^0$. Consequently, the foregoing assumption that $x^0\in\Gamma_{\rm tp}$ is unreasonable.
\end{rem}

Now we come to our conclusion for $\lambda=0$. The analysis for such one-phase free boundary point $x^0$ follows from the works \cite{VW11} and \cite{VW12} directly, and we omit the proof here.

\begin{thm}
	Let $u$ be a weak solution of (\ref{eq1.1}) in $D$ with $\lambda=0$ satisfying Assumption \ref{assume2}, and $x^0$ is a stagnation point on the free boundary. Suppose that $|\nabla u^+|^2\leq x_2^-$ in $D$. Then $x^0$ must be a one-phase free boundary point.
	
	To be specific, suppose moreover that $\{u=0\}$ has locally only finite many connected components. Then at each stagnation point $x^0$, we have $M_{x^0,u}(0+) = \int_{B_1} x_2^- \chi_{\{x:-\frac{5\pi}{6}<\theta<-\frac{\pi}{6}\}}$, and
	\begin{equation*}
	\frac{u(x^0+rx)}{r^{3/2}}\rightarrow
	\begin{cases}
	\frac{\sqrt{2}}{3}\rho^{3/2}\cos(\frac32\theta+\frac{3\pi}{4}), \quad -\frac{5\pi}{6}<\theta<-\frac{\pi}{6}, \\
	0, \qquad\qquad\qquad\qquad\ \ \text{otherwise}
	\end{cases}
	\end{equation*}
	or
	\begin{equation*}
	\frac{u(x^0+rx)}{r^{3/2}}\rightarrow
	\begin{cases}
	-\frac{\sqrt{2}}{3}\rho^{3/2}\cos(\frac32\theta+\frac{3\pi}{4}), \quad -\frac{5\pi}{6}<\theta<-\frac{\pi}{6}, \\
	0, \qquad\qquad\qquad\qquad\quad\ \text{otherwise}
	\end{cases}
	\end{equation*}
	strongly in $W_{\rm{loc}}^{1,2}(\mathbb{R}^2)$ and locally uniformly in $\mathbb{R}^2$. Moreover, the free boundary $\partial\{u>0\}$ or $\partial\{u<0\}$ is in a neighborhood of $x^0$ the union of two $C^1$-graphs with right and left tangents at $x^0$. See Figure \ref{corner2-3}. 
\end{thm}

\begin{figure}[!h] 
	\includegraphics[width=150mm]{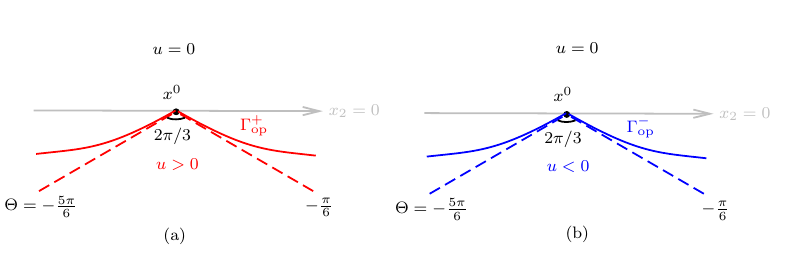}
	\caption{Stokes corner when $\lambda=0$.}
	\label{corner2-3}
\end{figure}

\appendix

\section{Additional free boundary conditions}

In Appendix A, we verify the following additional free boundary conditions for the variational solution $u$ of (\ref{eq1.1}) in $D$.

\begin{prop}
	Suppose $u$ is a variational solution of (\ref{eq1.1}) in $D$ with $\lambda\geq0$, namely,
	$$\int_{D} \left( |\nabla u|^2 div\bm\phi - 2\nabla u D\bm\phi \nabla u -(\phi_2+x_2 div\bm\phi)\chi_{\{u>0\}} + (-\phi_2 + (-x_2+\lambda^2)div\bm\phi)\chi_{\{u<0\}} \right)dx = 0$$
	for any $\bm\phi=(\phi_1,\phi_2)\in W_0^{1,2}(D;\mathbb{R}^2)$ and $\lambda\geq0$. Then, $u$ satisfies the additional free boundary conditions
	$$|\nabla u^+|^2\geq -x_2 \quad \text{on} \quad \partial\{u>0\}\cap D, \quad |\nabla u^-|^2\geq -x_2+\lambda^2 \quad \text{on} \quad \partial\{u<0\}\cap D.$$
\end{prop}

\begin{proof}
	We prove the additional condition on $\partial\{u>0\}$, since the proof is similar on $\partial\{u<0\}$. For the sake of convenience we suppose $\partial\{u>0\}$ has the outer normal vector $\nu(x)$ at the point $x$, otherwise the following process can be obtained by the approximation of $\partial\{u>t\}$ and its outer normal $\nu_t(x)$ as $t\rightarrow0+$.
	
	For any $\bm\phi(x)\in W_0^{1,2}(D;\mathbb{R}^2)$ such that $\bm\phi(x)\cdot\nu(x)\leq0$ at $x\in\partial\{u>0\}$, let $y=x+\epsilon\bm\phi(x)$ and set $u_\epsilon(y) = u^+(x) - u^-(y)$. Then $\{u_\epsilon>0\}\subset\{u>0\}$. Hence
	\begin{equation*}
	\begin{aligned}
	0 &= \lim_{\epsilon\rightarrow0} \frac1\epsilon \Bigg[ \int_{D\cap\{u_\epsilon>0\}} \left( |\nabla u_\epsilon(y)|^2 + (-y_2) \right)dy + \int_{D\cap\{u<0\}} \left( |\nabla u|^2 + (-x_2+\lambda^2) \right)dx \\
	& \quad + \int_{D\cap\{u_\epsilon<0<u\}} \left( |\nabla u_\epsilon(y)|^2 + (-y_2+\lambda^2) \right)dy - \mathcal{J}(u) \Bigg] \\
	&\leq \int_{D\cap\{u>0\}} \left( |\nabla u|^2 div\bm\phi - 2\nabla u D\bm\phi \nabla u + div(-x_2\bm\phi) \right)dx \\
	&= \int_{D\cap\partial\{u>0\}} \left( -|\nabla u|^2 -x_2 \right)\bm\phi\cdot\nu d\mathcal{S},
	\end{aligned}
	\end{equation*}
	which leads to $|\nabla u^+|^2\geq -x_2$ on $\partial\{u>0\}$ in the weak sense since $\bm\phi\cdot\nu\leq0$.
\end{proof}

\section{Definitions of variational solution and weak solution}

In Appendix B, we give the definitions of variational solution and weak solution of (\ref{eq1.1}) for $\lambda\geq0$. We suppose that $D$ is a  domain in $\mathbb{R}^2$.

\begin{define} \label{def1} (Variational solution for $\lambda>0$)
	We define $u\in W^{1,2}(D)$ to be a variational solution of (\ref{eq1.1}), if $u\in C^0(D)\cap C^2(D\cap\{u\neq0\})$, $u\leq0$ in $B_R(x^0)\cap\{x_2\geq0\}$ where $x^0\in\Gamma_{\rm tp}$ is a stagnation point and $B_R(x^0)\Subset D$ is a small neighborhood of $x^0$, and the first variational with respect to domain variations of the functional
	$$\mathcal{J}(v)=\int_{D} \left( |\nabla v|^2 + (-x_2)\chi_{\{v>0\}} + (-x_2+\lambda^2)\chi_{\{u<0\}} \right)dx$$
	vanishes at $v=u$, i.e.
	\begin{equation*}
	\begin{aligned}
	0&=\frac{d}{d\epsilon}\Big|_{\epsilon=0} \mathcal{J}(u(x+\epsilon\bm\phi(x))) \\
	&=\int_{D} \left( |\nabla u|^2 div\bm\phi - 2\nabla u D\bm\phi \nabla u -(\phi_2+x_2 div\bm\phi)\chi_{\{u>0\}} + (-\phi_2 + (-x_2+\lambda^2)div\bm\phi)\chi_{\{u<0\}} \right)dx
	\end{aligned}
	\end{equation*}
	for any $\bm\phi=(\phi_1,\phi_2)\in W_0^{1,2}(D;\mathbb{R}^2)$.
\end{define}

\begin{define} \label{def2} (Variational solution for $\lambda=0$)
	We define $u\in W_{\rm{loc}}^{1,2}(D)$ to be a variational solution of (\ref{eq1.1}), if $u\in C^0(D)\cap C^2(D\cap\{u\neq0\})$, $u\equiv0$ in $B_R(x^0)\cap\{x_2\geq0\}$ where $x^0\in\Gamma_{\rm tp}$ is a stagnation point and $B_R(x^0)\subset D$ is a small neighborhood of $x^0$, and the first variational with respect to domain variations of the functional
	$$\mathcal{J}(v)=\int_{D} \left( |\nabla v|^2 + (-x_2)\chi_{\{v\neq0\}} \right)dx$$
	vanishes at $v=u$, i.e.
	\begin{equation*}
	\begin{aligned}
	0&=\frac{d}{d\epsilon}\Big|_{\epsilon=0} \mathcal{J}(u(x+\epsilon\bm\phi(x))) \\
	&=\int_{D} \left( |\nabla u|^2 div\bm\phi - 2\nabla u D\bm\phi \nabla u -(\phi_2+x_2 div\bm\phi)\chi_{\{u\neq0\}} \right)dx
	\end{aligned}
	\end{equation*}
	for any $\bm\phi=(\phi_1,\phi_2)\in W_0^{1,2}(D;\mathbb{R}^2)$.
\end{define}

We also use weak solutions of (\ref{eq1.1}) to help the analysis of stagnation points.

\begin{define}(Weak solution for $\lambda>0$) \label{wsol}
	We define $u\in W_{\rm{loc}}^{1,2}(D)$ to be a weak solution of (\ref{eq1.1}), if $u$ is a variational solution of (\ref{eq1.1}) when $\lambda^2>0$ and the free boundaries $\partial\{u>0\}\cap D\cap\{x_2<0\}$ and $\partial\{u<0\}\cap D$ are locally $C^{2,\alpha}$ surfaces.
\end{define}

\begin{define}(Weak solution for $\lambda=0$) \label{wsol2}
	We define $u\in W_{\rm{loc}}^{1,2}(D)$ to be a weak solution of (\ref{eq1.1}), if $u$ is a variational solution of (\ref{eq1.1}) when $\lambda^2=0$ and the free boundaries $\partial\{u>0\}\cap D\cap\{x_2<0\}$ and $\partial\{u<0\}\cap D\cap\{x_2<0\}$ are locally $C^{2,\alpha}$ surfaces.
\end{define}

\begin{rem}
	The assumption that $\partial\{u>0\}\cap D\cap\{x_2<0\}$ and $\partial\{u<0\}\cap D$ are locally a $C^{2,\alpha}$ surfaces can be verified from the free boundary regularity theory directly, for minimizers $u$ of $\mathcal{J}(u;D)$ as an example.
\end{rem}

\section{Proof of the concentration compactness}

For the completeness of the exposition, we give the proof of the concentration compactness.

\begin{proof}[Proof of Lemma \ref{cc}]
	We first prove the consequence (i).
	
	Notice that the homogeneity of $v_0$ given by (\ref{6-1}), together with the fact that $v_0$ belongs to $W^{1,2}(B_1)$, imply that $v_0$ is continuous. As
	\begin{equation} \label{a-1}
	\Delta v_k^+ = \frac{r_k^2 \Delta u^+(x^0+r_kx)}{\sqrt{ r_k^{-1}\int_{\partial B_{r_k}}(u^+)^2 d\mathcal{S} }}=0 \quad \text{for} \quad v_k^+(x)>0,
	\end{equation}
	we obtain from the sign of the singular part of $\Delta v_k^+$ with respect to Lebesgue measure that $\Delta v_k^+ \geq 0$ in $B_1$ in sense of measures. It follows that for each non-negative $\eta\in C_0^{\infty}(B_1)$ such that $\eta=1$ in $B_{(\sigma+1)/2}$ for $\sigma\in(0,1)$,
	\begin{equation} \label{a-2}
	\int_{B_{(\sigma+1)/2}} d \Delta v_k^+ =\int_{B_{(\sigma+1)/2}} \eta d\Delta v_k^+ \leq \int_{B_1} \eta d\Delta v_k^+ = \int_{B_1} v_k^+ \Delta\eta dx \leq C\int_{B_1} v_k^+ dx \leq C
	\end{equation}
	for all $k$. From (\ref{a-1}) and the fact that $v_k^+$ is bounded in $L^1(B_1)$, we obtain also that $\Delta v_0$ is a non-negative Radon measure on $B_1$. The continuity of $v_0$ implies therefore that $v_0\Delta v_0$ is well defined as a non-negative Radon measure on $B_1$.
	
	In order to apply the concentrated compactness \cite{EM94}, we modify each $v_k^+$ to
	$$\tilde{v}_k^+:=v_k^+ * \phi_k\in C^{\infty}(B_1),$$
	where $\phi_k$ is a standard mollifier such that
	$$\Delta \tilde{v}_k^+ \geq 0, \qquad \int_{B_{\sigma}}d\Delta\tilde{v}_k^+\leq C<\infty \quad \text{for all} \ \ k,$$
	and
	$$\Vert v_k^+ - \tilde{v}_k^+ \Vert_{W^{1,2}(B_{\sigma})}\rightarrow0 \quad \text{as} \quad k\rightarrow\infty.$$
	By Chapter 4, Theorem 3 in \cite{E90} we know that $\nabla\tilde{v}_k^+$ converges a.e. to the weak limit $\nabla v_0$, and the only possible problem is concentration of $|\nabla\tilde{v}_k^+|^2$. By Theorem 1.1 and Theorem 3.1 in \cite{EM94}, we obtain that
	\begin{equation} \label{a-3}
	\partial_1\tilde{v}_k^+ \partial_2\tilde{v}_k^+ \rightarrow \partial_1 v_0 \partial_2 v_0
	\end{equation}
	and
	$$(\partial_1\tilde{v}_k^+)^2 - (\partial_2\tilde{v}_k^+)^2 \rightarrow (\partial_1 v_0)^2 - (\partial_2 v_0)^2$$
	in the sense of distributions on $B_{\sigma}$ as $k\rightarrow\infty$. Let us remark that this alone would allow us to pass to the limit in the domain variation formula for $v_k^+$ in the set $\{x_2<0\}$.
	
	Observe now that (\ref{6-1}) shows that
	$$\nabla v_k^+(x)\cdot x - H_+(0+)v_k^+(x) \rightarrow 0$$
	strongly in $L^2(B_{\sigma}\backslash B_{\rho})$ as $k\rightarrow\infty$. It follows that
	$$\partial_1 v_k^+ x_1 + \partial_2 v_k^+ x_2 \rightarrow \partial_1 v_0 x_1 + \partial_2 v_0 x_2$$
	strongly in $L^2(B_{\sigma}\backslash B_{\rho})$ as $k\rightarrow\infty$. But then
	$$\int_{B_\sigma\backslash B_\rho} (\partial_1 v_k^+ \partial_1 v_k^+ x_1 + \partial_1 v_k^+ \partial_2 v_k^+ x_2)\eta dx \rightarrow \int_{B_\sigma\backslash B_\rho} (\partial_1 v_0 \partial_1 v_0 x_1 + \partial_1 v_0 \partial_2 v_0 x_2)\eta dx$$
	for each $\eta\in C_0^0(B_{\sigma}\backslash \bar{B}_{\rho})$ as $m\rightarrow\infty$. Using (\ref{a-3}) we obtain that
	$$\int_{B_\sigma\backslash B_\rho} (\partial_1 v_k^+)^2 x_1\eta dx \rightarrow \int_{B_\sigma\backslash B_\rho} (\partial_1 v_0)^2 x_1\eta dx$$
	for each $0\leq\eta\in C_0^0((B_{\sigma}\backslash \bar{B}_{\rho})\cap\{x_1>0\})$ and each $0\geq\eta\in C_0^0((B_{\sigma}\backslash \bar{B}_{\rho})\cap\{x_1<0\})$ as $k\rightarrow\infty$. Repeating the above procedure three times for rotated sequences of solutions (by $45^\circ$) yields that $\nabla v_k^+$ converges strongly in $L_{\rm{loc}}^2(B_{\sigma}\backslash \bar{B}_{\rho})$. Since $\sigma$ and $\rho$ with $0<\sigma<\rho<1$ are arbitrary, it follows that $\nabla v_k^+$ converges to $\nabla v_0$ strongly in $L_{\rm{loc}}^2(B_1\backslash\{0\})$.
	
	As a consequence of the strong convergence, we see that
	$$\int_{B_1} \nabla(\eta v_0)\cdot\nabla v_0 dx = 0 \quad \text{for all} \quad \eta\in C_0^1(B_1\backslash\{0\}).$$
	Combined with the fact that $v_0=0$ in $B_1\cap\{x_2\geq0\}$, this proves that $v_0\Delta v_0=0$ in the sense of Radon measures on $B_1$.
	
	Next we prove conclusion (ii).
	
	Let $r_k\rightarrow0+$ be an arbitrary sequence such that the sequence $v_k^+$ given by (\ref{vk}) converges weakly in $W^{1,2}(B_1)$ to a limit $v_0$. By Lemma \ref{cc} (i), $v_0\neq0$, $v_0$ is homogeneous of degree $H_{+,x^0,u}(0+)\geq\frac32$, $v_0$ is continuous, $v_0\geq0$ and $v_0\equiv0$ in $\{x_2\geq0\}$, $v_0\Delta v_0=0$ in $B_1$ as a Radon measure, and the convergence of $v_k^+$ to $v_0$ is strong in $W_{\rm{loc}}^{1,2}(B_1\backslash\{0\})$. Moreover, remember that we have $u_0^-=\lambda^2 x_2^-$, the strong convergence of $v_k^+$ and the fact that $V_+(r_k)\rightarrow0$ as $k\rightarrow\infty$ imply that
	$$0 = \int_{B_1\cap\{x_2\leq-\delta\}} (|\nabla v_0|^2 div\bm\phi - 2\nabla v_0 D\bm\phi \nabla v_0) dx$$
	for every $\bm\phi\in W_0^{1,2}(B_1\cap\{x_2\leq-\delta\};\mathbb{R}^2)$ and $0<\delta<1$ small. It follows that at each point $(\cos\theta,\sin\theta)\in\partial B_1\cap\partial\{v_0>0\}$,
	$$\lim_{\tau\rightarrow\theta+} \partial_{\theta} v_0(1,\tau) = \lim_{\tau\rightarrow\theta-} \partial_{\theta} v_0(1,\tau) \quad \text{in polar coordinates.}$$
	Computing the solution of ODE on $\partial B_1$, using the homogeneity of degree $H_{+,x^0,u}(0+)$ of $v_0$ and the fact that $\int_{\partial B_1} v_0^2 d\mathcal{S} = 1$, yields that $H_{+,x^0,u}(0+)$ must be an integer $N(x^0)\geq2$ and that
	$$v_0 = \frac{\rho^{N(x^0)}|\sin(N(x^0)\min(\max(\theta,0),\pi))|}{\sqrt{\int_0^\pi \sin^2(N(x^0)\theta)d\theta }}.$$
	The desired conclusion follows from Lemma \ref{cc} (i).
\end{proof}

\subsubsection*{Acknowledgment}
This work is supported by National Nature Science Foundation of China under Grant 12125102, and Nature Science Foundation of Guangdong Province under Grant 2024A1515012794.
\subsubsection*{Declaration of competing interest} The authors declare they have no conflicts of interests. 
\subsubsection*{Data availability} No data was used for the research described in the article.

\vskip .4in

\bibliographystyle{plain}
\bibliography{refs}

\end{document}